\documentclass[12pt]{article}
\usepackage{amsmath,amsfonts,amssymb,amsthm,mathrsfs}
\usepackage[a4paper,vmargin={3.5cm,3.5cm},hmargin={2.5cm,2.5cm}]{geometry}
\usepackage[margin=1cm]{caption}
\usepackage{graphicx,graphics}
\usepackage{epsfig}
\usepackage{latexsym}
\usepackage[utf8]{inputenc}
\usepackage{ae,aecompl}
\usepackage[english]{babel}
\usepackage[colorlinks=true]{hyperref}
\usepackage{enumerate}
\usepackage{bbm}
\usepackage{pxfonts}
\usepackage{dsfont}
\usepackage{tikz}
\usepackage{tabularx}
\usepackage{multirow}

\usepackage[auto]{contour}
\usetikzlibrary{shapes, angles, decorations.text, patterns, fadings,decorations.pathreplacing,fit}
\pgfdeclarelayer{edgelayer}
\pgfdeclarelayer{nodelayer}
\pgfsetlayers{edgelayer,nodelayer,main}

\date{}

\newtheorem{theorem}{Theorem}[]
\newtheorem{proposition}{Proposition}[section]
\newtheorem{lemma}[proposition]{Lemma}
\newtheorem{cor}[proposition]{Corollary}
\theoremstyle{definition}
\newtheorem{definition}{Definition}[section]
\newtheorem{remark}{Remark}[section]
\newtheorem{construction}{Construction}[]

\newcommand{\M}{\mathsf{M}}
\newcommand{\T}{\mathsf{T}} 
\newcommand{\CT}{\mathsf{T}} 
\newcommand{\Mob}{\mathsf{Mob}}
\newcommand{\LMob}{\mathsf{LMob}}
\newcommand{\STr}{\mathsf{STr}}
\newcommand{\BT}{\mathsf{BT}}

\renewcommand{\epsilon}{\varepsilon}

\newcommand{\grow}{\operatorname{grow}}
\newcommand{\coll}{\operatorname{coll}}

\DeclareSymbolFont{extraup}{U}{zavm}{m}{n}
\DeclareMathSymbol{\varheart}{\mathalpha}{extraup}{86}
\DeclareMathSymbol{\vardiamond}{\mathalpha}{extraup}{87}

\makeatletter
\renewcommand*{\@fnsymbol}[1]{\ensuremath{\ifcase#1\or  \vardiamond \or \clubsuit\or \spadesuit\or
   \mathsection\or \mathparagraph\or \|\or **\or \dagger\dagger
   \or \ddagger\ddagger \else\@ctrerr\fi}}
\makeatother

\title{\bf \textsc{Growing uniform planar maps face by face}}
\author{Alessandra Caraceni\thanks{INdAM research unit at Scuola Normale Superiore, Pisa, Italy.\hfill  \texttt{alessandra.caraceni@sns.it}} \hspace{12pt} \& \hspace{4pt} Alexandre Stauffer\thanks{Department of Mathematics and Physics, Univ.~Roma Tre, Rome, Italy\newline \textcolor{white}{-----} Department of Mathematical Sciences, University of Bath, UK.\hfill  \texttt{astauffer@mat.uniroma3.it}\newline \textcolor{white}{-----} Supported by EPSRC Fellowship EP/N004566/1.}}
\begin{document}
\maketitle

\abstract{We provide ``growth schemes'' for inductively generating uniform random $2p$-angulations of the sphere with $n$ faces, as well as uniform random simple triangulations of the sphere with $2n$ faces. In the case of $2p$-angulations, we provide a way to insert a new face at a random location in a uniform $2p$-angulation with $n$ faces in such a way that the new map is precisely a uniform $2p$-angulation with $n+1$ faces. Similarly, given a uniform simple triangulation of the sphere with $2n$ faces, we describe a way to insert two new adjacent triangles so as to obtain a uniform simple triangulation of the sphere with $2n+2$ faces. The latter is based on a new bijective presentation of simple triangulations that relies on a construction by Poulalhon and Schaeffer.}

\section{Introduction}

The study of uniform random planar maps of given size conditioned to satisfy certain constraints is by now a sprawling subject rich with connections both internal and to other parts of mathematics and physics. Its combinatorial roots are firmly embedded in the works of Tutte, who in the 60s lay the groundwork for an enumerative theory of planar maps \cite{Tut62, Tut63}. Later works, whose original inspiration can in many cases be traced back to the celebrated bijection by Cori, Vanquelin and Schaeffer \cite{CV81, Sch98}, have uncovered crucial links between certain classes of planar maps and classes of trees or tree-like structures. Armed with these bijections, probabilists have been able to embark on a deep investigation of the metric structure of large random planar maps and achieve very general results concerning scaling limits~\cite{LGM12}, local limits~\cite{AS03} as well as many other aspects.

The topic of \emph{random generation} of planar maps is one that naturally inserts itself near the core of this subject, and one that we wish to further develop with this paper. More specifically, we are interested in generating a uniform map of size $n+1$ within a certain class by somehow \emph{growing} it out of a uniform map of size $n$ via a small local modification. Part of our motivation for considering this problem comes from investigating the mixing time of certain edge flip Markov chains on various classes of planar maps. A strategy for designing canonical paths using a uniform ``growth'' scheme has led us in previous papers \cite{CS19,Car20} to consider this problem for the (rather special) case of quadrangulations; in a future paper, we shall leverage the results obtained here to achieve polynomial upper bounds for the mixing time of the edge flip Markov chain on $2p$-angulations and simple triangulations of the sphere.

Naturally, the problem of random generation within classes of planar maps has been considered before: many of the papers establishing links to branching structures do so with an eye towards efficient random generation (see for example~\cite{PS06}). However, few results have been obtained so far concerning iterative, increasing procedures, where a uniform map of size $n+1$ is coupled with a uniform map of size $n$ in such a way that the larger map is always obtained from the smaller one by a small local (and random) modification that increases the number of faces. Other than the aforementioned paper \cite{Car20}, a result in this direction is that of Bettinelli \cite{Bet14}, who uses an explicit bijection to obtain a uniform quadrangulation with $n+1$ faces by performing a certain (random) surgery operation on a uniform quadrangulation with $n$ faces, cutting it along a path, inserting a new face and re-glueing the edges of the cut in a slightly different way. Note that, though this operation may indeed count as a small random modification, it does affect a potentially large region of the quadrangulation, which makes it ill-suited to some applications (such as the one to edge flip chains mentioned above). Moreover, efforts of this kind have born fruits -- to our knowledge -- mainly in regards to quadrangulations, whose direct link to plane trees provides access to a more complete set of tools.

If one looks at random generation of uniform plane trees of fixed size, a rich panorama unfolds. From R\'emy's algorithm for the generation of binary trees~\cite{Rem85}, to bijective approaches involving letter sequences and the cycle lemma~\cite{ARS97}, to the more probabilistic approach involving critical Galton--Watson trees, this topic continues to be investigated in its many facets, with very recent papers such as \cite{Mar21} providing further insights. Within this larger scope, one can consider the problem of generating uniform random trees in an iterative, increasing way, by grafting leaves onto smaller uniform trees; this is the approach of Lukzac and Winkler in the paper \cite{LW04}, which is one of the main sources of inspiration for this work. Note that the question of whether certain random trees can be coupled in an increasing way has been asked in various contexts and sometimes answered in the negative, see for example \cite{HS19}.

Our main objective within this paper is to bring increasing random generation in the style of~\cite{LW04} to more complicated classes of maps, and in particular to rooted $2p$-angulations of the sphere and rooted simple triangulations of the sphere. To this end, we better formalise the notion of a growth scheme (see Definitions~\ref{def:growth scheme for 2p-angulations} and~\ref{def:growth scheme for triangulations}), and we consolidate the existing setup of bijections between $2p$-angulations and labelled mobiles (see~\cite{BDFG04} and Section~\ref{2p-angulations, mobiles, trees}), as well as simple triangulations and blossoming trees (see~\cite{PS06} and Section~\ref{triangulations, blossoming trees}), in order to obtain a direct correspondence between the notion of growing uniform $2p$-angulations (respectively, simple triangulations), and that of growing labelled mobiles (respectively, blossoming trees). Finally, we explore the links between mobiles, blossoming trees and $d$-ary trees, relying on the main result of~\cite{LW04} in order to show

\begin{theorem}\label{2p-angulations growth}
For all $p\geq 2$, there exists a growth scheme for rooted $2p$-angulations of the sphere.	
\end{theorem}

as well as

\begin{theorem}\label{triangulation growth}
There exists a growth scheme for rooted simple triangulations of the sphere.	
\end{theorem}

Along the way to reaching the above results, we describe a construction proving the following proposition, highlighting new connections between combinatorial structures which might be of independent interest:
 
\begin{proposition}\label{4-ary trees and triangulations}
		There is an explicit $2n$-to-$1$ map (induced by Construction~\ref{BT to (l,r)}) from $\mathcal{T}_{n-1}\times\{\pm 1\}$ to $\STr_n$, where $\mathcal{T}_{n-1}$ is the set of rooted plane trees with $n-1$ vertices of degree $5$ and $3n-1$ leaves, and $\STr_n$ is the set of rooted simple triangulations of the sphere with $2n$ faces.
\end{proposition}

Though Theorems~\ref{2p-angulations growth} and \ref{triangulation growth} were obtained with a view of their applications to Markov chain mixing, we believe they might be a useful tool to add to the arsenal available in the investigation of planar maps. Moreover, they open up various avenues of further study, as one might want to clarify their relation to local limits and peeling processes, or attempt to shed more light on the explicit underlying probabilities that are only characterised recursively. Growth schemes of this type could also be useful in developing graphical tools for the visualisation of large uniform triangulations or $2p$-angulations of the sphere, as they give a way to iteratively construct such maps ``from the ground up'', with each step being a small modification that does not strongly disrupt the overall structure.

The paper is organised as follows. Sections~\ref{2p-angulations, mobiles, trees} and \ref{triangulations, blossoming trees} contain the relevant notation, with compact reminders of  known bijections and their most relevant features in our context. Section~\ref{Growing $2p$-angulations} explains how growing $d$-ary trees is equivalent to growing $d$-mobiles, which in turn lays the groundwork for a growth scheme for $2d$-angulations: the section culminates with the proof of Proposition~\ref{2p-angulation growth prop}, which yields Theorem~\ref{2p-angulations growth} as a consequence. Finally, Section~\ref{Growing simple triangulations} proves Proposition~\ref{4-ary trees and triangulations} and Theorem~\ref{triangulation growth}.

\section{$2p$-angulations, mobiles and labelled $p$-ary trees}\label{2p-angulations, mobiles, trees}

In this section, we briefly define some of the combinatorial objects of interest and remind the reader of some relationships between them. The bijections we will be discussing are all known or easy to deduce from known facts, but we shall attempt a compact presentation of those parts of the constructions which will actually be involved in subsequent proofs. We refer readers to \cite{BDFG04} for additional details.

\begin{definition}Given an integer $p\geq 2$, a $p$-angulation is a rooted planar map all of whose faces have a contour of length $p$. We denote the set of all $p$-angulations with $n$ faces by $\M^p_n$.	We shall sometimes also consider the set of \emph{pointed} $p$-angulations $\M^{p,\bullet}_n$, that is, the set of pairs $(m,v)$, where $m\in \M^p_n$ and $v\in V(m)$ is a vertex of $m$.
\end{definition}

Bijective approaches are especially successful in the case of $p$-angulations with $p$ even, and more generally of classes of bipartite planar maps, which lend themselves well to a certain type of construction involving mobiles, introduced by Bouttier, Di Francesco and Guitter in \cite{BDFG04}. Though bijections with other classes of decorated trees do exist in the case of odd $p$, they involve additional conditions on the $p$-angulations and do not in general yield enumerative results of the same explicit nature as those one obtains in the case of even $p$. We will discuss the special case of (simple) triangulations in the next section; this section will be mostly devoted to $2p$-angulations and to related branching structures.

In connection to $2p$-angulations, we will consider various types of plane trees.
\begin{definition}\label{def:complete d-ary tree}
Let $d\geq 2$ and $n\geq 1$ be integers. A \emph{(complete) $d$-ary tree} is a rooted plane tree each of whose vertices is of one of three types: the \emph{origin}, which carries the root corner and has degree $d$; the \emph{leaves}, which have degree $1$; the \emph{internal vertices}, which have degree $d+1$. We say that a complete $d$-ary tree $\tau$ has \emph{size} $|\tau|=n$, where $n\geq 1$, if it has $n-1$ internal vertices.

We denote by $\CT^d_n$ the set of all complete $d$-ary trees of size $n$ (see Figure~\ref{fig:3-ary trees}), and conventionally define $\CT^d_0$ as the set containing the tree with only one vertex, which we will sometimes call the \emph{degenerate} $d$-ary tree, in spite of it not satisfying the definition above.
\end{definition}

Remark that a tree in $\CT^d_n$ has a total of $(d-1)n+1$ leaves, and therefore $dn+1$ vertices and $dn$ edges. This is easy to see using the following construction, which we will repeatedly discuss throughout the paper: given a complete $d$-ary tree $\tau$ and a leaf $v$ of $\tau$, we denote by $\grow(\tau,v)$ the complete $d$-ary tree of size $|\tau|+1$ obtained by turning $v$ into an internal vertex, grafting $d$ leaves onto it (right part of Figure~\ref{fig:3-ary trees}); we say that $\grow(\tau,v)$ is obtained by \emph{growing $\tau$ at $v$}. The fact that every tree in $\CT^d_{n+1}$ can be obtained from some tree in $\CT^d_{n}$ by growing it at some leaf immediately yields the number of leaves and vertices given above, by induction.

\begin{remark}\label{rem:completeness of trees}
It is worth remarking that a complete $d$-ary tree of size $n>0$, as presented by Definition~\ref{def:complete d-ary tree}	, can easily be reinterpreted as a pointed tree with $n$ vertices, all labelled save for the origin: each vertex $v$ has at most $d$ children, each of which is assigned a different label in $\{1,2,\ldots,d\}$. Indeed, this interpretation of a (complete) $d$-ary tree comes from numbering sibling vertices in left-to-right order and then erasing all the leaves (Figure~\ref{fig:3-ary trees}). In this alternative interpretation, the operation of growing $\tau\in\CT^d_n$ at a leaf $v$ corresponds to making $v$ into an internal vertex, thus adding it and its label to our labelled tree.

It can be useful to have this alternative view in mind, particularly when referring to the paper~\cite{LW04}, which predominantly uses it. The interpretation of $\CT_0^d$ is a formal matter, but in the context of labelled trees we would interpret the degenerate complete $d$-ary tree as the empty tree with no vertices and no edges, which can only be grown into the tree with one (distinguished) vertex and no edges, corresponding to the single element of $\CT^d_1$.  
\end{remark}

\begin{figure}\centering
\begin{tikzpicture}[vertex/.style={circle, inner sep=1.6pt, draw=white, fill=black}, gem/.style={diamond, inner sep=1.7pt, fill=green!70!black}, branch/.style={very thick, brown},blossom/.style={green!70!black}, simple/.style={}, scale=.68]
	\begin{pgfonlayer}{nodelayer}
		\node [style=vertex] (0) at (-6, 4) {};
		\node [style=vertex] (1) at (-4, 4) {};
		\node [style=vertex] (2) at (-4.5, 4.5) {};
		\node [style=vertex] (3) at (-4, 4.5) {};
		\node [style=vertex] (4) at (-3.5, 4.5) {};
		\node [style=vertex] (5) at (-1, 4) {};
		\node [style=vertex] (6) at (-1.5, 4.5) {};
		\node [style=vertex] (7) at (-1, 4.5) {};
		\node [style=vertex] (8) at (-0.5, 4.5) {};
		\node [style=vertex] (9) at (-2, 5) {};
		\node [style=vertex] (10) at (-1.5, 5) {};
		\node [style=vertex] (11) at (-1, 5) {};
		\node [style=vertex] (12) at (0.5, 4) {};
		\node [style=vertex] (13) at (0, 4.5) {};
		\node [style=vertex] (14) at (0.5, 4.5) {};
		\node [style=vertex] (15) at (1, 4.5) {};
		\node [style=vertex] (16) at (0, 5) {};
		\node [style=vertex] (17) at (0.5, 5) {};
		\node [style=vertex] (18) at (1, 5) {};
		\node [style=vertex] (19) at (2, 4) {};
		\node [style=vertex] (20) at (1.5, 4.5) {};
		\node [style=vertex] (21) at (2, 4.5) {};
		\node [style=vertex] (22) at (2.5, 4.5) {};
		\node [style=vertex] (23) at (2, 5) {};
		\node [style=vertex] (24) at (2.5, 5) {};
		\node [style=vertex] (25) at (3, 5) {};
		\node [style=vertex] (26) at (-4, 1.5) {};
		\node [style=vertex] (27) at (-1, 1.5) {};
		\node [style=vertex] (28) at (0.5, 1.5) {};
		\node [style=vertex] (29) at (2, 1.5) {};
		\node [circle, fill=white, draw=black, inner sep=0.5pt] (30) at (-1, 2.2) {\footnotesize$1$};
		\node [circle, fill=white, draw=black, inner sep=0.5pt] (31) at (0.5, 2.2) {\footnotesize$2$};
		\node [circle, fill=white, draw=black, inner sep=0.5pt] (32) at (2, 2.2) {\footnotesize$3$};
		\node (33) at (-6, 3) {$\CT_0^3$};
		\node (34) at (-4, 3) {$\CT_1^3$};
		\node (35) at (-6, 1.5) {$\emptyset$};
		\node (36) at (0.5, 3) {$\CT_2^3$};
	\end{pgfonlayer}
	\begin{pgfonlayer}{edgelayer}
		\draw [style=simple] (2) to (1);
		\draw [style=simple] (3) to (1);
		\draw [style=simple] (4) to (1);
		\draw [style=simple] (9) to (6);
		\draw [style=simple] (10) to (6);
		\draw [style=simple] (11) to (6);
		\draw [style=simple] (6) to (5);
		\draw [style=simple] (7) to (5);
		\draw [style=simple] (8) to (5);
		\draw [style=simple] (13) to (12);
		\draw [style=simple] (14) to (12);
		\draw [style=simple] (15) to (12);
		\draw [style=simple] (16) to (14);
		\draw [style=simple] (17) to (14);
		\draw [style=simple] (18) to (14);
		\draw [style=simple] (20) to (19);
		\draw [style=simple] (21) to (19);
		\draw [style=simple] (22) to (19);
		\draw [style=simple] (23) to (22);
		\draw [style=simple] (24) to (22);
		\draw [style=simple] (25) to (22);
		\draw [style=simple] (30) to (27);
		\draw [style=simple] (31) to (28);
		\draw [style=simple] (32) to (29);
	\end{pgfonlayer}
	\clip[use as bounding box] (-8,-0.5) rectangle (4.5,6);
\end{tikzpicture}\qquad
\begin{tikzpicture}[vertex/.style={circle, inner sep=1.6pt, draw=white, fill=black}, gem/.style={diamond, inner sep=1.7pt, fill=green!70!black}, branch/.style={very thick, brown},blossom/.style={green!70!black}, simple/.style={thick}, scale=1.1]
	\begin{pgfonlayer}{nodelayer}
		\node (0) at (5.75,3.25) {$\tau$};
		\node [style=vertex] (1) at (6.75, 4.75) {};
		\node [circle, fill=white, draw=black, inner sep=0.5pt] (2) at (6.75, 2) {\footnotesize3};
		\node [style=vertex] (3) at (8.25, 4.25) {};
		\node [circle, fill=white, draw=black, inner sep=0.5pt] (4) at (9.75, 2) {\footnotesize3};
		\node [style=vertex] (5) at (10, 5.25) {};
		\node [circle, fill=white, draw=black, inner sep=0.5pt] (6) at (9.25, 1.5) {\footnotesize3};
		\node [style=vertex] (7) at (6, 5.25) {};
		\node [style=vertex] (8) at (6.25, 4.25) {};
		\node [style=vertex, red, label=above:$v$] (9) at (5.75, 5.25) {};
		\node [style=vertex] (10) at (5.75, 1) {};
		\node [style=vertex, red] (11) at (8.5, 5.75) {};
		\node [style=vertex] (12) at (10, 5.75) {};
		\node [style=vertex] (13) at (7, 5.75) {};
		\node [style=vertex] (14) at (7, 5.25) {};
		\node [style=vertex, red] (15) at (8.75, 5.25) {};
		\node [style=vertex] (16) at (9, 5.25) {};
		\node [style=vertex] (17) at (6.75, 5.75) {};
		\node [circle, fill=white, draw=black, inner sep=0.5pt] (18) at (7, 2.5) {\footnotesize3};
		\node [circle, fill=white, draw=black, inner sep=0.5pt] (19) at (5.75, 2) {\footnotesize1};
		\node [circle, fill=white, draw=black, inner sep=0.5pt] (20) at (6.25, 1.5) {\footnotesize3};
		\node [style=vertex] (21) at (10.25, 5.75) {};
		\node [style=vertex] (22) at (8.75, 1) {};
		\node (23) at (9, 3.25) {$\grow(\tau,v)$};
		\node [style=vertex] (24) at (6.25, 4.75) {};
		\node [style=vertex] (25) at (6.75, 5.25) {};
		\node [style=vertex] (26) at (9.75, 4.75) {};
		\node [circle, fill=white, draw=black, inner sep=0.5pt] (27) at (10, 2.5) {\footnotesize3};
		\node [style=vertex] (28) at (6.5, 5.25) {};
		\node [style=vertex] (29) at (8.75, 4.75) {};
		\node [style=vertex] (30) at (5.75, 4.75) {};
		\node [style=vertex] (31) at (8.5, 5.25) {};
		\node [circle, fill=red!10, draw=red, thick, inner sep=0.5pt] (32) at (8.75, 2.5) {\footnotesize2};
		\node [style=vertex] (33) at (5.75, 4.25) {};
		\node [style=vertex] (34) at (5.75, 3.75) {};
		\node [style=vertex, red] (35) at (8.75, 5.75) {};
		\node [style=vertex, red] (36) at (9, 5.75) {};
		\node [style=vertex] (37) at (9.25, 4.25) {};
		\node [style=vertex] (38) at (8.75, 4.25) {};
		\node [style=vertex] (39) at (9.75, 5.25) {};
		\node [style=vertex] (40) at (7.25, 5.75) {};
		\node [style=vertex] (41) at (9.25, 4.75) {};
		\node [style=vertex] (42) at (5.25, 4.25) {};
		\node [style=vertex] (43) at (9.5, 5.25) {};
		\node [style=vertex] (44) at (5.5, 5.25) {};
		\node [style=vertex] (45) at (9.75, 5.75) {};
		\node [style=vertex] (46) at (8.75, 3.75) {};
		\node [circle, fill=white, draw=black, inner sep=0.5pt] (47) at (8.75, 2) {\footnotesize1};
	\end{pgfonlayer}
	\begin{pgfonlayer}{edgelayer}
		\draw [style=simple] (27) to (4);
		\draw [style=simple] (29) to (37);
		\draw [style=simple] (43) to (26);
		\draw [style=simple] (45) to (5);
		\draw [style=simple] (4) to (6);
		\draw [style=simple] (31) to (29);
		\draw [style=simple] (44) to (30);
		\draw [style=simple, red] (11) to (15);
		\draw [style=simple] (13) to (14);
		\draw [style=simple] (41) to (37);
		\draw [style=simple] (8) to (34);
		\draw [style=simple, red] (32) to (47);
		\draw [style=simple] (30) to (8);
		\draw [style=simple] (5) to (26);
		\draw [style=simple, red] (35) to (15);
		\draw [style=simple] (38) to (46);
		\draw [style=simple] (28) to (1);
		\draw [style=simple] (1) to (8);
		\draw [style=simple] (25) to (1);
		\draw [style=simple] (16) to (29);
		\draw [style=simple] (24) to (8);
		\draw [style=simple] (6) to (22);
		\draw [style=simple] (3) to (46);
		\draw [style=simple] (15) to (29);
		\draw [style=simple] (17) to (14);
		\draw [style=simple] (21) to (5);
		\draw [style=simple] (18) to (2);
		\draw [style=simple] (12) to (5);
		\draw [style=simple] (20) to (10);
		\draw [style=simple] (40) to (14);
		\draw [style=simple] (37) to (46);
		\draw [style=simple] (7) to (30);
		\draw [style=simple] (2) to (20);
		\draw [style=simple] (42) to (34);
		\draw [style=simple] (26) to (37);
		\draw [style=simple] (19) to (20);
		\draw [style=simple] (9) to (30);
		\draw [style=simple] (47) to (6);
		\draw [style=simple] (14) to (1);
		\draw [style=simple] (33) to (34);
		\draw [style=simple] (39) to (26);
		\draw [style=simple, red] (36) to (15);
	\end{pgfonlayer}
\end{tikzpicture}
\caption{\label{fig:3-ary trees}On the left, the set of complete 3-ary trees of size 0, 1 and 2, reinterpreted below as pointed labelled trees. On the right, growing a 3-ary tree at a leaf $v$.}
\end{figure}
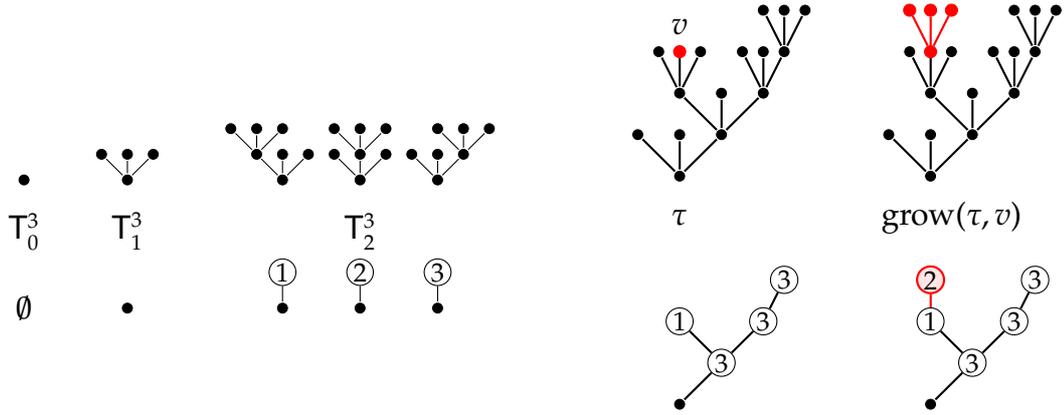

Finally, it is time to introduce a branching structure that is closer to the original presentation of the BDFG bijection:

\begin{definition}
Let $p\geq 2$ and $n\geq 1$ be integers. An \emph{unlabelled $p$-mobile of size $n$} is a rooted plane tree $\tau$ with the following properties:
\begin{itemize}
\item vertices of $V(\tau)$ are coloured either white or black and each edge of $\tau$ has a black endpoint and a white endpoint;
\item the origin is a white vertex;
\item each black vertex has degree $p$;
\item $\tau$ has $n$ black vertices.
\end{itemize}
A \emph{labelled $p$-mobile of size $n$} is an unlabelled $p$-mobile of size $n$ endowed with a labelling $l$ of its black vertices. Each black vertex is assigned an element in the set $S_p\subset\{\pm1\}^{2p-1}$ of strings of $\pm1$'s that have length $2p-1$ and sum 1 (or equivalently an integer between $1$ and $2p-1 \choose p$).
We call $\Mob^p_n$ the set of all unlabelled $p$-mobiles of size $n$ and $\LMob^p_n$ the set of all labelled $p$-mobiles of size $n$.

Conventionally, we let $\Mob^p_0=\LMob^p_0$ be the singleton of the \emph{degenerate $p$-mobile} consisting of a single white vertex.
\end{definition}

Given a corner $c$ of an unlabelled $p$-mobile $\tau$ that is adjacent to a white vertex, we let $\grow(\tau,c)$ be the $p$-mobile obtained by grafting a size 1 $p$-mobile (which is a star with a central black vertex and $p$ arms, rooted at one of the leaf corners) at the corner $c$ (see Figure~\ref{fig:growing mobiles}). If $\tau$ is a labelled $p$-mobile, for each $l\in S_p$ we let $\grow(\tau,c,l)$ be the labelled $p$-mobile obtained by grafting the star as above and labelling the new black vertex with the label $l$.
\begin{figure}\centering
\begin{tikzpicture}[vertex/.style={circle, fill=black, inner sep=2.2pt, draw=white, very thick}, gem/.style={circle, fill=white, inner sep=1.2pt, draw=black, thick}, simple/.style={thick}, scale=.9]
	\begin{pgfonlayer}{nodelayer}
		\node [style=vertex] (0) at (5.25, 2.75) {};
		\node [style=gem] (1) at (5.25, 2) {};
		\node [style=gem] (2) at (4.75, 3.5) {};
		\node [style=gem] (3) at (5.25, 3.5) {};
		\node [style=gem] (4) at (5.75, 3.5) {};
		\node [style=vertex] (5) at (6.25, 2.5) {};
		\node [style=gem] (6) at (6.25, 3.25) {};
		\node [style=gem] (7) at (6.75, 3.25) {};
		\node [style=gem] (8) at (7.25, 3.25) {};
		\node [style=vertex] (9) at (7.25, 4) {};
		\node [style=gem] (10) at (6.25, 4.75) {};
		\node [style=gem] (11) at (7, 4.75) {};
		\node [style=gem] (12) at (7.75, 4.75) {};
		\node [style=vertex] (13) at (4.75, 4.25) {};
		\node [style=gem] (14) at (4.25, 5) {};
		\node [style=gem] (15) at (4.75, 5) {};
		\node [style=gem] (16) at (5.25, 5) {};
		\node [style=vertex] (17) at (4.75, 6) {};
		\node [style=vertex] (18) at (5.75, 6) {};
		\node [style=gem] (19) at (4, 6.75) {};
		\node [style=gem] (20) at (4.5, 6.75) {};
		\node [style=gem] (21) at (5, 6.75) {};
		\node [style=gem] (22) at (5.5, 6.75) {};
		\node [style=gem] (23) at (6, 6.75) {};
		\node [style=gem] (24) at (6.5, 6.75) {};
		\node [style=vertex] (25) at (8.25, 5.5) {};
		\node [style=gem] (26) at (8.25, 6.25) {};
		\node [style=gem] (27) at (8.75, 6.25) {};
		\node [style=gem] (28) at (9.25, 6.25) {};
		\node [style=gem] (29) at (13, 3.5) {};
		\node [style=gem] (30) at (11.25, 6.75) {};
		\node [style=gem] (31) at (12, 3.5) {};
		\node [style=gem] (32) at (13.75, 6.75) {};
		\node [style=vertex] (33) at (12, 6) {};
		\node [style=gem] (34) at (13.5, 3.25) {};
		\node [style=vertex] (35) at (15.5, 5.5) {};
		\node [style=vertex] (36) at (12, 4.25) {};
		\node [style=gem] (37) at (13.5, 4.75) {};
		\node [style=gem] (38) at (15, 4.75) {};
		\node [style=gem] (39) at (12.5, 5) {};
		\node [style=gem] (40) at (16, 6.25) {};
		\node [style=gem] (41) at (14.25, 4.75) {};
		\node [style=gem] (42) at (12.5, 2) {};
		\node [style=gem] (43) at (12.5, 3.5) {};
		\node [style=gem] (44) at (13.25, 6.75) {};
		\node [style=vertex] (45) at (13.5, 2.5) {};
		\node [style=vertex] (46) at (14.5, 4) {};
		\node [style=gem] (47) at (14, 3.25) {};
		\node [style=vertex] (48) at (12.5, 2.75) {};
		\node [style=gem] (49) at (16.5, 6.25) {};
		\node [style=gem] (50) at (12.75, 6.75) {};
		\node [style=gem] (51) at (15.5, 6.25) {};
		\node [style=gem] (52) at (14.5, 3.25) {};
		\node [style=gem] (53) at (11.75, 6.75) {};
		\node [style=gem] (54) at (12.25, 6.75) {};
		\node [style=gem] (55) at (12, 5) {};
		\node [style=vertex] (56) at (13, 6) {};
		\node [style=gem] (57) at (11.5, 5) {};
		\node [style=vertex, fill=red] (58) at (14.5, 5.5) {};
		\node [style=gem, draw=red] (59) at (14, 6.25) {};
		\node [style=gem, draw=red] (60) at (14.5, 6.25) {};
		\node [style=gem, draw=red] (61) at (15, 6.25) {};
		\node at (5.25,1) {$\tau$};
		\node at (13,1) {$\grow(\tau,c)$};
	\end{pgfonlayer}
	\begin{pgfonlayer}{edgelayer}
	\begin{scope}
	\clip (9.center) to (25.center) to (13.center) to (9.center);
	\fill[red!20, draw=red] (7.75, 4.75) circle (8pt);
	\end{scope}
	\node[red] at (7.75, 5.25) {$c$};
		\draw [style=simple] (0) to (1);
		\draw [style=simple] (5) to (1);
		\draw [style=simple] (2) to (0);
		\draw [style=simple] (3) to (0);
		\draw [style=simple] (4) to (0);
		\draw [style=simple] (6) to (5);
		\draw [style=simple] (7) to (5);
		\draw [style=simple] (8) to (5);
		\draw [style=simple] (9) to (8);
		\draw [style=simple] (10) to (9);
		\draw [style=simple] (11) to (9);
		\draw [style=simple] (12) to (9);
		\draw [style=simple] (13) to (2);
		\draw [style=simple] (14) to (13);
		\draw [style=simple] (15) to (13);
		\draw [style=simple] (16) to (13);
		\draw [style=simple] (19) to (17);
		\draw [style=simple] (20) to (17);
		\draw [style=simple] (21) to (17);
		\draw [style=simple] (22) to (18);
		\draw [style=simple] (23) to (18);
		\draw [style=simple] (24) to (18);
		\draw [style=simple] (17) to (16);
		\draw [style=simple] (18) to (16);
		\draw [style=simple] (12) to (25);
		\draw [style=simple] (26) to (25);
		\draw [style=simple] (25) to (27);
		\draw [style=simple] (25) to (28);
		\draw [style=simple] (48) to (42);
		\draw [style=simple] (45) to (42);
		\draw [style=simple] (31) to (48);
		\draw [style=simple] (43) to (48);
		\draw [style=simple] (29) to (48);
		\draw [style=simple] (34) to (45);
		\draw [style=simple] (47) to (45);
		\draw [style=simple] (52) to (45);
		\draw [style=simple] (46) to (52);
		\draw [style=simple] (37) to (46);
		\draw [style=simple] (41) to (46);
		\draw [style=simple] (38) to (46);
		\draw [style=simple] (36) to (31);
		\draw [style=simple] (57) to (36);
		\draw [style=simple] (55) to (36);
		\draw [style=simple] (39) to (36);
		\draw [style=simple] (30) to (33);
		\draw [style=simple] (53) to (33);
		\draw [style=simple] (54) to (33);
		\draw [style=simple] (50) to (56);
		\draw [style=simple] (44) to (56);
		\draw [style=simple] (32) to (56);
		\draw [style=simple] (33) to (39);
		\draw [style=simple] (56) to (39);
		\draw [style=simple] (38) to (35);
		\draw [style=simple] (51) to (35);
		\draw [style=simple] (35) to (40);
		\draw [style=simple] (35) to (49);
		\draw [style=simple, red] (59) to (58);
		\draw [style=simple, red] (60) to (58);
		\draw [style=simple, red] (61) to (58);
		\draw [style=simple, red] (58) to (38);
	\end{pgfonlayer}
\end{tikzpicture}
\caption{\label{fig:growing mobiles}On the left, an unlabelled mobile $\tau\in \Mob_7^4$ with a marked corner $c$. On the right, the mobile $\grow(\tau,c)\in\Mob_8^4$ obtained by grafting the size one $4$-mobile onto the corner $c$.}
\end{figure}
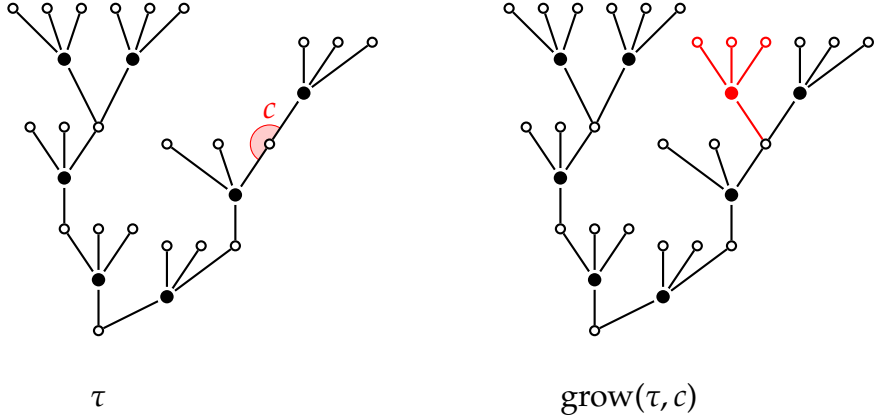

We now make some observations about unlabelled $d$-mobiles of size $n$, their enumeration and their relation to $d$-ary trees. We will then discuss how labelled $d$-mobiles can be interpreted as $2d$-angulations.

\begin{remark}\label{rem:mobile to tree}
Any element $\tau\in\Mob^p_n$, where $n\geq 1$, can be represented as an ordered $p$-tuple $(\tau_1,\ldots,\tau_p)\in\cup_{n_1+\ldots+n_{p}=n-1}\Mob^p_{n_1}\times\ldots\times\Mob^p_{n_{p}}$ by considering the rightmost black vertex of height 1 in $\tau$, denoted by $v$, and letting $\tau_1$ be the mobile obtained by erasing $v$ and all of its descendants; we then let $\tau_2,\ldots,\tau_p,$ be the submobiles of descendants of the $p-1$ children of $v$, ordered from left to right. This immediately yields the recursion
$$|\Mob^p_n|=\sum_{n_1+\ldots+n_p=n-1}\prod_{i=1}^p |\Mob^p_{n_i}|.$$

Clearly, complete $p$-ary trees follow the same recursive structure (a tree in $\CT^p_n$ corresponds to the $p$-tuple of subtrees of descendants of the $p$ children of the origin), and we have $|\Mob^p_0|=|\CT^p_0|=1$, which yields that
\begin{equation}\label{CT and Mob cardinality}|\Mob^p_n|=|\CT^p_n|=\frac{1}{(p-1)n+1}{pn \choose n}.\end{equation}

Not only that, but the recursion yields an explicit bijection $\Phi:\Mob^p_n\to\CT^p_n$ that one can define inductively: given $\tau\in\Mob^p_n$, $\Phi(\tau)$ is obtained by considering the complete $p$-ary trees $\Phi(\tau_1),\ldots,\Phi(\tau_p)$ and grafting them in order (from left to right) onto the $p$ leaves of a complete $p$-ary tree of size $1$ (see Figure~\ref{fig:mob to d-ary tree} for an example of this correspondence).

As for labelled $p$-mobiles, we have that $|\LMob_n^p|={2p-1\choose p}^n|\Mob_n^p|$, and the bijection $\Phi$ can be used to interpret elements of $\LMob_n^p$ as pairs $(t,l)$, where $t\in \CT^p_n$ and $l$ is a map from the set of vertices of $t$ that are not leaves to the set $\{1,\ldots, {2p-1 \choose p}\}$.
\end{remark}

\begin{figure}\centering
\begin{tikzpicture}[vertex/.style={circle, fill=black, inner sep=2.2pt, draw=white, very thick}, gem/.style={circle, fill=white, inner sep=1.2pt, draw=black, thick}, simple/.style={thick}, scale=.9]
	\begin{pgfonlayer}{nodelayer}
		\node [style=gem, draw=blue] (29) at (13, 3.5) {};
		\node [style=gem, draw=blue] (30) at (11.25, 6.75) {};
		\node [style=gem, draw=blue] (31) at (12, 3.5) {};
		\node [style=gem, draw=blue] (32) at (13.75, 6.75) {};
		\node [style=vertex, fill=blue] (33) at (12, 6) {};
		\node [style=gem, draw=purple] (34) at (13.5, 3.25) {};
		\node [style=vertex, fill=red] (35) at (15.5, 5.5) {};
		\node [style=vertex, fill=blue] (36) at (12, 4.25) {};
		\node [style=gem, draw=red] (37) at (13.5, 4.75) {};
		\node [style=gem, draw=red] (38) at (15, 4.75) {};
		\node [style=gem, draw=blue] (39) at (12.5, 5) {};
		\node [style=gem, draw=red] (40) at (16, 6.25) {};
		\node [style=gem, draw=red] (41) at (14.25, 4.75) {};
		\node [style=gem, draw=blue] (42) at (12.5, 2) {};
		\node [style=gem, draw=blue] (43) at (12.5, 3.5) {};
		\node [style=gem, draw=blue] (44) at (13.25, 6.75) {};
		\node [style=vertex] (45) at (13.5, 2.5) {};
		\node [style=vertex, fill=red] (46) at (14.5, 4) {};
		\node [style=gem, draw=green!50!black] (47) at (14, 3.25) {};
		\node [style=vertex, fill=blue] (48) at (12.5, 2.75) {};
		\node [style=gem, draw=red] (49) at (16.5, 6.25) {};
		\node [style=gem, draw=blue] (50) at (12.75, 6.75) {};
		\node [style=gem, draw=red] (51) at (15.5, 6.25) {};
		\node [style=gem, draw=red] (52) at (14.5, 3.25) {};
		\node [style=gem, draw=blue] (53) at (11.75, 6.75) {};
		\node [style=gem, draw=blue] (54) at (12.25, 6.75) {};
		\node [style=gem, draw=blue] (55) at (12, 5) {};
		\node [style=vertex, fill=blue] (56) at (13, 6) {};
		\node [style=gem, draw=blue] (57) at (11.5, 5) {};
		\node [style=vertex, fill=red] (58) at (14.5, 5.5) {};
		\node [style=gem, draw=red] (59) at (14, 6.25) {};
		\node [style=gem, draw=red] (60) at (14.5, 6.25) {};
		\node [style=gem, draw=red] (61) at (15, 6.25) {};
		\node at (13,1) {$\tau=(\textcolor{blue}{\tau_1},\textcolor{purple}{\tau_2},\textcolor{green!50!black}{\tau_3},\textcolor{red}{\tau_4})\in\Mob^4_8$};
	\end{pgfonlayer}
	\begin{pgfonlayer}{edgelayer}
		\draw [style=simple, blue] (48) to (42);
		\draw [style=simple] (45) to (42);
		\draw [style=simple, blue] (31) to (48);
		\draw [style=simple, blue] (43) to (48);
		\draw [style=simple, blue] (29) to (48);
		\draw [style=simple] (34) to (45);
		\draw [style=simple] (47) to (45);
		\draw [style=simple] (52) to (45);
		\draw [style=simple, red] (46) to (52);
		\draw [style=simple, red] (37) to (46);
		\draw [style=simple, red] (41) to (46);
		\draw [style=simple, red] (38) to (46);
		\draw [style=simple, blue] (36) to (31);
		\draw [style=simple, blue] (57) to (36);
		\draw [style=simple, blue] (55) to (36);
		\draw [style=simple, blue] (39) to (36);
		\draw [style=simple, blue] (30) to (33);
		\draw [style=simple, blue] (53) to (33);
		\draw [style=simple, blue] (54) to (33);
		\draw [style=simple, blue] (50) to (56);
		\draw [style=simple, blue] (44) to (56);
		\draw [style=simple, blue] (32) to (56);
		\draw [style=simple, blue] (33) to (39);
		\draw [style=simple, blue] (56) to (39);
		\draw [style=simple, red] (38) to (35);
		\draw [style=simple, red] (51) to (35);
		\draw [style=simple, red] (35) to (40);
		\draw [style=simple, red] (35) to (49);
		\draw [style=simple, red] (59) to (58);
		\draw [style=simple, red] (60) to (58);
		\draw [style=simple, red] (61) to (58);
		\draw [style=simple, red] (58) to (38);
	\end{pgfonlayer}
\end{tikzpicture}
\begin{tikzpicture}[vertex/.style={circle, fill=black, inner sep=2.2pt, draw=white, very thick}, gem/.style={circle, fill=white, inner sep=1.2pt, draw=black, thick}, simple/.style={thick}, scale=.9]
	\begin{pgfonlayer}{nodelayer}
		\node at (0,-3) {$\Phi(\tau)\in\CT^4_8$};
		\node [style=vertex] (0) at (0, -2) {};
		\node [style=vertex, fill=blue] (1) at (-0.75, -1.25) {};
		\node [style=vertex, fill=purple] (2) at (-0.25, -1.25) {};
		\node [style=vertex, fill=green!50!black] (3) at (0.25, -1.25) {};
		\node [style=vertex, fill=red] (4) at (0.75, -1.25) {};
		\node [style=vertex, fill=red] (5) at (0.25, -0.5) {};
		\node [style=vertex, fill=red] (6) at (0.75, -0.5) {};
		\node [style=vertex, fill=red] (7) at (1.25, -0.5) {};
		\node [style=vertex, fill=red] (8) at (1.75, -0.5) {};
		\node [style=vertex, fill=red] (9) at (1.75, 0.25) {};
		\node [style=vertex, fill=red] (10) at (2.25, 0.25) {};
		\node [style=vertex, fill=red] (11) at (2.75, 0.25) {};
		\node [style=vertex, fill=red] (12) at (3.25, 0.25) {};
		\node [style=vertex, fill=red] (13) at (1, 1) {};
		\node [style=vertex, fill=red] (14) at (1.5, 1) {};
		\node [style=vertex, fill=red] (15) at (2, 1) {};
		\node [style=vertex, fill=red] (16) at (2.5, 1) {};
		\node [style=vertex, fill=blue] (17) at (-2, -0.5) {};
		\node [style=vertex, fill=blue] (18) at (-1.5, -0.5) {};
		\node [style=vertex, fill=blue] (19) at (-1, -0.5) {};
		\node [style=vertex, fill=blue] (20) at (-0.5, -0.5) {};
		\node [style=vertex, fill=blue] (21) at (-2, 0.25) {};
		\node [style=vertex, fill=blue] (22) at (-1.5, 0.25) {};
		\node [style=vertex, fill=blue] (23) at (-1, 0.25) {};
		\node [style=vertex, fill=blue] (24) at (-0.5, 0.25) {};
		\node [style=vertex, fill=blue] (25) at (-1, 1) {};
		\node [style=vertex, fill=blue] (26) at (-0.5, 1) {};
		\node [style=vertex, fill=blue] (27) at (0, 1) {};
		\node [style=vertex, fill=blue] (28) at (0.5, 1) {};
		\node [style=vertex, fill=blue] (29) at (-1.75, 1.75) {};
		\node [style=vertex, fill=blue] (30) at (-1.25, 1.75) {};
		\node [style=vertex, fill=blue] (31) at (-0.75, 1.75) {};
		\node [style=vertex, fill=blue] (32) at (-0.25, 1.75) {};
	\end{pgfonlayer}
	\begin{pgfonlayer}{edgelayer}
		\draw [style=simple] (1) to (0);
		\draw [style=simple] (2) to (0);
		\draw [style=simple] (3) to (0);
		\draw [style=simple] (4) to (0);
		\draw [style=simple, red] (5) to (4);
		\draw [style=simple, red] (6) to (4);
		\draw [style=simple, red] (7) to (4);
		\draw [style=simple, red] (4) to (8);
		\draw [style=simple, red] (9) to (8);
		\draw [style=simple, red] (10) to (8);
		\draw [style=simple, red] (11) to (8);
		\draw [style=simple, red] (12) to (8);
		\draw [style=simple, red] (13) to (9);
		\draw [style=simple, red] (14) to (9);
		\draw [style=simple, red] (15) to (9);
		\draw [style=simple, red] (16) to (9);
		\draw [style=simple, blue] (21) to (18);
		\draw [style=simple, blue] (22) to (18);
		\draw [style=simple, blue] (23) to (18);
		\draw [style=simple, blue] (24) to (18);
		\draw [style=simple, blue] (17) to (1);
		\draw [style=simple, blue] (18) to (1);
		\draw [style=simple, blue] (19) to (1);
		\draw [style=simple, blue] (20) to (1);
		\draw [style=simple, blue] (29) to (25);
		\draw [style=simple, blue] (30) to (25);
		\draw [style=simple, blue] (31) to (25);
		\draw [style=simple, blue] (25) to (32);
		\draw [style=simple, blue] (25) to (24);
		\draw [style=simple, blue] (26) to (24);
		\draw [style=simple, blue] (27) to (24);
		\draw [style=simple, blue] (28) to (24);
	\end{pgfonlayer}
\end{tikzpicture}
\caption{\label{fig:mob to d-ary tree}On the left, a mobile $\tau\in\Mob_8^4$ and its decomposition into the blue, purple, green and red submobiles $\tau_1\in\Mob_4^4$, $\tau_2\in\Mob_0^4$, $\tau_3\in\Mob_0^4$, $\tau_4\in\Mob_3^4$. On the right, the corresponding 4-ary tree $\Phi(\tau)$ obtained by grafting $\Phi(\tau_1), \Phi(\tau_2), \Phi(\tau_3), \Phi(\tau_4)$ onto the four children of the origin.}
\end{figure}
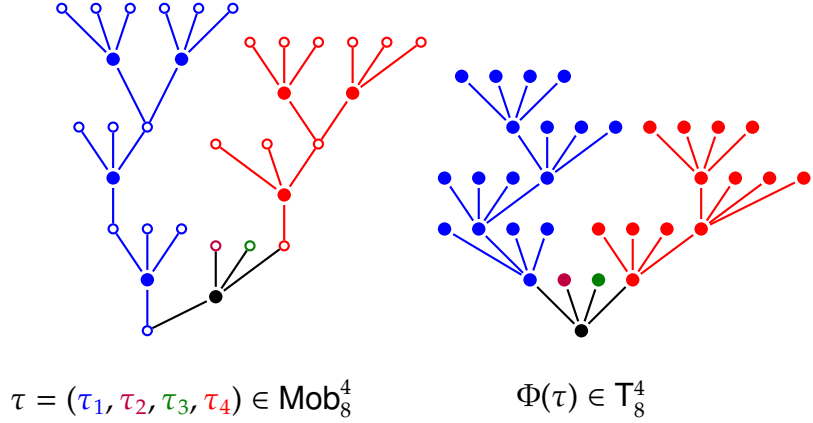

The final link we need to describe is the one between $2p$-angulations and labelled $p$-mobiles, which has been thoroughly explored \cite{BDFG04}. We shall outline the BDFG construction of a pointed, rooted $2p$-angulation $\Psi(\tau,\epsilon)\in\M^{2p,\bullet}_n$ from a pair $(\tau,\epsilon)\in\LMob^p_n\times \{+1,-1\}$ and summarise at the end of this section the main features of this correspondence $\Psi$ that we will need.

\begin{construction}[Bouttier, Di Francesco, Guitter]\label{BDFG}
Given a pair $(\tau,\epsilon)\in\LMob^p_n\times\{\pm1\}$, construct the map $\Psi(\tau)$ as follows:
\begin{itemize}
\item first, deduce labels for the white vertices from the labels of the black vertices, as follows. Label the origin, which is a white vertex, with zero. Suppose $v$ is a black vertex labelled $l=(l_1,l_2,\ldots,l_{2p-1})\in S_p\subset\{\pm1\}^{2p-1}$ such that the label of its (white) parent has already been determined and is equal to $a$. Let $i_1<i_2<\ldots<i_{p-1}$ be such that $l_{i_1}=\ldots=l_{i_{p-1}}=-1$. Label the $p-1$ (white) children of $v$ from left to right as $a+i_j-2j$ for $j=1,\ldots,p-1$. Continue until all white vertices are labelled with an integer (see Figure~\ref{fig:BDFG}).  
\item Draw an additional white vertex $\delta$.
\item For each corner $c$ of a white vertex other than $\delta$,  draw a new ``map edge'' $e(c)$ issued from $c$. Assuming the white vertex carrying $c$ is labelled $a$, let $t(c)$ (the ``target corner'' of $c$) be the first corner of a white vertex labelled $a-1$ in a clockwise contour of the mobile started from the corner $c$, if such a corner exists. If the label $a-1$ does not appear on white vertices, then set $t(c)$ to be the corner of $\delta$. For each $c$, draw an edge $e(c)$ joining $c$ to $t(c)$, taking care to perform this operation in such a way that edges do not cross (which is always possible).
\item Erase all edges of the original mobile and all black vertices. Forget all labels. Root the ensuing map in the edge drawn from the root corner of $\tau$, oriented away from the corner if $\epsilon=1$ and towards the corner if $\epsilon=-1$. Make $\delta$ the distinguished vertex.
\end{itemize}
\end{construction}

For most applications, an important feature of Construction~\ref{BDFG} is the significance of the labels in terms of graph distances on the $2p$-angulation. This will not be the case in this paper; the main features we need are given by
\begin{proposition}
For each $n\geq 1$, the map $\Psi$ of Construction~\ref{BDFG} is a bijection from the set $\LMob^p_n\times \{\pm1\}$ to the set $\M^{2p,\bullet}_n$ of pointed, rooted $2p$-angulations with $n$ faces. Given $\tau\in\LMob^p_n$, this bijection induces a correspondence between black vertices of $\tau$ and faces of $\Psi(\tau)$ and between white vertices of $\tau$ and vertices of $\Psi(\tau)$ other than the distinguished vertex $\delta$ given by the pointing. Moreover, it induces a correspondence between corners of black vertices of $\tau$ and corners of $\Psi(\tau)$, and between corners of white vertices of $\tau$ and edges of $\Psi(\tau)$. The root corner of $\tau$ corresponds to the root edge of $\Psi(\tau)$.
\end{proposition}

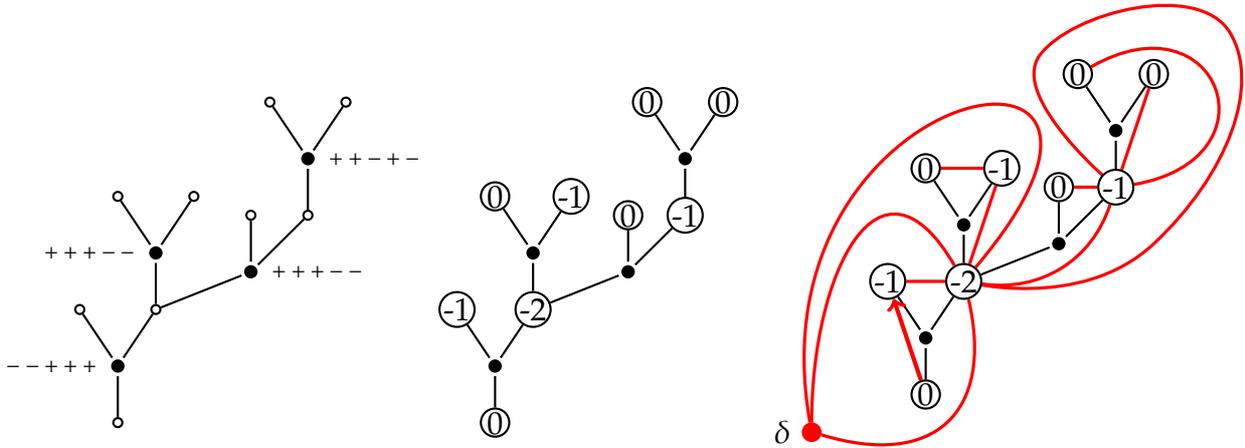
\begin{figure}\centering
\begin{tikzpicture}[vertex/.style={circle, fill=black, inner sep=2.2pt, draw=white, very thick}, gem/.style={circle, fill=white, inner sep=1.2pt, draw=black, thick}, simple/.style={thick}, scale=1]
	\begin{pgfonlayer}{nodelayer}
		\node [style=gem] (0) at (0, -0) {};
		\node [style=vertex, label=left:\scriptsize$--+++$] (1) at (0, 0.75) {};
		\node [style=gem] (2) at (-0.5, 1.5) {};
		\node [style=gem] (3) at (0.5, 1.5) {};
		\node [style=vertex, label=left:\scriptsize$+++--$] (4) at (0.5, 2.25) {};
		\node [style=gem] (5) at (0, 3) {};
		\node [style=gem] (6) at (1, 3) {};
		\node [style=vertex, label=right:\scriptsize$+++--$] (7) at (1.75, 2) {};
		\node [style=gem] (8) at (1.75, 2.75) {};
		\node [style=gem] (9) at (2.5, 2.75) {};
		\node [style=vertex, label=right:\scriptsize$++-+-$] (10) at (2.5, 3.5) {};
		\node [style=gem] (11) at (2, 4.25) {};
		\node [style=gem] (12) at (3, 4.25) {};
		\node (delta) at (0,-0.5) {};
	\end{pgfonlayer}
	\begin{pgfonlayer}{edgelayer}
		\draw [style=simple] (1) to (0);
		\draw [style=simple] (2) to (1);
		\draw [style=simple] (3) to (1);
		\draw [style=simple] (4) to (3);
		\draw [style=simple] (5) to (4);
		\draw [style=simple] (6) to (4);
		\draw [style=simple] (3) to (7);
		\draw [style=simple] (8) to (7);
		\draw [style=simple] (9) to (7);
		\draw [style=simple] (10) to (9);
		\draw [style=simple] (11) to (10);
		\draw [style=simple] (12) to (10);
	\end{pgfonlayer}
\end{tikzpicture}
\begin{tikzpicture}[vertex/.style={circle, fill=black, inner sep=2.2pt, draw=white, very thick}, gem/.style={circle, fill=white, inner sep=0pt, draw=black, thick}, simple/.style={thick}, scale=1]
	\begin{pgfonlayer}{nodelayer}
		\node [style=gem] (0) at (0, -0) {0};
		\node [style=vertex] (1) at (0, 0.75) {};
		\node [style=gem] (2) at (-0.5, 1.5) {-1};
		\node [style=gem] (3) at (0.5, 1.5) {-2};
		\node [style=vertex] (4) at (0.5, 2.25) {};
		\node [style=gem] (5) at (0, 3) {0};
		\node [style=gem] (6) at (1, 3) {-1};
		\node [style=vertex] (7) at (1.75, 2) {};
		\node [style=gem] (8) at (1.75, 2.75) {0};
		\node [style=gem] (9) at (2.5, 2.75) {-1};
		\node [style=vertex] (10) at (2.5, 3.5) {};
		\node [style=gem] (11) at (2, 4.25) {0};
		\node [style=gem] (12) at (3, 4.25) {0};
		\node (delta) at (0,-0.5) {};
	\end{pgfonlayer}
	\begin{pgfonlayer}{edgelayer}
		\draw [style=simple] (1) to (0);
		\draw [style=simple] (2) to (1);
		\draw [style=simple] (3) to (1);
		\draw [style=simple] (4) to (3);
		\draw [style=simple] (5) to (4);
		\draw [style=simple] (6) to (4);
		\draw [style=simple] (3) to (7);
		\draw [style=simple] (8) to (7);
		\draw [style=simple] (9) to (7);
		\draw [style=simple] (10) to (9);
		\draw [style=simple] (11) to (10);
		\draw [style=simple] (12) to (10);
	\end{pgfonlayer}
\end{tikzpicture}
\begin{tikzpicture}[vertex/.style={circle, fill=black, inner sep=2.2pt, draw=white, very thick}, gem/.style={circle, fill=white, inner sep=0pt, draw=black, thick}, simple/.style={thick}, map/.style={red, very thick}, scale=1]
	\begin{pgfonlayer}{nodelayer}
		\node [style=gem] (0) at (0, -0) {0};
		\node [style=vertex] (1) at (0, 0.75) {};
		\node [style=gem] (2) at (-0.5, 1.5) {-1};
		\node [style=gem] (3) at (0.5, 1.5) {-2};
		\node [style=vertex] (4) at (0.5, 2.25) {};
		\node [style=gem] (5) at (0, 3) {0};
		\node [style=gem] (6) at (1, 3) {-1};
		\node [style=vertex] (7) at (1.75, 2) {};
		\node [style=gem] (8) at (1.75, 2.75) {0};
		\node [style=gem] (9) at (2.5, 2.75) {-1};
		\node [style=vertex] (10) at (2.5, 3.5) {};
		\node [style=gem] (11) at (2, 4.25) {0};
		\node [style=gem] (12) at (3, 4.25) {0};
		\node [style=vertex, red, label=left:$\delta$] (delta) at (-1.5,-0.5) {};
	\end{pgfonlayer}
	\begin{pgfonlayer}{edgelayer}
		\draw [style=simple] (1) to (0);
		\draw [style=simple] (2) to (1);
		\draw [style=simple] (3) to (1);
		\draw [style=simple] (4) to (3);
		\draw [style=simple] (5) to (4);
		\draw [style=simple] (6) to (4);
		\draw [style=simple] (3) to (7);
		\draw [style=simple] (8) to (7);
		\draw [style=simple] (9) to (7);
		\draw [style=simple] (10) to (9);
		\draw [style=simple] (11) to (10);
		\draw [style=simple] (12) to (10);
		\draw[map] (2)--(3) (5)--(6) (8)--(9)--(12) (6)--(3);
		\draw[map, ->, ultra thick] (0)--(2);
		\draw[map, bend left=40] (9) to (3);
		\draw[map, in=120, out=90, looseness=1.8] (delta) to (3);
		\draw[map, in=50, out=100, looseness=4] (delta) to (3);
		\draw[map, bend right=60, looseness=1.5] (delta) to (3);
		\draw[map, bend left=100, looseness=3] (11) to (9);
		\draw[map] plot [smooth, tension=1] coordinates {(2.5,2.75) (1.5, 4.5) (4,4.8) (3,2) (0.5,1.5)};
	\end{pgfonlayer}
\end{tikzpicture}
\caption{\label{fig:BDFG}On the left, a labelled mobile $\tau\in\LMob^3_4$. In the middle, the labels of black vertices are turned into integer labels for white vertices. On the right, the full construction of the pointed 6-angulation $\Psi(\tau,1)$, consisting of the red edges, white vertices and distinguished vertex $\delta$. Note how every face contains a single black vertex and has degree 6.}
\end{figure}

\section{Simple triangulations and blossoming trees}\label{triangulations, blossoming trees}

\begin{definition}
A \emph{simple triangulation} of size $n$ is a rooted map in $\M^3_n$ with no loops or multiple edges.
\end{definition}

Note that since in a simple triangulation every edge is adjacent to exactly two faces (this is not necessarily the case if the triangulation is not simple: if loops are allowed, then a triangular face might contain one edge in its interior), the number of faces has to be even; indeed, the number of edges of a simple triangulation with $F$ faces is $E=\frac32 F$. Moreover, Euler's formula yields that the triangulation must have $V=\frac12 F+2$ vertices.

It is therefore convenient to index simple triangulations by half their number of faces: we will denote the set of all simple triangulations of size $2n$ by $\STr_n$; they have $2n$ faces, $n+2$ vertices and $3n$ edges.

A convenient tree structure to define in conjunction to simple triangulation is blossoming trees, used in \cite{PS06}:

\begin{definition}
A \emph{blossoming tree} of size $n$ (Figure~\ref{fig:PS}, left) is a rooted plane tree $t$ whose vertices are of two types, \emph{blossoms} and \emph{non-blossoms}, with the following properties:
\begin{itemize}
\item every blossom has degree 1 (it is a leaf);
\item every non-blossom is adjacent to exactly two blossoms;
\item $t$ has exactly $n$ non-blossoms;
\item the origin of $t$ is a blossom.
\end{itemize}
We denote by $|t|$ the size of the blossoming tree $t$ and by $\BT_n$ the set of all blossoming trees of size $n$. Note that any blossoming tree of size $n$ has exactly $2n$ blossoms.	
\end{definition}

\begin{figure}
\centering
\begin{tikzpicture}[vertex/.style={circle, inner sep=2pt, fill=black}, blossom/.style={diamond, inner sep=2pt, fill=green!70!black}, branch/.style={very thick, brown}, gem/.style={green!70!black}]
\begin{pgfonlayer}{nodelayer}
		\node [style=blossom, label=below:\phantom{$R_3$}] (0) at (-3, -1.5) {};
		\node [style=vertex] (1) at (-3, -1) {};
		\node [style=vertex] (2) at (-4.25, -0) {};
		\node [style=vertex] (3) at (-4.25, 1) {};
		\node [style=vertex] (4) at (-3, -0) {};
		\node [style=vertex] (5) at (-2, -0) {};
		\node [style=vertex] (6) at (-1.25, 0.75) {};
		\node [style=blossom] (7) at (-3.25, -0.5) {};
		\node [style=blossom] (8) at (-4.75, -0) {};
		\node [style=blossom] (9) at (-4.75, 0.5) {};
		\node [style=blossom] (10) at (-4.5, 1.5) {};
		\node [style=blossom] (11) at (-4, 1.5) {};
		\node [style=blossom] (12) at (-3.25, 0.5) {};
		\node [style=blossom] (13) at (-2.75, 0.5) {};
		\node [style=blossom] (14) at (-2, 0.5) {};
		\node [style=blossom] (15) at (-1.5, -0) {};
		\node [style=blossom] (16) at (-1.5, 1.25) {};
		\node [style=blossom] (17) at (-1, 1.25) {};
	\end{pgfonlayer}
	\begin{pgfonlayer}{edgelayer}
		\draw [gem] (1) to (0);
		\draw [branch] (4) to (1);
		\draw [branch] (2) to (1);
		\draw [branch] (5) to (1);
		\draw [branch] (3) to (2);
		\draw [branch] (6) to (5);
		\draw [gem](8) to (2);
		\draw [gem](9) to (2);
		\draw [gem](10) to (3);
		\draw [gem](3) to (11);
		\draw [gem](7) to (1);
		\draw [gem](12) to (4);
		\draw [gem](4) to (13);
		\draw [gem](14) to (5);
		\draw [gem](5) to (15);
		\draw [gem](16) to (6);
		\draw [gem](6) to (17);
	\end{pgfonlayer}
\end{tikzpicture}
\begin{tikzpicture}[vertex/.style={circle, inner sep=2pt, fill=black}, blossom/.style={diamond, inner sep=2pt, fill=green!70!black}, branch/.style={very thick, brown},gem/.style={green!70!black}]
	\begin{pgfonlayer}{nodelayer}
		\node [style=blossom, label=below:$R_3$] (0) at (-3, -1.5) {};
		\node [style=vertex] (1) at (-3, -1) {};
		\node [style=vertex] (2) at (-4.25, -0) {};
		\node [style=vertex] (3) at (-4.25, 1) {};
		\node [style=vertex] (4) at (-3, -0) {};
		\node [style=vertex] (5) at (-2, -0) {};
		\node [style=vertex] (6) at (-1.25, 0.75) {};
		\node [style=blossom, label={$L_3$}] (7) at (-3.25, -0.25) {};
		\node [style=blossom, label=left:$R_4$] (8) at (-4.75, -0) {};
		\node [style=blossom, label=$L_1$] (9) at (-4.75, 0.5) {};
		\node [style=blossom, label=$L_2$] (10) at (-4.5, 1.5) {};
		\node [style=blossom, label=$L_4$] (11) at (-3, 0.5) {};
		\node [style=blossom, label=$L_5$] (12) at (-2, 0.5) {};
		\node [style=blossom, label=right:$R_2$] (13) at (-1.5, -0) {};
		\node [style=blossom, label=$L_6$] (14) at (-1.5, 1.25) {};
		\node [style=blossom, label=$R_1$] (15) at (-1, 1.25) {};
	\end{pgfonlayer}
	\begin{pgfonlayer}{edgelayer}
		\draw [gem](1) to (0);
		\draw [branch] (4) to (1);
		\draw [branch] (2) to (1);
		\draw [branch] (5) to (1);
		\draw [branch] (3) to (2);
		\draw [branch] (6) to (5);
		\draw [gem](8) to (2);
		\draw [gem](9) to (2);
		\draw [gem](10) to (3);
		\draw [gem](7) to (1);
		\draw [gem](11) to (4);
		\draw [gem](12) to (5);
		\draw [gem](5) to (13);
		\draw [gem](14) to (6);
		\draw [gem](6) to (15);
		\draw [thick, gem] (3) to (1);
		\draw [thick, gem] (4) to (5);
	\end{pgfonlayer}
\end{tikzpicture}
\begin{tikzpicture}[vertex/.style={circle, inner sep=2pt, fill=black}, rn/.style={rectangle, inner sep=3pt, fill=red}, branch/.style={very thick, brown},gem/.style={green!70!black}]
	\begin{pgfonlayer}{nodelayer}
		\node [style=vertex] (0) at (-3, -1) {};
		\node [style=vertex] (1) at (-4.25, -0) {};
		\node [style=vertex] (2) at (-4.25, 1) {};
		\node [style=vertex] (3) at (-3, -0) {};
		\node [style=vertex] (4) at (-2, -0) {};
		\node [style=vertex] (5) at (-1.25, 0.75) {};
		\node [style=rn, label=above:$L$] (6) at (-3, 2.5) {};
		\node [style=rn, label=left:$R$] (7) at (-3, -2.25) {};
	\end{pgfonlayer}
	\begin{pgfonlayer}{edgelayer}
		\draw [branch] (3) to (0) (1) to (0) (4) to (0) (2) to (1) (5) to (4);
		\draw [gem, thick](2) to (0);
		\draw [gem, thick](3) to (4);
		\draw [orange]  (1) to (7)  (0) to (7) (4) to (7)  (5) to (7);
		\draw [blue, bend left=75, looseness=1.25] (1) to (6);
		\draw [blue]  (2) to (6);
		\draw [blue, bend left, looseness=0.75] (0) to (6);
		\draw [blue] (3) to (6)  (6) to (4) (6) to (5);
		\draw [bend left=75, looseness=1.75, very thick, red, ->] (6) to (7);
	\end{pgfonlayer}
\end{tikzpicture}\caption{\label{fig:PS}On the left, a blossoming tree $t\in \BT_6$. In the middle, two closure operations have been performed on $t$ to create two new triangular faces. No other closure operations are possible and the remaining 10 blossoms have been labelled $L_1$ through $L_6$ and $R_1$ through $R_4$. Finally, on the right, the triangulation $\Xi(t,1)\in\STr_6$. Its root edge is marked in red; the thicker brown edges are those involving non-blossoms in the original tree. The two green edges are the ones obtained from the two initial closures and the blue and orange edges are obtained from the remaining blossoms during the last steps of the construction.}
\end{figure}
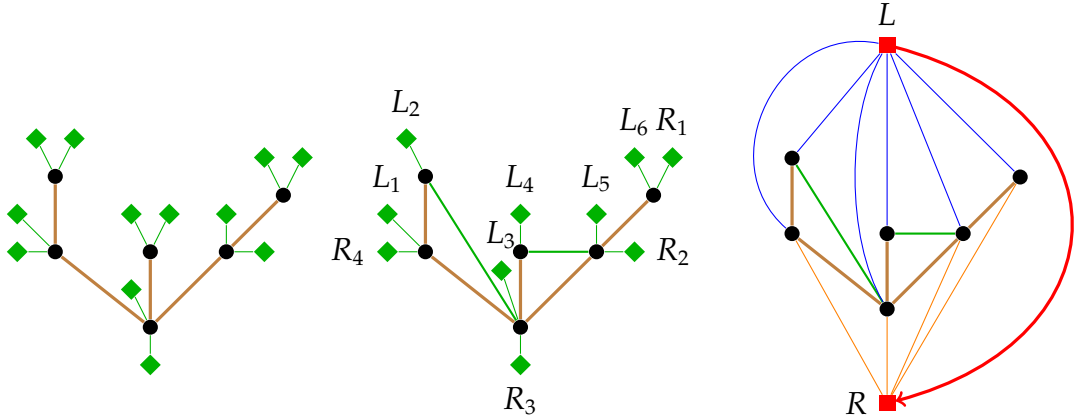

Poulalhon and Schaeffer provide in \cite{PS06} an explicit bijection between the set $\STr_n$ and a special subset of $\BT_n$ containing \emph{balanced blossoming trees}; for our purposes, a natural $2n$-to-$1$ map from $\BT_n\times\{\pm1\}\to\STr_n$, is actually enough. We briefly remind the reader of the construction of a triangulation in $	\STr_n$ from a blossoming tree in $\BT_n$ and a spin $\epsilon=\pm1$, highlighting those features that we will actually need.

\begin{construction}[Poulalhon, Schaeffer]\label{PS construction}
Given $t\in\BT_n$ and $\epsilon=\pm1$, the construction of $\Xi(t,\epsilon)\in\STr_n$ goes as follows:
\begin{itemize}
\item First, perform a sequence of \emph{closure} operations, each of which creates one new face. Consider the cyclic counterclockwise contour of the infinite face in the plane (which is initially the only face); if there is a corner $c$ of a blossom that is immediately followed by corners $c_1, c_2, c_3$ adjacent to non-blossoms, we say that $c$ is \emph{closable}. One performs the closure of $c$ by bringing its blossom to the vertex of $c_3$ as in Figure~\ref{fig:closure}, thereby creating a new triangular face. After a closure is performed, re-compute the contour of the infinite face and keep performing closures until no more blossom corners are closeable. Note that the order in which closures are performed is not important.
\item When no more closures are possible, the resulting map has a number of triangular faces, and one infinite face that is not necessarily triangular. It can be shown that the boundary of the infinite face $f_\infty$ is a cycle of non-blossoms, each of which is attached to one or two blossoms within $f_{\infty}$ that have not been closed, and that there are exactly two non-blossoms that each have two blossoms attached to them (Figure~\ref{fig:PS}, middle). As in the figure, label the blossoms in $f_\infty$ as $L_1,\ldots,L_k, R_1,\ldots,R_h$, clockwise around $f_\infty$, in such a way that $L_1,R_h$ and $L_k,R_1$ are attached to the same non-blossom, and additionally that, if you consider the clockwise contour of the original blossoming tree started at the root corner, the corner of $R_1$ occurs before the corner of $L_1$.
\item Draw two new vertices $L$ and $R$ within $f_\infty$. Identify $L_1,\ldots,L_k$ with $L$ and $R_1,\ldots,R_h$ with $R$.
\item Draw an edge between $L$ and $R$, which will be the root edge of the triangulation $\Xi(t,\epsilon)$. If $\epsilon=1$, orient it from $L$ to $R$; if $\epsilon=-1$, orient it from $R$ to $L$.
\end{itemize}	
\end{construction}

Note that the root blossom serves no purpose in the construction described above (other than determining how $\epsilon$ pairs with the orientation of the root edge), and is simply forgotten, which is why $\Xi$ is a $2n$-to-one map (since $t\in\BT_n$ has $2n$ blossoms, each of which can potentially serve as the origin).

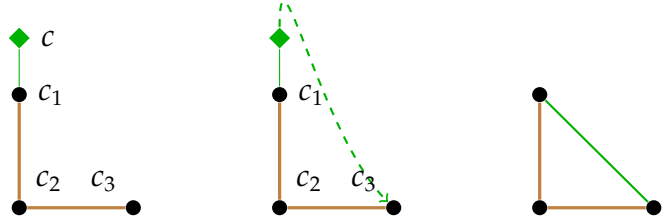
\begin{figure}
\centering
\begin{tikzpicture}[vertex/.style={circle, inner sep=2pt, fill=black}, blossom/.style={diamond, inner sep=2pt, fill=green!70!black}, branch/.style={very thick, brown},gem/.style={green!70!black}, scale=1.5]
		\begin{pgfonlayer}{nodelayer}
		\node [style=vertex, label=above right:$c_2$] (1) at (-3, -1) {};
		\node [style=vertex, label=right:$c_1$] (4) at (-3, -0) {};
		\node [style=vertex, label=above left:$c_3$] (5) at (-2, -1) {};
		\node [style=blossom, label=right:$c$] (11) at (-3, 0.5) {};
	\end{pgfonlayer}
	\begin{pgfonlayer}{edgelayer}
		\draw [branch] (4) to (1);
		\draw [branch] (5) to (1);
		\draw[gem] (11) to (4);
	\clip[use as bounding box] (-3.5,-1.5) rectangle (-1.5,1);
	\end{pgfonlayer}
\end{tikzpicture}\quad
\begin{tikzpicture}[vertex/.style={circle, inner sep=2pt, fill=black}, blossom/.style={diamond, inner sep=2pt, fill=green!70!black}, branch/.style={very thick, brown},gem/.style={green!70!black}, scale=1.5]
		\begin{pgfonlayer}{nodelayer}
		\node [style=vertex, label=above right:$c_2$] (1) at (-3, -1) {};
		\node [style=vertex, label=right:$c_1$] (4) at (-3, -0) {};
		\node [style=vertex, label=above left:$c_3$] (5) at (-2, -1) {};
		\node [style=blossom] (11) at (-3, 0.5) {};
	\end{pgfonlayer}
	\begin{pgfonlayer}{edgelayer}
		\draw [branch] (4) to (1);
		\draw [branch] (5) to (1);
		\draw[gem] (11) to (4);
		\draw[gem, looseness=1.2, out=90, dashed, thick, ->] (11) to (5);
		\clip[use as bounding box] (-3.5,-1.5) rectangle (-1.5,1);
		\end{pgfonlayer}
\end{tikzpicture}\quad
\begin{tikzpicture}[vertex/.style={circle, inner sep=2pt, fill=black}, blossom/.style={diamond, inner sep=2pt, fill=green!70!black}, branch/.style={very thick, brown},gem/.style={green!70!black}, scale=1.5]
		\begin{pgfonlayer}{nodelayer}
		\node [style=vertex] (1) at (-3, -1) {};
		\node [style=vertex] (4) at (-3, -0) {};
		\node [style=vertex] (5) at (-2, -1) {};	\end{pgfonlayer}
	\begin{pgfonlayer}{edgelayer}
		\draw [branch] (4) to (1);
		\draw [branch] (5) to (1);
		\draw[gem, thick] (4) to (5);
		\clip[use as bounding box] (-3.5,-1.5) rectangle (-1.5,1);
		\end{pgfonlayer}
\end{tikzpicture}
\caption{\label{fig:closure}A closure operation performed within a blossoming tree. In the counterclockwise contour of the infinite face, after the corner $c$ of a blossom, we encounter three corners in a row belonging to non-blossoms, namely, $c_1, c_2, c_3$. We bring the blossom to $c_3$, creating a new triangular face.}	
\end{figure}

\begin{proposition} The map $\Xi:\BT_n\times\{\pm1\}\to\STr_n$ of Construction~\ref{PS construction} is $2n$-to-one. Non-blossom vertices of $t\in\BT_n$ naturally correspond to vertices of $\Xi(t,\epsilon)$ that are \emph{not} endpoints of the root edge. Edges between non-blossom vertices of $t$ form a spanning tree of the map obtained from $\Xi(t,\epsilon)$ by erasing the endpoints of its root edge (and any edges involving them). Blossoms of $t$ correspond to edges of $\Xi(t)$ other than the root edge and the edges of the spanning tree just mentioned. The value of $\epsilon$ corresponds to the orientation of the root edge of $\Xi(t,\epsilon)$.
\end{proposition}

\section{Growing $2p$-angulations}\label{Growing $2p$-angulations}

Our main concern in this article is constructing a uniform random map of size $n+1$ in a certain class from a uniform random map of size $n$ in the same class via a simple local (random) manipulation. What kind of manipulation we allow ourselves to make to ``grow'' our uniform maps is, to some extent, an arbitrary choice; in our case, we shall grow maps by adding faces, one at a time. We will presently give a more formal description of what we mean by adding a face: the fact that this is a very simple, natural, local operation will hopefully be apparent; as explained in the introduction, our motivation for considering this operation in particular comes from the potential applications to mixing time questions, which we shall explore in a subsequent paper.

Note that some previous work in the same direction does exist, but differs from ours in a few respects. Bettinelli \cite{Bet14} studies bijections that yield a procedure to grow uniform random quadrangulations; however, his allowed growth operations are more general than ours, as they involve ``adding a face'' but also performing certain types of surgery on the map, cutting it along a geodesic and re-gluing it in a different way. These surgeries make Bettinelli's growth operations non-local, which makes them ill suited to certain applications, such as the one of \cite{Car20}, where a different (but still explicit) growth scheme for quadrangulations was provided.
\medskip
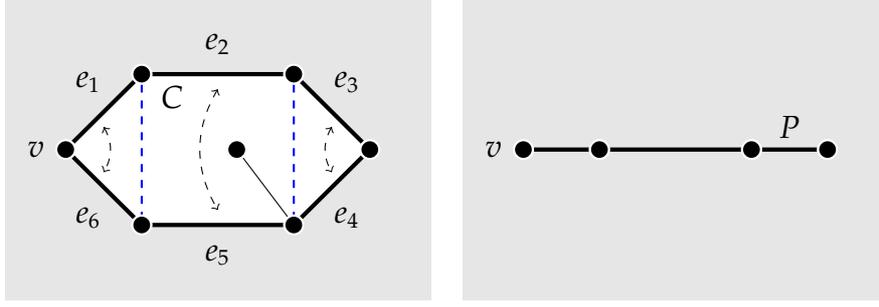
\begin{figure}
\centering
\begin{tikzpicture}[vertex/.style={fill=black, draw=white, thick, circle, inner sep=2.5pt}]
	\begin{pgfonlayer}{nodelayer}
		\node [style=vertex, label=left:$v$] (0) at (-2, -0) {};
		\node [style=vertex] (1) at (-1, 1) {};
		\node [style=vertex] (2) at (1, 1) {};
		\node [style=vertex] (3) at (2, -0) {};
		\node [style=vertex] (4) at (-1, -1) {};
		\node [style=vertex] (5) at (1, -1) {};
		\node [style=vertex] (6) at (0.25, -0) {};
		\node at (-1.7,0.9) {$e_1$};
		\node at (-1.7,-0.9) {$e_6$};
		\node at (0,1.4) {$e_2$};
		\node at (0,-1.4) {$e_5$};
		\node at (1.7,0.9) {$e_3$};
		\node at (1.7,-0.9) {$e_4$};
		\node at (-0.6,0.7) {$C$};
	\end{pgfonlayer}
	\begin{pgfonlayer}{edgelayer}
	\fill [gray!20] (-2.8,-2) rectangle (2.8,2);
	\fill [white] (0.center) to (4.center) to (5.center) to (3.center) to (2.center) to (1.center) to (0.center);
		\draw[ultra thick] (0) to (4);
		\draw[ultra thick] (4) to (5);
		\draw[ultra thick] (5) to (3);
		\draw[ultra thick] (3) to (2);
		\draw[ultra thick] (0) to (1);
		\draw[ultra thick] (1) to (2);
		\draw (6) to (5);
		\draw[<->, dashed, bend left] (-1.5,0.3) to (-1.5,-0.3);
		\draw[<->, dashed, bend right] (1.5,0.3) to (1.5,-0.3);
		\draw[<->, dashed, bend right] (0,0.8) to (0,-0.8);
		\draw[dashed, blue, thick] (1)--(4) (2)--(5);
	\end{pgfonlayer}
\end{tikzpicture}\quad
\begin{tikzpicture}[vertex/.style={fill=black, draw=white, thick, circle, inner sep=2.5pt}]
	\begin{pgfonlayer}{nodelayer}
		\node [style=vertex, label=left:$v$] (0) at (-2, -0) {};
		\node [style=vertex] (1) at (-1, 0) {};
		\node [style=vertex] (2) at (1, 0) {};
		\node [style=vertex] (3) at (2, -0) {};
		\node at (1.5,0.3) {$P$};
	\end{pgfonlayer}
	\begin{pgfonlayer}{edgelayer}
	\fill [gray!20] (-2.8,-2) rectangle (2.8,2);
		\draw[ultra thick] (0) to (1) to (2) to (3);
	\end{pgfonlayer}
\end{tikzpicture}
\caption{\label{fig:face collapse}Collapsing a face with outer boundary $C$ of length $4$ in an 8-angulation $M$.}	
\end{figure}

Given a $2p$-angulation $M\in\M^{2p}_{n+1}$, let $C$ be a simple cycle of $M$ that encloses exactly one face and let $v$ be a vertex along the cycle $C$. Since all $2p$-angulations are bipartite, the cycle $C$ has even length; since it encloses exactly one face, its length is $2k\leq 2p$. Thus, we can express $C$ as a sequence of edges $e_1,\ldots,e_{2k}$ such that $v$ is the common endpoint of $e_1$ and $e_{2k}$. Define the $2p$-angulation $\coll(M,C,v)\in\M^{2p}_{n}$ as the one obtained by erasing all vertices and edges in the interior of $C$ and pairwise identifying the edges of $C$: $e_1$ with $e_{2k}$, $e_2$ with $e_{2k-1}$, and so on, in such a way as to turn the cycle into a path of length $P$ started at $v$ (Figure~\ref{fig:face collapse}). We say that $\coll(M,C,v)$ is obtained form $M$ \emph{by collapsing a face} (in this expression, we are referring to the face enclosed by $C$). Vice-versa, we say that $M$ can be obtained by $\coll(M,C,v)$ by \emph{growing a face}; more specifically, by \emph{growing a face at $P$}.

Note that not all faces can be collapsed and that, given a face that can be, it can be collapsed in $k$ ways, where $2k$ is the length of its outer boundary, depending on the choice of the vertex $v$ (choosing the vertex opposite $v$ in the cycle $C$ yields the same result). Given a $2p$-angulation $m\in\M^{2p}_{n}$ and a simple path $P$, there is a unique way to grow a face at $P$ if the length of the path is $p$ and it does not contain the root edge. If $P$ contains the root edge, we might grow a face at $P$ to the left or to the right of the root edge; additionally, if the length of the path is less than $k$, we get a choice of how to arrange edges and vertices in the interior of the newly grown face.

The notation introduced here is slightly different than the one used in \cite{Car20} in the case of quadrangulations, where we allowed the collapse of faces whose outer boundary has other faces in its interior (see Figure~10 of \cite{Car20}). Such operations were not actually used in the growth scheme we constructed, and they are markedly easier to deal with in the case $p=2$, so we shall avoid them altogether. 

What we wish to do is show that it is possible to grow uniform quadrangulation by the successive growing of faces at random locations $P$. To this end, we define the notion of a growth scheme:
\begin{definition}\label{def:growth scheme for 2p-angulations}
A \emph{growth scheme} for $2p$-angulations is a collection of functions $f_n:\M^{2p}_{n+1}\times\M^{2p}_{n}\to[0,1]$, for $n\geq 1$, such that
\begin{itemize}
\item[i)] if $f_n(M,m)>0$, then $m=\coll(M,C,v)$ for some $C,v$.
\item[ii)] for all $M\in \M^{2p}_{n+1}$, $\sum_{m\in \M^{2p}_{n}}f_n(M,m)=1$;	
\item[iii)] for all $m\in \M^{2p}_{n}$, $\sum_{M\in \M^{2p}_{n+1}}f_n(M,m)=\frac{|\M^{2p}_{n+1}|}{|\M^{2p}_{n}|}$.
\end{itemize}

Note that a growth scheme allows us to couple a uniform $2p$-angulation with $n$ faces $U_n$ with a uniform $2p$-angulation with $n+1$ faces $U_{n+1}$ in such a way that $U_n=\coll(U_{n+1},C,v)$, where $C$ is a random simple cycle of $U_{n+1}$ enclosing a single face and $v$ a random vertex of $U_{n+1}$ (whose distributions conditional on $U_{n+1}$ can be determined from the function $f_n$). We can define a random variable $X$ taking values in $\M_{n+1}^{2p}$ on the same probability space as $U_n$, as follows. Conditionally on $U_n=m$, set $X=M$ with probability proportional to $f_n(M,m)$. Property i) ensures that $U_n$ is of the form $\coll(X,C,v)$, and properties ii) and iii) that the distribution of $X$ is the same as the distribution of $U_{n+1}$.

In \cite{Car20}, a growth scheme was provided for quadrangulations, i.e.~for $p=2$; note that the probabilities $f_n(\cdot, \cdot)$ were computed explicitly in terms of the structure of the tree underlying the quadrangulation; this was thanks to the fact that growth schemes for plane trees are easy to determine explicitly. In this paper, we will show the existence of a growth scheme for $2p$-angulations, but not give explicit expressions for the functions $f_n(\cdot, \cdot)$.

The rest of this section will be devoted to proving Theorem~\ref{2p-angulations growth}.

This can be done by using the fact, proven by Luczak and Winkler in \cite{LW04}, that there exists a growth scheme for $d$-ary trees, by which we mean a sequence of functions $g_n:\T^d_{n+1}\times\CT^d_{n}\to[0,1]$ satisfying properties ii) and iii) from Definition~\ref{def:growth scheme for 2p-angulations}, where occurrences of $\M^{2p}_{k}$ are replaced by $\CT^{d}_{k}$, and the property i') that, if $g_n(T,t)>0$, then $T=\grow(t,v)$ for some leaf $v$ of $t$. For ease of reference, we include here the statement of Theorem 4.1 of \cite{LW04}:

\begin{theorem}[Luczak, Winkler]\label{d-ary trees growth}For each $d\geq 1$, there exists a growth scheme $(g_n)_{n\geq 1}$ for complete $d$-ary trees.	
\end{theorem}

The ways in which our notation differs from the one of \cite{LW04} are very minor: our explicit definition of a growth scheme, which we have adopted in order to maintain the same notation as our previous works \cite{Car20, CS19}, is equivalent to their notion of a \emph{building scheme}; as per Remark~\ref{rem:completeness of trees}, our use of the complete $d$-ary trees of Definition~\ref{def:complete d-ary tree} and their ``growth'' operation does entirely correspond to the subtrees of $\mathbb{T}^d$ and their containment relation as defined in \cite{LW04}.

A first step in order to obtain Theorem~\ref{2p-angulations growth} from Theorem~\ref{d-ary trees growth} is to recast Theorem~\ref{d-ary trees growth} in terms of $p$-mobiles. Thanks to Remark~\ref{rem:mobile to tree} and the bijection $\Phi$, this is rather immediate:

\begin{cor}\label{p-mobiles growth}
Given a growth scheme $(g_n)_{n\geq 1}$ for $p$-ary trees, the functions $(\tau,\tau')\mapsto g_n(\Phi(\tau),\Phi(\tau'))$ yield a growth scheme for unlabelled $p$-mobiles.	
\end{cor}

\begin{proof}
This follows immediately from the fact that $\Phi$ is a bijection and that $\tau$ is obtained by growing $\tau'$ at a corner $c$ if and only if $\Phi(\tau)$ is obtained by growing $\Phi(\tau')$ at a leaf.	
\end{proof}

What remains to do is to establish a link between growing $p$-mobiles and growing $2p$-angulations via the BDFG construction:
\begin{lemma}\label{lemma:growing mobile to growing map}
Let $\tau$ be a labelled $p$-mobile in $\LMob^p_n$, let $c$ be a corner of $\tau$ adjacent to a white vertex and let $s\in S_p$; let $\epsilon=\pm1$. There exists a simple cycle $K$ and a vertex $w$ of the pointed $2p$-angulation $\Psi(\grow(\tau,c),\epsilon)$ such that $\coll(\Psi(\grow(\tau,c,s),\epsilon),K,w)=\Psi(\tau,\epsilon)$, where $\Psi$ is the bijection from Construction~\ref{BDFG}.
\end{lemma}

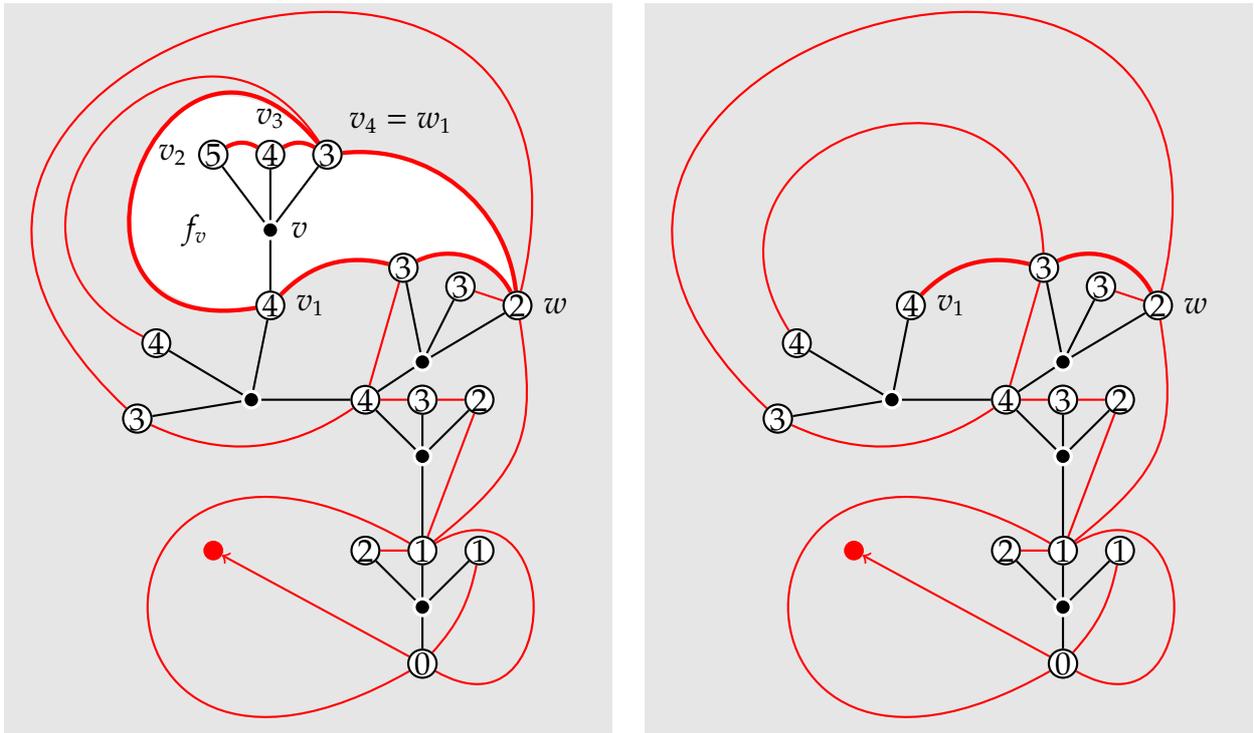
\begin{figure}\centering
\begin{tikzpicture}[vertex/.style={circle, fill=black, inner sep=2.2pt, draw=white, very thick}, gem/.style={circle, fill=white, inner sep=0pt, draw=black, thick}, simple/.style={ultra thick, red}, mobile/.style={thick}, map/.style={red, thick}, scale=1]
\clip[use as bounding box] (-3.5,-5.7) rectangle (4.5,4);
	\begin{pgfonlayer}{nodelayer}
		\node [style=vertex, label=right:$v$] (0) at (0, 1) {};
		\node [style=gem, label=right:$v_1$] (1) at (0, -0) {4};
		\node [style=gem, label=left:$v_2$] (2) at (-0.75, 2) {5};
		\node [style=gem, label=above:$v_3$] (3) at (0, 2) {4};
		\node [style=gem, label=above right:{$v_4=w_1$}] (4) at (0.75, 2) {3};
		\node [style=gem] (5) at (1.75, 0.5) {3};
		\node [style=gem, label=right:$w$] (6) at (3.25, -0) {2};
		\node at (-1,1) {$f_v$};
		\node [style=vertex] (7) at (-0.25, -1.25) {};
		\node [style=gem] (8) at (-1.5, -0.5) {4};
		\node [style=gem] (9) at (2, -1.25) {3};
		\node [style=vertex] (10) at (2, -0.75) {};
		\node [style=gem] (11) at (2.5, 0.25) {3};
		\node [style=gem] (12) at (1.25, -1.25) {4};
		\node [style=gem] (13) at (2.75, -1.25) {2};
		\node [style=vertex] (14) at (2, -2) {};
		\node [style=gem] (15) at (2, -3.25) {1};
		\node [style=gem] (16) at (-1.75, -1.5) {3};
		\node [style=vertex] (17) at (2, -4) {};
		\node [style=gem] (18) at (1.25, -3.25) {2};
		\node [style=gem] (19) at (2.75, -3.25) {1};
		\node [style=gem] (20) at (2, -4.75) {0};
		\node [style=vertex, red] (21) at (-0.75, -3.25) {};
	\end{pgfonlayer}
	\begin{pgfonlayer}{edgelayer}
	\fill[gray!20] (-3.5,-5.7) rectangle (4.5,4);
	\fill[white] (1.center) [bend left=120, looseness=4.2] to (4.center) [bend left=45, looseness=1.10] to (6.center) [bend right=45, looseness=1.20] to (5.center) [bend right, looseness=0.90] to (1.center);
		\draw [mobile] (2) to (0);
		\draw [mobile] (0) to (1);
		\draw [mobile] (3) to (0);
		\draw [mobile] (4) to (0);
		\draw [style=simple, bend left=120, looseness=3.50] (1) to (4);
		\draw [style=simple, bend left, looseness=1.00] (1) to (5);
		\draw [style=simple, bend left=45, looseness=1.00] (5) to (6);
		\draw [style=simple, bend left=45, looseness=1.00] (4) to (6);
		\draw [style=simple, bend left, looseness=1.00] (2) to (3);
		\draw [style=simple, bend left, looseness=1.00] (3) to (4);
		\draw [mobile] (8) to (7);
		\draw [mobile] (1) to (7);
		\draw [mobile] (7) to (12);
		\draw [mobile] (12) to (14);
		\draw [mobile] (9) to (14);
		\draw [mobile] (14) to (15);
		\draw [mobile] (13) to (14);
		\draw [mobile] (5) to (10);
		\draw [mobile] (11) to (10);
		\draw [mobile] (6) to (10);
		\draw [mobile] (16) to (7);
		\draw [style=map, bend left=105, looseness=2.25] (8) to (4);
		\draw [map] (12) to (9);
		\draw [map] (9) to (13);
		\draw [mobile] (12) to (10);
		\draw [map, bend left=120, looseness=3.25] (16) to (6);
		\draw [mobile] (18) to (17);
		\draw [mobile] (20) to (17);
		\draw [mobile] (15) to (17);
		\draw [mobile] (19) to (17);
		\draw[map, ->] (20) to (21);
		\draw [map] (18) to (15);
		\draw [map, bend right=120, looseness=8](15) to (20);
		\draw [map, bend left=15, looseness=1.00] (19) to (20);
		\draw [map] (13) to (15);
		\draw[map, bend left=120, looseness=3] (15) to (20);
		\draw [map, bend left, looseness=1.25] (6) to (15);
		\draw[map, bend left] (12) to (16);
		\draw[map] (11) to (6);
		\draw[map] (12) to (5);
	\end{pgfonlayer}
	\end{tikzpicture}\quad
	\begin{tikzpicture}[vertex/.style={circle, fill=black, inner sep=2.2pt, draw=white, very thick}, gem/.style={circle, fill=white, inner sep=0pt, draw=black, thick}, simple/.style={ultra thick, red}, mobile/.style={thick}, map/.style={red, thick}, scale=1]
	\clip[use as bounding box] (-3.5,-5.7) rectangle (4.5,4);
	\begin{pgfonlayer}{nodelayer}
		\node [style=gem, label=right:$v_1$] (1) at (0, -0) {4};
		\node [style=gem] (5) at (1.75, 0.5) {3};
		\node [style=gem, label=right:$w$] (6) at (3.25, -0) {2};
		\node [style=vertex] (7) at (-0.25, -1.25) {};
		\node [style=gem] (8) at (-1.5, -0.5) {4};
		\node [style=gem] (9) at (2, -1.25) {3};
		\node [style=vertex] (10) at (2, -0.75) {};
		\node [style=gem] (11) at (2.5, 0.25) {3};
		\node [style=gem] (12) at (1.25, -1.25) {4};
		\node [style=gem] (13) at (2.75, -1.25) {2};
		\node [style=vertex] (14) at (2, -2) {};
		\node [style=gem] (15) at (2, -3.25) {1};
		\node [style=gem] (16) at (-1.75, -1.5) {3};
		\node [style=vertex] (17) at (2, -4) {};
		\node [style=gem] (18) at (1.25, -3.25) {2};
		\node [style=gem] (19) at (2.75, -3.25) {1};
		\node [style=gem] (20) at (2, -4.75) {0};
		\node [style=vertex, red] (21) at (-0.75, -3.25) {};
	\end{pgfonlayer}
	\begin{pgfonlayer}{edgelayer}
	\fill[gray!20] (-3.5,-5.7) rectangle (4.5,4);
		\draw [style=simple, bend left, looseness=1.00] (1) to (5);
		\draw [style=simple, bend left=45, looseness=1.00] (5) to (6);
		\draw [mobile] (8) to (7);
		\draw [mobile] (1) to (7);
		\draw [mobile] (7) to (12);
		\draw [mobile] (12) to (14);
		\draw [mobile] (9) to (14);
		\draw [mobile] (14) to (15);
		\draw [mobile] (13) to (14);
		\draw [mobile] (5) to (10);
		\draw [mobile] (11) to (10);
		\draw [mobile] (6) to (10);
		\draw [mobile] (16) to (7);
		\draw [style=map, bend left=105, looseness=2.25] (8) to (5);
		\draw [map] (12) to (9);
		\draw [map] (9) to (13);
		\draw [mobile] (12) to (10);
		\draw [map, bend left=120, looseness=3.25] (16) to (6);
		\draw [mobile] (18) to (17);
		\draw [mobile] (20) to (17);
		\draw [mobile] (15) to (17);
		\draw [mobile] (19) to (17);
		\draw[map, ->] (20) to (21);
		\draw [map] (18) to (15);
		\draw [map, bend right=120, looseness=8](15) to (20);
		\draw [map, bend left=15, looseness=1.00] (19) to (20);
		\draw [map] (13) to (15);
		\draw[map, bend left=120, looseness=3] (15) to (20);
		\draw [map, bend left, looseness=1.25] (6) to (15);
		\draw[map, bend left] (12) to (16);
		\draw[map] (11) to (6);
		\draw[map] (12) to (5);
	\end{pgfonlayer}
	\end{tikzpicture}
	\caption{\label{fig:grow leaf/face}On the left, an 8-angulation $m=\Psi(\grow(\tau,c),\epsilon)$ for some $(\tau,\epsilon)\in\LMob_4^4\times\{\pm 1\}$ and a corner $c$ of $\tau$. On the right, the 8-angulation $\Psi(\tau,\epsilon)$, which is obtained from $m$ by collapsing the face $f_v$, whose outer boundary becomes the thick red path of length two from $v_1$ to $w$ in the picture.}	
\end{figure}

\begin{proof}
Let $v$ be the new black vertex in $\grow(\tau,c,s)$ and $v_1,\ldots,v_p$ be its neighbours, ordered clockwise starting with the parent of $v$.

In $\Psi(\grow(\tau,c,s),\epsilon)$, a map which we call $m$, the black vertex $v$ corresponds to a face $f_v$ whose external boundary is a simple cycle $C$.

We shall show that $C$ contains the vertex $v_1$, that it has only the face $f_v$ in its interior, and that the statement of the theorem is satisfied if we take $K=C$, $w=v_1$.

The cycle $C$ is obtained as follows in the context of the construction $\Psi$: let $a$ be the label of $v_1$ after the first step of Construction~\ref{BDFG}; start with the corner $c_l$ of $v_1$ immediately to the left of the edge $(v_1,v)$, and follow the map edge drawn towards the corner of the first vertex $w_1$ you meet in the contour that is labelled $a-1$; if $w_1\notin\{v_2,\ldots,v_p\}$, stop; otherwise, follow the map edge drawn from that corner to the next corner labelled $a-2$, and so on. This procedure yields a simple path $P_l$ of map edges which starts at $v_1$ and ends at a vertex $w\notin\{v_1,\ldots,v_p\}$, labelled $a-k$ for some $k>0$. Consider then the simple path $P_r$ obtained by following successive map edges drawn from the corner $c_r$ of $v$ immediately to the right of the edge $(v_1,v)$ until they hit $w$. The fact that $w$ will be hit is clear, as the corner of $w$ hit by the first path is the first one labelled $a-k$ coming after $c_r$. 

The cycle $C$ obtained by concatenating the two simple paths $P_l$ and $P_r$ does not enclose any black vertices other than $v$ because $v_2,\ldots,v_p$ are leaves of the mobile, so it is the external boundary of the face $f_v$. The $2p$-angulation $\coll(m,C,v_1)$ is obtained by collapsing the face $f_v$, gluing together the simple paths $P_l$ and $P_r$ starting from their common origin $v_1$.

The same $2p$-angulation is obtained as $\Psi(\tau,\epsilon)$, where the construction of path $P_r$ is unaffected and map edges that were hitting corners of $P_l$ in the previous construction now hit the corners with the same label hit by $P_r$ (Figure~\ref{fig:grow leaf/face}).
\end{proof}

\begin{remark}
Note that, given $2p$-angulations $m,M$ such that $M$ is obtained from $m$ by growing a face, it is not necessarily true that $\Psi^{-1}(M)$ can be obtained by growing $\Psi^{-1}(m)$, in the sense of grafting a $p$-mobile of size 1 somewhere. Since our construction of growth schemes for $2p$-angulations will rely on growth schemes for mobiles, what this entails is that there are some allowed growth operations on $2p$-angulations that will never be used.
\end{remark}

Finally, we have
\begin{proposition}\label{2p-angulation growth prop}
Suppose $(g_n)_{n\geq 1}$ is a growth scheme for unlabelled $p$-mobiles. Define
	$$f_n(m,m')=\frac{1}{|V(m)|{2p-1 \choose p}^n}\sum_{v\in V(m)}\sum_{v'\in V(m')}1_{\epsilon=\epsilon'}g_n(\tilde{\tau},\tilde{\tau'}),$$
where $(m,v)=\Psi(\tau,\epsilon)$ and $(m',v')=\Psi(\tau',\epsilon')$, and $\tilde{\tau}, \tilde{\tau'}$ are obtained by forgetting the labels on the black vertices of $\tau, \tau'$. 

The sequence of functions $(f_n)_{n\geq 1}$ is a growth scheme for $2p$-angulations.
\end{proposition}

\begin{proof}
The fact that the functions $(f_n)_{n\geq 1}$ satisfy condition i) of Definition~\ref{def:growth scheme for 2p-angulations} is due to Lemma~\ref{lemma:growing mobile to growing map}. Indeed, if $f_n(m,m')>0$, then there are $v\in m$ and $v'\in m'$ such that the mobile $\tau$ corresponding to $(m,v)$ is obtained by growing the mobile $\tau'$ corresponding to $(m',v')$. By Lemma~\ref{lemma:growing mobile to growing map}, it follows that $m'$ differs from $m$ by collapsing a face.

As for property ii), given $m\in\M^p_{n+1}$, set $C(n)=\frac{1}{|V(m)|{2p-1 \choose p}^n}$ and consider the sum
$$\sum_{m'\in\M^{2p}_n}f_n(m,m')=\sum_{m'\in\M^{2p}_n}C(n)\sum_{v\in V(m)}\sum_{v'\in V(m')}1_{\epsilon=\epsilon'}g_n(\tilde{\tau},\tilde{\tau'})=C(n)\sum_{v\in V(m)}\sum_{(m',v')\in\M^{2p,\bullet}_n}1_{\epsilon=\epsilon'}g_n(\tilde{\tau},\tilde{\tau'});$$
using the BDFG correspondence, we can rewrite the internal sum to obtain
$$C(n)\sum_{v\in V(m)}\sum_{(\tau',\epsilon')\in\LMob^p_n\times \{\pm1\}}1_{\epsilon=\epsilon'}g_n(\tilde{\tau},\tilde{\tau'})=C(n)\sum_{v\in V(m)}\sum_{\tau'\in\LMob^p_n}g_n(\tilde{\tau},\tilde{\tau'})=$$
$$=C(n)\sum_{v\in V(m)}\sum_{\tilde{\tau'}\in\Mob^p_n}{2p-1\choose p}^n g_n(\tilde{\tau},\tilde{\tau'})=C(n){2p-1 \choose p}^n\sum_{v\in V(m)}1=1.$$

Similarly, for iii) one has
$$\sum_{m\in\M^{2p}_{n+1}}f_n(m,m')=C(n)\sum_{v'\in V(m')}\sum_{\tau\in\LMob^p_{n+1}}g_n(\tilde{\tau},\tilde{\tau'})=$$
$$=C(n){2p-1 \choose p}^{n+1}|V(m')|\frac{|\Mob^p_{n+1}|}{|\Mob^p_{n}|}=\frac{|\LMob^p_{n+1}||V(m')|}{|\LMob^p_{n}||V(m)|}=\frac{|\M^{2p}_{n+1}|}{|\M^{2p}_{n}|}.$$
\end{proof}

\end{definition}

\section{Growing simple triangulations}\label{Growing simple triangulations}

In this section, we show the existence of a growth scheme for simple triangulations; we intend to use a similar growth operation to the one introduced for $2p$-angulations in the previous section, but note that some small tweak is necessary: it is of course not possible to collapse a single face of a simple triangulation to obtain a smaller simple triangulation. The operations we shall allow in this case consist in growing (or collapsing) \emph{two faces} at a time.

Given a triangulation $t\in\STr_n$ and an edge $e$ of $t$ that is not the root edge, we let $\coll(t,e)$ be the triangulation obtained as follows: the edge $e$ has endpoints $v,w$ and is adjacent to two triangular faces; we collapse these two faces by eliminating $e$, identifying $v$ with $w$, and ``gluing together'' the remaining pairs of edges that now form two cycles of length two (Figure~\ref{fig: triangulation collapse}). If the root edge is on the boundary of one of the two collapsed faces, we simply maintain its status and orientation in the new triangulation.

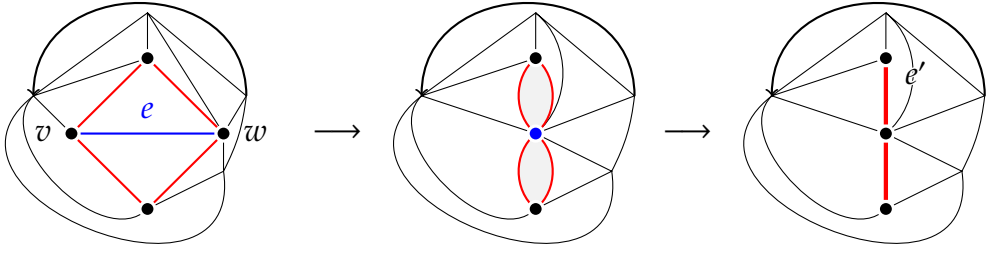
\begin{figure}
\centering
\begin{tikzpicture}[vertex/.style={circle, fill=black, inner sep=2pt, draw=white, very thick}]
\clip[use as bounding box] (-1,-2) rectangle (4,2);
\draw (0,0)--(-0.5,0.5)--(1,1)--(1,1.6)--(-0.5,0.5) (1,1.6)--(2,0)--(2.3,0.5)--(1,1.6) [bend right=90] (-0.5,0.5) to (1,-1);
\draw (1,-1)--(2,-0.5)--(2,0) [bend right=10] (2,-0.5) to (2.3,0.5);
\draw [bend left=90, <-, looseness=1.5, thick] (-0.5,0.5) to (2.3,0.5);
\draw [bend right=120, looseness=2] (-0.5,0.5) to (2,-0.5);
\node[vertex, label=left:$v$] (A) at (0,0) {};
\node[vertex] (B) at (1,-1) {};
\node[vertex, label=right:$w$] (C) at (2,0) {};
\node[vertex] (D) at (1,1) {};

\draw[thick, red] (A)--(B)--(C)--(D)--(A);
\draw[thick, blue] (A)--(C);

\node (e) at (1,0.3) {\textcolor{blue}{\contour{white}{$e$}}};
\node (->) at (3.5,0) {$\longrightarrow$};
\end{tikzpicture}
\begin{tikzpicture}[vertex/.style={circle, fill=black, inner sep=2pt, draw=white, very thick}]
\clip[use as bounding box] (-1,-2) rectangle (3.5,2);
\draw [bend left=90, <-, looseness=1.5, thick] (-0.5,0.5) to (2.3,0.5);
\draw (1,0)--(-0.5,0.5)--(1,1)--(1,1.6)--(-0.5,0.5) [bend left=50] (1,1.6) to (1,0) (1,0)--(2.3,0.5)--(1,1.6) [bend right=90] (-0.5,0.5) to (1,-1);
\draw (1,-1)--(2,-0.5)--(1,0) [bend right=10] (2,-0.5) to (2.3,0.5);
\draw [bend right=120, looseness=2] (-0.5,0.5) to (2,-0.5);
\draw[thick, red, bend left=50, fill=gray!10] (1,-1) to (1,0) to (1,1) to (1,0) to (1,-1);
\node[vertex] (B) at (1,-1) {};
\node[vertex, fill=blue] (AC) at (1,0) {};
\node[vertex] (D) at (1,1) {};
\node (->) at (3,0) {$\longrightarrow$};
\end{tikzpicture}
\begin{tikzpicture}[vertex/.style={circle, fill=black, inner sep=2pt, draw=white, very thick}]
\clip[use as bounding box] (-1,-2) rectangle (2.5,2);
\draw [bend left=90, <-, looseness=1.5, thick] (-0.5,0.5) to (2.3,0.5);
\draw (1,0)--(-0.5,0.5)--(1,1)--(1,1.6)--(-0.5,0.5) [bend left=50] (1,1.6) to (1,0) (1,0)--(2.3,0.5)--(1,1.6) [bend right=90] (-0.5,0.5) to (1,-1);
\draw (1,-1)--(2,-0.5)--(1,0) [bend right=10] (2,-0.5) to (2.3,0.5);
\draw [bend right=120, looseness=2] (-0.5,0.5) to (2,-0.5);
\node[vertex] (B) at (1,-1) {};
\node[vertex] (AC) at (1,0) {};
\node[vertex] (D) at (1,1) {};
\draw[ultra thick, red] (B)--(AC)--(D);
\node (e') at (1.4,0.8) {\contour{white}{$e'$}};
\end{tikzpicture}
\caption{\label{fig: triangulation collapse}On the left, a simple triangulation $T\in\STr_{12}$; collapsing the two faces adjacent to the edge $e$ identifies vertices $v$ and $w$ an yields the triangulation $\coll(T,e)\in\STr_{10}$ depicted on the right. Note that collapsing a pair of faces does not necessarily yield a simple triangulation: for example, the triangulation $\coll(\coll(T,e),e')$ is no longer simple.}
\end{figure}

We say that $\coll(t,e)$ is obtained from $t$ by \emph{collapsing a pair of (adjacent) faces} and that $t$ is obtained from $\coll(t,e)$ by \emph{growing a pair of (adjacent) faces}.

In complete analogy with Definition~\ref{def:growth scheme for 2p-angulations}, we give
\begin{definition}\label{def:growth scheme for triangulations}
	A \emph{growth scheme} for simple triangulations is a collection of functions $f_n:\STr_{n+1}\times\STr_{n}\to[0,1]$, for $n\geq 1$, such that
\begin{itemize}
\item[i)] if $f_n(T,t)>0$, then $t=\coll(T,e)$ for some edge $e$ of $T$;
\item[ii)] for all $T\in \STr_{n+1}$, $\sum_{t\in \STr_{n}}f_n(T,t)=1$;	
\item[ii)] for all $t\in \STr_{n}$, $\sum_{T\in \STr_{n+1}}f_n(T,t)=\frac{|\STr_{n+1}|}{|\STr_{n}|}$.
\end{itemize}
\end{definition}

In order to show the existence of a growth scheme for simple triangulations, we intend to relate the operation of growing/collapsing pairs of faces to the growth operations defined on complete $d$-ary trees in Section~\ref{2p-angulations, mobiles, trees}, and then adapt a growth scheme for $d$-ary trees obtained by Theorem~\ref{d-ary trees growth}.

In order to do this, we first need to bridge the gap between triangulations and $d$-ary trees, which we do via the Poulalhon-Schaeffer construction we recalled in Section~\ref{triangulations, blossoming trees}. But first, we express a general blossoming tree $\tau$ as a pair $(L_\tau,R_\tau)$ of complete 4-ary trees, which will enable us to introduce a new bijection relating such pairs to simple triangulations.

\begin{construction}[\emph{of a pair of complete 4-ary trees from a blossoming tree}]\label{BT to (l,r)}
Given a blossoming tree $t\in\BT_n$, consider the two (multi-type) plane trees $L(t), R(t)$ obtained as follows. The origin $\rho$ of $t$ has one child, which is a non-blossom vertex $v$; $v$ has two blossom neighbours, $\rho$ and $\rho'$. Build $L(t)$ by erasing $\rho, \rho'$ and all descendants of $v$ that come after $\rho'$ in the clockwise contour of $t$; root the ensuing tree in the corner that used to contain $\rho$. Similarly, $R(t)$ is built by erasing $\rho, \rho'$ and all descendants of $v$ that come before $\rho'$ in the clockwise contour of $t$, and rooting the ensuing tree in the corner that used to contain $\rho$. The objects $L(t)$ and $R(t)$ (see Figure~\ref{fig:L(t),R(t)}) are plane trees whose vertices are either blossoms or non-blossoms; their origin is a non-blossom and has no blossom neighbours; all other non-blossoms have exactly two blossom neighbours and all blossoms are leaves. We shall call the set of such plane trees having $n$ non-blossoms $T_n$. Naturally, given $L\in T_a$ and $R\in T_b$, by gluing the two trees together and separating them with two blossoms (one of which becomes the root), one can recover the corresponding blossoming tree in $\BT_{a+b-1}$. The mapping $t\mapsto (L(t),R(t))$ is a natural bijection between $\BT_n$ and $\cup_{k=1}^n T_k\times T_{n+1-k}$.

We now show how $L(t)\in T_a$ can be interpreted as an element of $\CT^4_{a-1}$ (similarly for $R(t)$). Indeed, the sets $(T_k)_{k\geq 1}$ and $(\CT^4_k)_{k\geq 0}$ essentially satisfy the same recursive structure; we will therefore build the desired correspondence inductively.

Suppose $a=1$, in which case $L(t)$ is the tree with a single non-blossom vertex and no blossoms, which is the only element of $T_1$; this corresponds to the ``degenerate'' complete 4-ary tree with only an origin, which is the only element of $\CT^4_0$. Now assume that we have a way to interpret elements of $T_i$, for $1\leq i< k$, as elements of $\CT^4_{i-1}$, and that $L(t)$ is an element of $T_k$, where $k>1$. We can decompose $L(t)$ into 4 multi-type plane trees $t_1, t_2, t_3, t_4$, each of them in $T_i$ for some $i<k$, as follows.

Since $k>1$, $L(t)$ has an origin $w$ and a rightmost child of the origin $z$ (both of them non-blossoms, since $w$ has no blossom neighbours). Let $t_1$ be the multi-type tree obtained by erasing $z$ and all its descendants, which is in $T_{k_1}$ for some $k_1<k$. Now consider the multi-type plane tree $t'$ obtained by erasing $t_1$ from $L(t)$ and rooting the ensuing tree in the corner that used to contain the edge from $w$ to $z$. The tree $t'$ is not quite a blossoming tree, as it's not rooted in a blossom; however, it can be split similarly to how we split $t$ at the very beginning, by considering the multi-type tree to the left of the leftmost blossom child of $z$ (rooted in $z$), the one between the two blossom children of $z$, and the one to the right of the right blossom child of $z$. These three trees are $t_2, t_3, t_4$, and they belong to $T_{k_2}, T_{k_3}, T_{k_4}$ for some $k_2, k_3, k_4<k$ such that $k_1+k_2+k_3+k_4=k+2$.

By assumption, $t_i$ corresponds to a complete 4-ary tree $\tilde{t}_i$ in $\CT^4_{k_i-1}$; we will simply have $L(t)$ correspond to the complete 4-ary tree obtained by taking the 4-ary tree $\tilde{t}_0\in \CT^4_1$ with leaves $v_1, v_2, v_3, v_4$ (in order from left to right) and grafting onto the leaf $v_i$ the 4-ary tree $\tilde{t}_i$, which yields a 4-ary tree in $\CT^4_{1+k_1+k_2+k_3+k_4-4}=\CT^4_{1+k+2-4}=\CT^4_{k-1}$.
\end{construction}

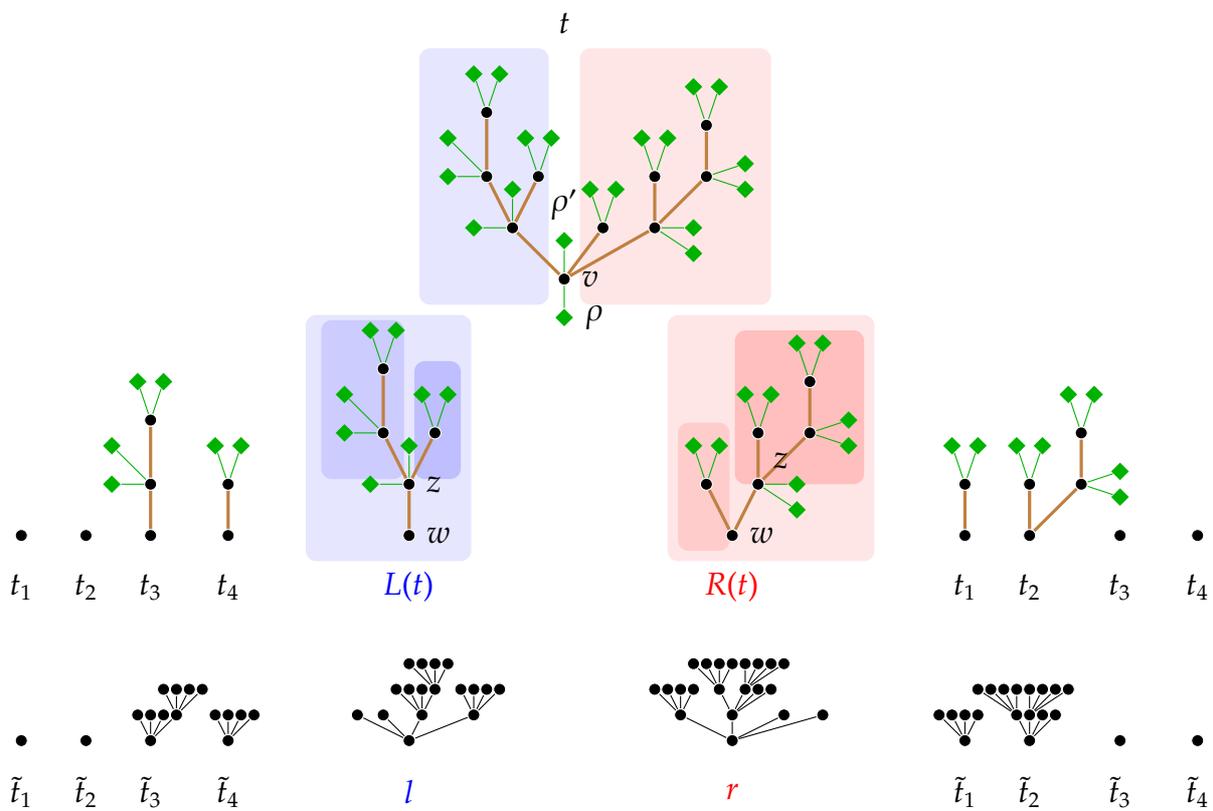
\begin{figure}\centering
\begin{tikzpicture}[vertex/.style={circle, inner sep=1.6pt, draw=white, fill=black}, gem/.style={diamond, inner sep=1.7pt, fill=green!70!black}, branch/.style={very thick, brown},blossom/.style={green!70!black}, simple/.style={}, scale=.68]
		\begin{pgfonlayer}{nodelayer}
		\node [style=vertex, label=right:$v$] (0) at (0, 2) {};
		\node [style=vertex] (1) at (-1, 3) {};
		\node [style=vertex] (2) at (0.75, 3) {};
		\node [style=vertex] (3) at (1.75, 3) {};
		\node [style=vertex] (4) at (-1.5, 4) {};
		\node [style=vertex] (5) at (-0.5, 4) {};
		\node [style=vertex] (6) at (-1.5, 5.25) {};
		\node [style=vertex] (7) at (1.75, 4) {};
		\node [style=vertex] (8) at (2.75, 4) {};
		\node [style=vertex] (9) at (2.75, 5) {};
		\node [style=gem, label=above:$\rho'$] (10) at (0, 2.75) {};
		\node [style=gem, label=right:$\rho$] (11) at (0, 1.25) {};
		\node [style=gem] (12) at (-1.75, 3) {};
		\node [style=gem] (13) at (-1, 3.75) {};
		\node [style=gem] (14) at (-2.25, 4) {};
		\node [style=gem] (15) at (-2.25, 4.75) {};
		\node [style=gem] (16) at (-1.75, 6) {};
		\node [style=gem] (17) at (-1.25, 6) {};
		\node [style=gem] (18) at (-0.75, 4.75) {};
		\node [style=gem] (19) at (-0.25, 4.75) {};
		\node [style=gem] (20) at (0.5, 3.75) {};
		\node [style=gem] (21) at (1, 3.75) {};
		\node [style=gem] (22) at (1.5, 4.75) {};
		\node [style=gem] (23) at (2, 4.75) {};
		\node [style=gem] (24) at (2.5, 3) {};
		\node [style=gem] (25) at (2.5, 2.5) {};
		\node [style=gem] (26) at (3.5, 4.25) {};
		\node [style=gem] (27) at (3.5, 3.75) {};
		\node [style=gem] (28) at (2.5, 5.75) {};
		\node [style=gem] (29) at (3, 5.75) {};
		\node [style=gem] (30) at (-3.75, 1) {};
		\node [style=gem] (31) at (-2.25, -0.25) {};
		\node [style=vertex, label=right:$w$] (32) at (-3, -3) {};
		\node [style=gem] (33) at (-3.75, -2) {};
		\node [style=vertex, label=right:$z$] (34) at (-3, -2) {};
		\node [style=vertex] (35) at (-2.5, -1) {};
		\node [style=gem] (36) at (-3, -1.25) {};
		\node [style=gem] (37) at (-4.25, -0.25) {};
		\node [style=gem] (38) at (-4.25, -1) {};
		\node [style=vertex] (39) at (-3.5, 0.25) {};
		\node [style=gem] (40) at (-3.25, 1) {};
		\node [style=gem] (41) at (-2.75, -0.25) {};
		\node [style=vertex] (42) at (-3.5, -1) {};
		\node [style=gem] (43) at (2.5, -1.25) {};
		\node [style=gem] (44) at (5, 0.75) {};
		\node [style=vertex] (45) at (4.75, -0) {};
		\node [style=vertex, label=above right:$z$] (46) at (3.75, -2) {};
		\node [style=vertex] (47) at (4.75, -1) {};
		\node [style=gem] (48) at (4.5, 0.75) {};
		\node [style=vertex] (49) at (3.75, -1) {};
		\node [style=vertex, label=right:$w$] (50) at (3.25, -3) {};
		\node [style=gem] (51) at (4.5, -2.5) {};
		\node [style=vertex] (52) at (2.75, -2) {};
		\node [style=gem] (53) at (5.5, -1.25) {};
		\node [style=gem] (54) at (3, -1.25) {};
		\node [style=gem] (55) at (4, -0.25) {};
		\node [style=gem] (56) at (5.5, -0.75) {};
		\node [style=gem] (57) at (3.5, -0.25) {};
		\node [style=gem] (58) at (4.5, -2) {};
		\node [style=gem] (59) at (7.5, -1.25) {};
		\node [style=gem] (60) at (8, -1.25) {};
		\node [style=vertex] (61) at (7.75, -2) {};
		\node [style=vertex] (62) at (10, -1) {};
		\node [style=vertex] (63) at (10, -2) {};
		\node [style=gem] (64) at (8.75, -1.25) {};
		\node [style=gem] (65) at (10.75, -2.25) {};
		\node [style=gem] (66) at (10.25, -0.25) {};
		\node [style=gem] (67) at (10.75, -1.75) {};
		\node [style=vertex] (68) at (9, -3) {};
		\node [style=gem] (69) at (9.75, -0.25) {};
		\node [style=vertex] (70) at (9, -2) {};
		\node [style=gem] (71) at (9.25, -1.25) {};
		\node [style=vertex] (72) at (10.75, -3) {};
		\node [style=vertex] (73) at (12.25, -3) {};
		\node [style=vertex] (74) at (-8, -3) {};
		\node [style=gem] (75) at (-8.25, -0) {};
		\node [style=gem] (76) at (-8.75, -2) {};
		\node [style=gem] (77) at (-8.75, -1.25) {};
		\node [style=vertex] (78) at (-8, -2) {};
		\node [style=gem] (79) at (-7.75, -0) {};
		\node [style=vertex] (80) at (-8, -0.75) {};
		\node [style=vertex] (81) at (-6.5, -2) {};
		\node [style=vertex] (82) at (-6.5, -3) {};
		\node [style=gem] (83) at (-6.75, -1.25) {};
		\node [style=gem] (84) at (-6.25, -1.25) {};
		\node [style=vertex] (85) at (-10.5, -3) {};
		\node [style=vertex] (86) at (-9.25, -3) {};
		\node [style=vertex] (87) at (-6.5, -7) {};
		\node [style=vertex] (88) at (-6.5, -6.5) {};
		\node [style=vertex] (89) at (-6.25, -6.5) {};
		\node [style=vertex] (90) at (-6, -6.5) {};
		\node [style=vertex] (91) at (-6.75, -6.5) {};
		\node [style=vertex] (92) at (-8, -6.5) {};
		\node [style=vertex] (93) at (-8, -7) {};
		\node [style=vertex] (94) at (-8.25, -6.5) {};
		\node [style=vertex] (95) at (-7.75, -6.5) {};
		\node [style=vertex] (96) at (-7.5, -6.5) {};
		\node [style=vertex] (97) at (-7.75, -6) {};
		\node [style=vertex] (98) at (-7.5, -6) {};
		\node [style=vertex] (99) at (-7.25, -6) {};
		\node [style=vertex] (100) at (-7, -6) {};
		\node [style=vertex] (101) at (-9.25, -7) {};
		\node [style=vertex] (102) at (-10.5, -7) {};
		\node [style=vertex] (103) at (-3, -7) {};
		\node [style=vertex] (104) at (-4, -6.5) {};
		\node [style=vertex] (105) at (-3.5, -6.5) {};
		\node [style=vertex] (106) at (-2.75, -6.5) {};
		\node [style=vertex] (107) at (-1.75, -6.5) {};
		\node [style=vertex] (108) at (-3.25, -6) {};
		\node [style=vertex] (109) at (-3, -6) {};
		\node [style=vertex] (110) at (-2.75, -6) {};
		\node [style=vertex] (111) at (-2.5, -6) {};
		\node [style=vertex] (112) at (-2, -6) {};
		\node [style=vertex] (113) at (-1.75, -6) {};
		\node [style=vertex] (114) at (-1.5, -6) {};
		\node [style=vertex] (115) at (-1.25, -6) {};
		\node [style=vertex] (116) at (-3, -5.5) {};
		\node [style=vertex] (117) at (-2.75, -5.5) {};
		\node [style=vertex] (118) at (-2.5, -5.5) {};
		\node [style=vertex] (119) at (-2.25, -5.5) {};
		\node [style=vertex] (120) at (7.75, -7) {};
		\node [style=vertex] (121) at (9, -7) {};
		\node [style=vertex] (122) at (10.75, -7) {};
		\node [style=vertex] (123) at (12.25, -7) {};
		\node [style=vertex] (124) at (7.75, -3) {};
		\node [style=vertex] (125) at (7.25, -6.5) {};
		\node [style=vertex] (126) at (7.5, -6.5) {};
		\node [style=vertex] (127) at (7.75, -6.5) {};
		\node [style=vertex] (128) at (8, -6.5) {};
		\node [style=vertex] (129) at (8.75, -6.5) {};
		\node [style=vertex] (130) at (9, -6.5) {};
		\node [style=vertex] (131) at (9.25, -6.5) {};
		\node [style=vertex] (132) at (9.5, -6.5) {};
		\node [style=vertex] (133) at (8, -6) {};
		\node [style=vertex] (134) at (8.25, -6) {};
		\node [style=vertex] (135) at (8.5, -6) {};
		\node [style=vertex] (136) at (8.75, -6) {};
		\node [style=vertex] (137) at (9, -6) {};
		\node [style=vertex] (138) at (9.25, -6) {};
		\node [style=vertex] (139) at (9.5, -6) {};
		\node [style=vertex] (140) at (9.75, -6) {};
		\node [style=vertex] (141) at (3.25, -7) {};
		\node [style=vertex] (142) at (2.25, -6.5) {};
		\node [style=vertex] (143) at (3.25, -6.5) {};
		\node [style=vertex] (144) at (4.25, -6.5) {};
		\node [style=vertex] (145) at (5, -6.5) {};
		\node [style=vertex] (146) at (1.75, -6) {};
		\node [style=vertex] (147) at (2, -6) {};
		\node [style=vertex] (148) at (2.25, -6) {};
		\node [style=vertex] (149) at (2.5, -6) {};
		\node [style=vertex] (150) at (3, -6) {};
		\node [style=vertex] (151) at (3.5, -6) {};
		\node [style=vertex] (152) at (3.75, -6) {};
		\node [style=vertex] (153) at (4, -6) {};
		\node [style=vertex] (154) at (2.5, -5.5) {};
		\node [style=vertex] (155) at (2.75, -5.5) {};
		\node [style=vertex] (156) at (3, -5.5) {};
		\node [style=vertex] (157) at (3.25, -5.5) {};
		\node [style=vertex] (158) at (3.5, -5.5) {};
		\node [style=vertex] (159) at (3.75, -5.5) {};
		\node [style=vertex] (160) at (4, -5.5) {};
		\node [style=vertex] (161) at (4.25, -5.5) {};
		\node   (162) at (0, 7) {$t$};
		\node[blue] (163) at (-3, -4) {$L(t)$};
		\node[blue] (163b) at (-3, -8) {$l$};
		\node   (164) at (-9.25, -4) {$t_2$};
		\node   (165) at (-8, -4) {$t_3$};
		\node   (166) at (-6.5, -4) {$t_4$};
		\node   (167) at (-9.25, -8) {$\tilde t_2$};
		\node   (168) at (-8, -8) {$\tilde t_3$};
		\node   (169) at (-6.5, -8) {$\tilde t_4$};
		\node   (170) at (-3, -8) {};
		\node[red] (171) at (3.25, -4) {$R(t)$};
		\node[red] (171b) at (3.25, -8) {$r$};
		\node   (172) at (7.75, -4) {$t_1$};
		\node   (173) at (9, -4) {$t_2$};
		\node   (174) at (10.75, -4) {$t_3$};
		\node   (175) at (12.25, -4) {$t_4$};
		\node   (176) at (3.25, -8) {};
		\node   (177) at (7.75, -8) {$\tilde t_1$};
		\node   (178) at (9, -8) {$\tilde t_2$};
		\node   (179) at (10.75, -8) {$\tilde t_3$};
		\node   (180) at (12.25, -8) {$\tilde t_4$};
		\node   (181) at (-10.5, -4) {$t_1$};
		\node   (182) at (-10.5, -8) {$\tilde t_1$};
	\end{pgfonlayer}
	\begin{pgfonlayer}{edgelayer}
	\fill[red!10, rounded corners] (0.3,1.5) rectangle (4,6.5);
	\fill[blue!10, rounded corners] (-0.3,1.5) rectangle (-2.8,6.5);
	\fill[red!10, rounded corners] (6,-3.5) rectangle (2,1.3);
	\fill[blue!10, rounded corners] (-5,-3.5) rectangle (-1.8,1.3);
	\fill[red!20, rounded corners] (2.2,-3.3) rectangle (3.2,-0.8);
	\fill[red!25, rounded corners] (3.3,-2) rectangle (5.8,1);
	\fill[blue!20, rounded corners] (-4.7,-1.9) rectangle (-3.1,1.2);
	\fill[blue!25, rounded corners] (-2.9,-1.9) rectangle (-2,0.4);
		\draw [style=branch] (1) to (0);
		\draw [style=branch] (2) to (0);
		\draw [style=branch] (3) to (0);
		\draw [style=branch] (7) to (3);
		\draw [style=branch] (8) to (3);
		\draw [style=branch] (9) to (8);
		\draw [style=branch] (5) to (1);
		\draw [style=branch] (4) to (1);
		\draw [style=branch] (6) to (4);
		\draw [style=blossom] (10) to (0);
		\draw [style=blossom] (0) to (11);
		\draw [style=blossom] (12) to (1);
		\draw [style=blossom] (13) to (1);
		\draw [style=blossom] (15) to (4);
		\draw [style=blossom] (14) to (4);
		\draw [style=blossom] (16) to (6);
		\draw [style=blossom] (17) to (6);
		\draw [style=blossom] (18) to (5);
		\draw [style=blossom] (19) to (5);
		\draw [style=blossom] (3) to (25);
		\draw [style=blossom] (3) to (24);
		\draw [style=blossom] (20) to (2);
		\draw [style=blossom] (21) to (2);
		\draw [style=blossom] (22) to (7);
		\draw [style=blossom] (23) to (7);
		\draw [style=blossom] (8) to (27);
		\draw [style=blossom] (8) to (26);
		\draw [style=blossom] (28) to (9);
		\draw [style=blossom] (29) to (9);
		\draw [style=branch] (34) to (32);
		\draw [style=branch] (35) to (34);
		\draw [style=branch] (42) to (34);
		\draw [style=branch] (39) to (42);
		\draw [style=blossom] (33) to (34);
		\draw [style=blossom] (36) to (34);
		\draw [style=blossom] (37) to (42);
		\draw [style=blossom] (38) to (42);
		\draw [style=blossom] (30) to (39);
		\draw [style=blossom] (40) to (39);
		\draw [style=blossom] (41) to (35);
		\draw [style=blossom] (31) to (35);
		\draw [style=branch] (52) to (50);
		\draw [style=branch] (46) to (50);
		\draw [style=branch] (49) to (46);
		\draw [style=branch] (47) to (46);
		\draw [style=branch] (45) to (47);
		\draw [style=blossom] (46) to (51);
		\draw [style=blossom] (46) to (58);
		\draw [style=blossom] (43) to (52);
		\draw [style=blossom] (54) to (52);
		\draw [style=blossom] (57) to (49);
		\draw [style=blossom] (55) to (49);
		\draw [style=blossom] (47) to (53);
		\draw [style=blossom] (47) to (56);
		\draw [style=blossom] (48) to (45);
		\draw [style=blossom] (44) to (45);
		\draw [style=blossom] (59) to (61);
		\draw [style=blossom] (60) to (61);
		\draw [style=branch] (70) to (68);
		\draw [style=branch] (63) to (68);
		\draw [style=branch] (62) to (63);
		\draw [style=blossom] (64) to (70);
		\draw [style=blossom] (71) to (70);
		\draw [style=blossom] (63) to (65);
		\draw [style=blossom] (63) to (67);
		\draw [style=blossom] (69) to (62);
		\draw [style=blossom] (66) to (62);
		\draw [style=branch] (78) to (74);
		\draw [style=branch] (80) to (78);
		\draw [style=blossom] (77) to (78);
		\draw [style=blossom] (76) to (78);
		\draw [style=blossom] (75) to (80);
		\draw [style=blossom] (79) to (80);
		\draw [style=branch] (81) to (82);
		\draw [style=blossom] (83) to (81);
		\draw [style=blossom] (84) to (81);
		\draw [style=simple] (88) to (87);
		\draw [style=simple] (89) to (87);
		\draw [style=simple] (90) to (87);
		\draw [style=simple] (91) to (87);
		\draw [style=simple] (92) to (93);
		\draw [style=simple] (95) to (93);
		\draw [style=simple] (96) to (93);
		\draw [style=simple] (94) to (93);
		\draw [style=simple] (97) to (96);
		\draw [style=simple] (98) to (96);
		\draw [style=simple] (99) to (96);
		\draw [style=simple] (100) to (96);
		\draw [style=simple] (104) to (103);
		\draw [style=simple] (105) to (103);
		\draw [style=simple] (106) to (103);
		\draw [style=simple] (103) to (107);
		\draw [style=simple] (108) to (106);
		\draw [style=simple] (109) to (106);
		\draw [style=simple] (110) to (106);
		\draw [style=simple] (111) to (106);
		\draw [style=simple] (116) to (111);
		\draw [style=simple] (117) to (111);
		\draw [style=simple] (118) to (111);
		\draw [style=simple] (119) to (111);
		\draw [style=simple] (112) to (107);
		\draw [style=simple] (113) to (107);
		\draw [style=simple] (114) to (107);
		\draw [style=simple] (115) to (107);
		\draw [style=branch] (61) to (124);
		\draw [style=simple] (125) to (120);
		\draw [style=simple] (126) to (120);
		\draw [style=simple] (127) to (120);
		\draw [style=simple] (128) to (120);
		\draw [style=simple] (129) to (121);
		\draw [style=simple] (130) to (121);
		\draw [style=simple] (131) to (121);
		\draw [style=simple] (132) to (121);
		\draw [style=simple] (133) to (129);
		\draw [style=simple] (134) to (129);
		\draw [style=simple] (135) to (129);
		\draw [style=simple] (136) to (129);
		\draw [style=simple] (137) to (130);
		\draw [style=simple] (138) to (130);
		\draw [style=simple] (139) to (130);
		\draw [style=simple] (140) to (130);
		\draw [style=simple] (143) to (141);
		\draw [style=simple] (142) to (141);
		\draw [style=simple] (144) to (141);
		\draw [style=simple] (145) to (141);
		\draw [style=simple] (149) to (142);
		\draw [style=simple] (148) to (142);
		\draw [style=simple] (147) to (142);
		\draw [style=simple] (146) to (142);
		\draw [style=simple] (150) to (143);
		\draw [style=simple] (151) to (143);
		\draw [style=simple] (152) to (143);
		\draw [style=simple] (153) to (143);
		\draw [style=simple] (154) to (150);
		\draw [style=simple] (155) to (150);
		\draw [style=simple] (156) to (150);
		\draw [style=simple] (157) to (150);
		\draw [style=simple] (158) to (151);
		\draw [style=simple] (159) to (151);
		\draw [style=simple] (151) to (160);
		\draw [style=simple] (151) to (161);
	\end{pgfonlayer}
	\end{tikzpicture}
	\caption{\label{fig:L(t),R(t)}Construction~\ref{BT to (l,r)} performed on $t\in\BT_{10}$. Below $t$, we see its left part $L(t)\in T_{5}$ and its right part $R(t)\in T_6$. Below these, the corresponding $4$-ary trees $l\in\CT^4_4$ and $r\in\CT^4_5$, constructing by grafting the four trees $\tilde{t}_1, \tilde{t}_2, \tilde{t}_3, \tilde{t}_4$ onto $\tilde{t}_1\in\CT^4_0$.}
\end{figure}

\begin{proof}[Proof of Proposition~\ref{4-ary trees and triangulations}]
The proposition now comes as an immediate consequence of the construction above: by joining the 4-ary trees $l$ and $r$ with an edge between their origins, oriented from $l$ to $r$, we obtain a tree all of whose vertices have degree $5$ or $1$, endowed with a distinguished oriented edge. This mapping is a bijection between $\cup_{a=0}^{n-1}\CT^4_a\times \CT^4_{n-1-a}$ and $\mathcal{T}_{n-1}$. Construction~\ref{BT to (l,r)} is in turn a bijection between the set $\BT_n$ and the set $\cup_{a=0}^{n-1}\CT^4_a\times \CT^4_{n-1-a}$. Construction~\ref{PS construction} gives a $2n$-to-1 map from $\BT_n\times\{\pm 1\}$ to $\STr_n$, hence a $2n$-to-1 map from $\mathcal{T}_{n-1}\times\{\pm 1\}$ to $\STr_n$, proving Proposition~\ref{4-ary trees and triangulations}.
\end{proof}

\begin{remark}

The map of Proposition~\ref{4-ary trees and triangulations}, which essentially induces a bijection between simple triangulation and a certain set of trees with vertices of degree 5 or 1, will prove very useful in our context, but it is not without some drawbacks. On the one hand, we no longer have to distinguish between blossoms and non-blossoms, and tasks related to enumeration tend to become more straightforward. Some important properties are preserved, and we will see later that growing the tree at a leaf does correspond to growing a pair of faces in the triangulation.

On the other hand, metric information about the triangulation, and even just the blossoming tree, becomes very ``distorted'' when seen through the lens of this 4-ary tree. Note that, because of the special role of the first child of every vertex in the construction, exploring ``on the right'' corresponds to moving faster in the blossoming tree than exploring ``on the left''. Moreover, the construction is non-canonical in many respects, as our choice of having the leftmost child represent the subtree to the left of the rightmost branch is one of several possible.  
\end{remark}

\begin{lemma}\label{growing 4-ary trees to growing blossom}
Construction~\ref{BT to (l,r)} gives a bijection between $\BT_n$ and $\cup_{a=0}^{n-1}\CT^4_a\times \CT^4_{n-1-a}$. Moreover, given $t\in \BT_n$ and its pair of corresponding $4$-ary trees $(l,r)$, we can naturally identify the set of internal vertices of $l$ and $r$ with the set of non-blossom vertices of $t$ other than the one connected to the origin blossom. Finally, given a leaf $v$ of $l$ or $r$, let $t^v$ be the tree in $\BT_{n+1}$ that corresponds to the pair $(\grow(l,v),r)$ or $(l,\grow(r,v))$, respectively. There is a corner of a non-blossom vertex $w$ of $t$ such that $t^v$ is obtained by grafting onto that corner a single edge from $w$ to a non-blossom $w'$, in turn connected to two blossom  children.
\end{lemma}

\begin{proof}
The fact that the given construction is a bijection immediately follows from the fact that, given any pair $(l,r)\in	\cup_{a=0}^{n-1}\CT^4_a\times \CT^4_{n-1-a}$, we can easily obtain a unique tree $t\in\BT_n$ by reversing it. Its inductive structure also makes it clear that internal vertices of $l$ and $r$ correspond to non-blossoms of $L(t)$ and $R(t)$, with the exception of their identified origin which is the non-blossom neighbour of the root of $t$.

Let us now consider what the multi-type tree $\grow(l,v)$ corresponds to. Consider the parent $p(v)$ of $v$ in $l$; this corresponds to a non-blossom vertex $w$ of $L(t)$. If $v$ is the leftmost child of $p(v)$, the fact that it is a leaf implies that $w$ has only one non-blossom child. Changing $l$ into $\grow(l,v)$ corresponds to adding a leftmost non-blossom child to $w$, with no left non-blossom siblings or children (and therefore only two blossom children). If $v$ is the second, third or fourth child of $p(v)$, the fact that it is a leaf tells us that, letting $w'$ be the rightmost non-blossom child of $w$, $w'$ has no non-blossom children to the left of its left blossom child, or no non-blossom children between its two blossom children, or no non-blossom children to the right of the right blossom child. Replacing $l$ by $\grow(l,v)$ corresponds to adding a non-blossom child to $w'$ in the appropriate position with respect to the blossoms; the child added must have two blossom children and no non-blossom neighbours other than $w'$.

The construction for $v\in r$ is completely analogous.
\end{proof}

At this point, given a simple triangulation $t\in\STr_n$ and a sign $\epsilon=\pm1$, we have $2n$ blossoming trees $\tau\in \BT_n$ such that $\Xi(\tau,\epsilon)=t$; each of these can be expressed as a pair consisting of its left and right part, which we can interpret as a pair of 4-ary trees with a total of $n-1$ internal vertices between them. Note that, indeed, we have the following identity involving the number of simple triangulations and the number of $4$-ary trees:
\begin{equation}\label{formula triangulations-trees}\frac{2(4n-3)!}{n!(3n-1)!}=|\STr_n|=\frac{2}{2n}\sum_{k=0}^{n-1}|\CT^4_k||\CT^4_{n-1-k}|=\frac{1}{n}\sum_{k=0}^{n-1}\frac{1}{(3k+1)(3n-3k-2)}{4k \choose k}{4(n-1-k) \choose n-1-k}.\end{equation}

 What we now wish to show is that the operation of growing one of these two $4$-ary trees at a leaf yields a new blossoming tree $\tau'\in\BT_{n+1}$ such that $\Xi(\tau',\epsilon)$ is obtained from $t$ by growing a pair of adjacent faces. More specifically,

\begin{lemma}\label{lemma:growing trees to growing triangulations}
Let $\tau$ be a blossoming tree in $\BT_n$ and let $(l,r)\in \cup_{a=0}^{n-1}\CT^4_a+\CT^4_{n-1-a}$ be obtained from $\tau$ via Construction~\ref{BT to (l,r)}; let $v$ be a leaf of $l$ or $r$ and let $\tau'$ be the blossoming tree in $\BT_{n+1}$ such that Construction~\ref{BT to (l,r)} on $\tau'$ yields $(\grow(l,v), r)$ or $(l, \grow(r,v))$ (depending on which tree $v$ is a leaf of). Given $\epsilon=\pm1$, there exists an edge $e$ of $T=\Xi(\tau',\epsilon)$ such that $\Xi(\tau,\epsilon)=\coll(T,e)$.\end{lemma}

\begin{proof}
By Lemma~\ref{growing 4-ary trees to growing blossom}, the blossoming tree $\tau'$ is obtained from $\tau$ by grafting a non-blossom $w$ with two blossom children onto some corner $c$. Let $\alpha,\beta$ be the left and right blossom child of $w$, respectively, and let $e$ be the edge of the triangulation $T=\Xi(\tau',\epsilon)$ that corresponds to $\beta$ according to Construction~\ref{PS construction}. We intend to show that $\Xi(\tau,\epsilon)=\coll(T,e)$.

Let us consider the following separate cases:
\begin{itemize}
\item[a)] $w$ is not a neighbour of an endpoint of the root edge in 	$T$;
\item[b)] $w$ is a neighbour of a single endpoint of the root edge in $T$;
\item[c)] $w$ is a neighbour of both endpoints of the root edge in $T$.
\end{itemize}

Suppose we are in case a). In the triangulation $T$, let $f_1$ and $f_2$ the two triangles adjacent to the edge $e$; we may assume that the edges of $f_1$ are $(w,p(w))$, by which we mean the edge of $\tau'$ joining $w$ to its parent $p(w)$, the edge $e$ and some edge $e_1$, and that the edges of $f_2$ are the edge created by closing the blossom $\alpha$, $e$ and some edge $e_2$. Let $z$ be the common vertex of $e_1$ and $e_2$.

We know that we can interpret $\tau'$ as a spanning tree of $T$, excluding the endpoints of the root edge. We can therefore consider te set of faces of $T$ that lie between $e_1$ and the shortest path in $\tau'$ between $p(w)$ and $z$ (call this set $S_1$), as well as the set of faces between $e_2$ and the shortest path in $\tau'$ joining the two endpoints of $e_2$ (call this set $S_2$). Note that if $e_i$ belongs to $\tau'$ then $S_i$ is empty.

Assuming $S_1$ is non-empty, it is a triangulation of a $k$-gon for some $k$; the diagonals are edges obtained by closing blossoms and the sides of the $k$-gon are edges originally in $\tau'$, save for $e_1$. By starting with triangles containing only one diagonal, one can perform successive closures on $\tau'$ in order to obtain all faces in $S_1$ (and no other faces of $T$). Similarly, one can perform closures of blossoms to create all faces of $S_2$. After $e_1$ and $e_2$ have been created in this way if they were not originally present in $\tau'$, one can close $\beta$ and then $\alpha$. This sequence of closures yields a partial closure $C'$ of $\tau'$ whose faces are the triangles in $S_1\cup S_2\cup\{f_1,f_2\}$, plus the infinite face (see the upper central part of Figure~\ref{fig:case a}). It is possible to perform the exact same closures (other than the last two) in the same order in $\tau$, yielding a partial closure $C$ that can be obtained from $C'$ by collapsing $f_1$ and $f_2$, identifying $w$ with $z$. Note that unclosed blossoms in $C$ and $C'$ correspond to each other and that the same successive closures can now be performed, so that completing Construction~\ref{PS construction} yields exactly the same triangulation but for $f_1$ and $f_2$, or in other words $\Xi(\tau,\epsilon)=\coll(T,e)$.

Now consider case b). We can again define $z, f_1, e_1, e_2$ as before; note that $z$ is \emph{not} an endpoint of the root edge of $T$, as if only one blossom among $\alpha,\beta$ becomes closable before creating the root edge endpoints, this must be the right blossom $\beta$. We can therefore define the set of faces $S_1$ as before, since $p(w)$ and $z$ are both vertices of $\tau'$. Again, close the necessary blossoms of $\tau'$ to obtain the faces of $S_1$, and then close $\beta$, thus yielding a partial closure $C'$ of $\tau'$. Perform the exact same closures on $\tau$ to obtain a map $C$ that differs from $C'$ by the absence of the face $f_1$ (see Figure~\ref{fig:case b}).

Consider now the following set $S_3$ of faces of $T$. Order the edges incident to $w$ clockwise starting with the edge obtained by closing $\alpha$ and let $\eta$ be the \emph{last} edge according to this ordering, which is either $(w,p(w))$ or is obtained by closing a blossom. Let $u$ be the endpoint of $\eta$ that is not $w$, which cannot be an endpoint of the root edge. Let the faces in $S_3$ be those that lie within the cycle formed by $\eta$ and the shortest path in $\tau'$ leading from $u$ to $w$. 

As before, a sequence of closures can be performed on $C'$ to create precisely the faces in $S_3$, thus yielding a map $C'_1$. Such closures can be performed on $C$ in the same order; what they yield is a map $C_1$ that can be constructed from $C'_1$ by collapsing the face $f_1$, identifying $w$ with $z$ and the blossom $\alpha$ with a blossom $\gamma$ issued from $z$, which must exist because otherwise $\alpha$ would be closable in $C'$. It follows that subsequent closures can be performed in $C_1$ and $C_1'$ at the same time, preserving the property that intermediate maps are obtained one from the other by collapsing $f_1$ and eliminating $\alpha$. The final steps of the construction are also the same, with an additional face involving $e$ being formed in $T$ (see the rightmost part of Figure~\ref{fig:case b}).

Case c) is quite similar. We can define $u, \eta, S_3$ exactly as above; letting $S$ be the set of faces of $T$ containing $w$, set $S_4=S_3\setminus S$. There is a sequence of closures of blossoms in $\tau'$ that builds exactly the faces in $S_4$, and such a sequence can be mirrored on $\tau$ to yield a map which only differs by the absence of $w$ and its two blossoms $\alpha, \beta$. Note that the fact that $w$ is adjacent to both root edge endpoints implies the presence of a blossom $\gamma$ immediately to the right of the edge $(w,p(w))$, which does not become closable until the very last phase of the construction. After the sequence of closures performed on $\tau$, there is a series of unclosed blossoms $b_1,\ldots, b_k$ (lower central part of Figure~\ref{fig:case c}) adjacent to the infinite face, starting with two blossoms attached to $u$ and ending with $\gamma$, which are not closable and will not become so until the root edge endpoints have been created. Any further closure performed on $\tau$ does not involve these blossoms and can be replicated on the partial closure of $\tau'$. After performing all possible closures on $\tau$ (and mirroring them on $\tau'$) we are faced with two partial closures $C$ and $C'$ such that
\begin{itemize}
\item $C$ is a triangulation with a (simple) boundary with one or two blossoms hanging off every boundary vertex; the vertex $u$ is adjacent to two unclosed blossoms;
\item $C'$ is exactly the map $C$, with an added edge $(w,p(w))$ hanging off the boundary and blossoms $\alpha,\beta$ attached to $w$.
\end{itemize}

The triangulation $T$ is obtained from $C'$ by first closing $b_2,\ldots, b_{k-1}$, thus creating the faces in $S_3\cap S$ and then adding the endpoints of the root edge and performing the final closures. The triangulation $t=\Xi(\tau,\epsilon)$ is obtained by drawing the endpoints of the root edge and joining them to unclosed blossoms, including $u_1,\ldots, u_k$. The difference between the two is given purely by the presence of the two faces $f_1,f_2$ adjacent to $e$ in $T$, and by the fact that the blossoms $u_2,\ldots,u_k$ create edges to $w$ rather than the appropriate root edge endpoint. Collapsing $f_1,f_2$ in $T$ by identifying the two endpoints of $e$ yields precisely the triangulation $t$ obtained by performing the last construction phase on the partial closure $C$.

\begin{figure}\centering
\begin{tikzpicture}[vertex/.style={circle, inner sep=2pt, fill=black}, blossom/.style={diamond, inner sep=2pt, fill=green!70!black}, branch/.style={very thick, brown}, gem/.style={green!70!black}, scale=.8]
	\begin{pgfonlayer}{nodelayer}
	\clip[use as bounding box] (-3,-0.5) rectangle (3,5.5);
	\node at (0, 5) {$\tau'$};
		\node [style=vertex, label=left:$w$] (0) at (-2, 3.5) {};
		\node [style=vertex, label=right:$p(w)$] (1) at (-2, 2.5) {};
		\node [style=blossom, label=above:$\alpha$] (2) at (-2.25, 4) {};
		\node [style=blossom, label=above:$\beta$] (3) at (-1.75, 4) {};
		\node [style=vertex] (4) at (-2, 1.5) {};
		\node [style=vertex] (5) at (-1, 1.5) {};
		\node [style=vertex] (6) at (0, 2.5) {};
		\node [style=vertex] (7) at (1, 2.5) {};
		\node [style=vertex] (8) at (2, 3.5) {};
		\node [style=vertex] (9) at (2, 2.5) {};
		\node [style=vertex] (10) at (2, 4.5) {};
		\node [style=vertex] (11) at (-2, 0.5) {};
		\node [style=blossom] (12) at (-1.5, 2) {};
		\node [style=blossom] (13) at (-1.5, 3) {};
		\node [style=blossom] (14) at (0.5, 3) {};
		\node [style=blossom] (15) at (1.75, 5) {};
		\node [style=blossom] (16) at (2.25, 5) {};
		\node [style=blossom] (17) at (1.5, 4) {};
		\node [style=blossom] (18) at (2.5, 3.5) {};
		\node [style=blossom] (19) at (2.5, 2.75) {};
		\node [style=blossom] (20) at (2.5, 2.25) {};
		\node [style=blossom] (21) at (0.75, 2) {};
		\node [style=blossom] (22) at (1.25, 2) {};
		\node [style=blossom] (23) at (0, 2) {};
		\node [style=blossom] (24) at (-0.5, 1.5) {};
		\node [style=blossom] (25) at (-0.5, 1) {};
		\node [style=blossom] (26) at (-2, -0) {};
		\node [style=blossom] (27) at (-2.5, 0.5) {};
		\node [style=blossom] (28) at (-2.5, 1.5) {};
		\node [style=blossom] (29) at (-2.5, 2.5) {};
	\end{pgfonlayer}
	\begin{pgfonlayer}{edgelayer}
		\draw [branch] (0) to (1);
		\draw [branch] (1) to (4);
		\draw [branch] (4) to (5);
		\draw [branch] (5) to (6);
		\draw [branch] (6) to (7);
		\draw [branch] (7) to (8);
		\draw [branch] (7) to (9);
		\draw [branch] (10) to (8);
		\draw [branch] (4) to (11);
		\draw [gem] (2) to (0);
		\draw [gem] (3) to (0);
		\draw [gem] (13) to (1);
		\draw [gem] (12) to (4);
		\draw [gem] (14) to (6);
		\draw [gem] (15) to (10);
		\draw [gem] (10) to (16);
		\draw [gem] (17) to (8);
		\draw [gem] (18) to (8);
		\draw [gem] (19) to (9);
		\draw [gem] (20) to (9);
		\draw [gem] (22) to (7);
		\draw [gem] (21) to (7);
		\draw [gem] (6) to (23);
		\draw [gem] (24) to (5);
		\draw [gem] (28) to (4);
		\draw [gem] (29) to (1);
		\draw [gem] (27) to (11);
		\draw [gem] (26) to (11);
		\draw [gem] (5) to (25);
	\end{pgfonlayer}
\end{tikzpicture}
\begin{tikzpicture}[vertex/.style={circle, inner sep=2pt, fill=black}, blossom/.style={diamond, inner sep=2pt, fill=green!70!black}, branch/.style={very thick, brown}, gem/.style={green!70!black}, scale=.8]
	\begin{pgfonlayer}{nodelayer}
	\clip[use as bounding box] (-3,-0.5) rectangle (3,5.5);
		\node [style=vertex, label=left:$w$] (0) at (-2, 3.5) {};
		\node at (-1.5, 2) {$S_1$};
		\node at (-1.2, 2.7) {$e_1$};
		\node at (0.8, 3.1) {$e_2$};
		\node at (1.5, 3) {$S_2$};
		\node at (0, 5) {$C'$};
		\node [style=vertex] (1) at (-2, 2.5) {};
		\node [style=vertex] (4) at (-2, 1.5) {};
		\node [style=vertex] (5) at (-1, 1.5) {};
		\node [style=vertex, label=above:$z$] (6) at (0, 2.5) {};
		\node [style=vertex] (7) at (1, 2.5) {};
		\node [style=vertex] (8) at (2, 3.5) {};
		\node [style=vertex] (9) at (2, 2.5) {};
		\node [style=vertex] (10) at (2, 4.5) {};
		\node [style=vertex] (11) at (-2, 0.5) {};
		\node [style=blossom] (15) at (1.75, 5) {};
		\node [style=blossom] (16) at (2.25, 5) {};
		\node [style=blossom] (17) at (1.5, 4) {};
		\node [style=blossom] (18) at (2.5, 3.5) {};
		\node [style=blossom] (19) at (2.5, 2.75) {};
		\node [style=blossom] (20) at (2.5, 2.25) {};
		\node [style=blossom] (21) at (0.75, 2) {};
		\node [style=blossom] (22) at (1.25, 2) {};
		\node [style=blossom] (23) at (0, 2) {};
		\node [style=blossom] (24) at (-0.5, 1.5) {};
		\node [style=blossom] (25) at (-0.5, 1) {};
		\node [style=blossom] (26) at (-2, -0) {};
		\node [style=blossom] (27) at (-2.5, 0.5) {};
		\node [style=blossom] (28) at (-2.5, 1.5) {};
		\node [style=blossom] (29) at (-2.5, 2.5) {};
	\end{pgfonlayer}
	\begin{pgfonlayer}{edgelayer}
	\fill[gray!30] (1.center)--(4.center)--(5.center)--(6.center)--cycle;
	\fill[gray!30] (6.center)--(7.center)--(8.center)--cycle;
		\draw [branch] (0) to (1);
		\draw [branch] (1) to (4);
		\draw [branch] (4) to (5);
		\draw [branch] (5) to (6);
		\draw [branch] (6) to (7);
		\draw [branch] (7) to (8);
		\draw [branch] (7) to (9);
		\draw [branch] (10) to (8);
		\draw [branch] (4) to (11);
		\draw [red, thick] (6) to (0);
		\draw [red, thick] (8) to (0);
		\draw [thick, gem] (1) to (6);
		\draw [gem] (4) to (6);
		\draw [thick, gem] (8) to (6);
		\draw [gem] (15) to (10);
		\draw [gem] (10) to (16);
		\draw [gem] (17) to (8);
		\draw [gem] (18) to (8);
		\draw [gem] (19) to (9);
		\draw [gem] (20) to (9);
		\draw [gem] (22) to (7);
		\draw [gem] (21) to (7);
		\draw [gem] (6) to (23);
		\draw [gem] (24) to (5);
		\draw [gem] (28) to (4);
		\draw [gem] (29) to (1);
		\draw [gem] (27) to (11);
		\draw [gem] (26) to (11);
		\draw [gem] (5) to (25);
	\end{pgfonlayer}
\end{tikzpicture}	
\begin{tikzpicture}[vertex/.style={circle, inner sep=2pt, fill=black}, blossom/.style={diamond, inner sep=2pt, fill=green!70!black}, branch/.style={very thick, brown}, gem/.style={green!70!black}, scale=.8]
\clip[use as bounding box] (-4,-0.5) rectangle (3,5.5);
		\begin{pgfonlayer}{nodelayer}
		\node [style=vertex] (0) at (-2, 3.5) {};
		\node [style=vertex] (1) at (-2, 2.5) {};
		\node [style=vertex] (2) at (-2, 1.5) {};
		\node [style=vertex] (3) at (-1, 1.5) {};
		\node [style=vertex] (4) at (0, 2.5) {};
		\node [style=vertex] (5) at (1, 2.5) {};
		\node [style=vertex] (6) at (2, 3.5) {};
		\node [style=vertex] (7) at (2, 2.5) {};
		\node [style=vertex] (8) at (2, 4.5) {};
		\node [style=vertex] (9) at (-2, 0.5) {};
		\node [style=blossom] (10) at (1.75, 5) {};
		\node [style=blossom] (11) at (1.75, 4.25) {};
		\node [style=blossom] (12) at (2.5, 2.75) {};
		\node [style=blossom] (13) at (2.5, 2.25) {};
		\node [style=blossom] (14) at (0.75, 2) {};
		\node [style=blossom] (15) at (1.25, 2) {};
		\node [style=blossom] (16) at (0, 2) {};
		\node [style=blossom] (17) at (-0.5, 1.5) {};
		\node [style=blossom] (18) at (-2, -0) {};
	\end{pgfonlayer}
	\begin{pgfonlayer}{edgelayer}
		\draw [branch] (0) to (1);
		\draw [branch] (1) to (2);
		\draw [branch] (2) to (3);
		\draw [branch] (3) to (4);
		\draw [branch] (4) to (5);
		\draw [branch] (5) to (6);
		\draw [branch] (5) to (7);
		\draw [branch] (8) to (6);
		\draw [branch] (2) to (9);
		\draw [gem] (10) to (8);
		\draw [gem] (11) to (6);
		\draw [gem] (12) to (7);
		\draw [gem] (13) to (7);
		\draw [gem] (15) to (5);
		\draw [gem] (14) to (5);
		\draw [gem] (4) to (16);
		\draw [gem] (17) to (3);
		\draw [gem] (2) to (4);
		\draw [gem] (1) to (4);
		\draw [thick, red] (0) to (4);
		\draw [gem] (4) to (6);
		\draw [thick, red] (0) to (6);
		\draw [style=gem, in=165, out=120, looseness=1.75] (1) to (6);
		\draw [style=gem, in=135, out=132, looseness=2.25] (2) to (6);
		\draw [style=gem, in=120, out=135, looseness=2.75] (9) to (6);
		\draw [gem] (18) to (9);
		\draw [style=gem] (3) to (9);
		\draw [style=gem, bend right, looseness=1.25] (6) to (7);
		\draw [style=gem, bend left=45, looseness=1.00] (8) to (7);
	\end{pgfonlayer}
	\end{tikzpicture}\\
	\begin{tikzpicture}[vertex/.style={circle, inner sep=2pt, fill=black}, blossom/.style={diamond, inner sep=2pt, fill=green!70!black}, branch/.style={very thick, brown}, gem/.style={green!70!black}, scale=.8]
	\clip[use as bounding box] (-3,-0.5) rectangle (3,5.5);
	\begin{pgfonlayer}{nodelayer}
	\node at (0, 5) {$\tau$};
	\node[red] at (-2, 3.1) {$c$};
		\node [style=vertex] (1) at (-2, 2.5) {};
		\node [style=vertex] (4) at (-2, 1.5) {};
		\node [style=vertex] (5) at (-1, 1.5) {};
		\node [style=vertex] (6) at (0, 2.5) {};
		\node [style=vertex] (7) at (1, 2.5) {};
		\node [style=vertex] (8) at (2, 3.5) {};
		\node [style=vertex] (9) at (2, 2.5) {};
		\node [style=vertex] (10) at (2, 4.5) {};
		\node [style=vertex] (11) at (-2, 0.5) {};
		\node [style=blossom] (12) at (-1.5, 2) {};
		\node [style=blossom] (13) at (-1.5, 3) {};
		\node [style=blossom] (14) at (0.5, 3) {};
		\node [style=blossom] (15) at (1.75, 5) {};
		\node [style=blossom] (16) at (2.25, 5) {};
		\node [style=blossom] (17) at (1.5, 4) {};
		\node [style=blossom] (18) at (2.5, 3.5) {};
		\node [style=blossom] (19) at (2.5, 2.75) {};
		\node [style=blossom] (20) at (2.5, 2.25) {};
		\node [style=blossom] (21) at (0.75, 2) {};
		\node [style=blossom] (22) at (1.25, 2) {};
		\node [style=blossom] (23) at (0, 2) {};
		\node [style=blossom] (24) at (-0.5, 1.5) {};
		\node [style=blossom] (25) at (-0.5, 1) {};
		\node [style=blossom] (26) at (-2, -0) {};
		\node [style=blossom] (27) at (-2.5, 0.5) {};
		\node [style=blossom] (28) at (-2.5, 1.5) {};
		\node [style=blossom] (29) at (-2.5, 2.5) {};
	\end{pgfonlayer}
	\begin{pgfonlayer}{edgelayer}
	\begin{scope}
	\clip (13.center) to (1.center) to (29.center) [bend left] to (13.center);
	\fill[red!20, draw=red] (-2,2.5) circle (8pt);
	\end{scope}
		\draw [branch] (1) to (4);
		\draw [branch] (4) to (5);
		\draw [branch] (5) to (6);
		\draw [branch] (6) to (7);
		\draw [branch] (7) to (8);
		\draw [branch] (7) to (9);
		\draw [branch] (10) to (8);
		\draw [branch] (4) to (11);
		\draw [gem] (13) to (1);
		\draw [gem] (12) to (4);
		\draw [gem] (14) to (6);
		\draw [gem] (15) to (10);
		\draw [gem] (10) to (16);
		\draw [gem] (17) to (8);
		\draw [gem] (18) to (8);
		\draw [gem] (19) to (9);
		\draw [gem] (20) to (9);
		\draw [gem] (22) to (7);
		\draw [gem] (21) to (7);
		\draw [gem] (6) to (23);
		\draw [gem] (24) to (5);
		\draw [gem] (28) to (4);
		\draw [gem] (29) to (1);
		\draw [gem] (27) to (11);
		\draw [gem] (26) to (11);
		\draw [gem] (5) to (25);
	\end{pgfonlayer}
\end{tikzpicture}
\begin{tikzpicture}[vertex/.style={circle, inner sep=2pt, fill=black}, blossom/.style={diamond, inner sep=2pt, fill=green!70!black}, branch/.style={very thick, brown}, gem/.style={green!70!black}, scale=.8]
	\clip[use as bounding box] (-3,-0.5) rectangle (3,5.5);
	\begin{pgfonlayer}{nodelayer}
	\node at (0, 5) {$C$};
		\node [style=vertex] (1) at (-2, 2.5) {};
		\node [style=vertex] (4) at (-2, 1.5) {};
		\node [style=vertex] (5) at (-1, 1.5) {};
		\node [style=vertex] (6) at (0, 2.5) {};
		\node [style=vertex] (7) at (1, 2.5) {};
		\node [style=vertex] (8) at (2, 3.5) {};
		\node [style=vertex] (9) at (2, 2.5) {};
		\node [style=vertex] (10) at (2, 4.5) {};
		\node [style=vertex] (11) at (-2, 0.5) {};
		\node [style=blossom] (15) at (1.75, 5) {};
		\node [style=blossom] (16) at (2.25, 5) {};
		\node [style=blossom] (17) at (1.5, 4) {};
		\node [style=blossom] (18) at (2.5, 3.5) {};
		\node [style=blossom] (19) at (2.5, 2.75) {};
		\node [style=blossom] (20) at (2.5, 2.25) {};
		\node [style=blossom] (21) at (0.75, 2) {};
		\node [style=blossom] (22) at (1.25, 2) {};
		\node [style=blossom] (23) at (0, 2) {};
		\node [style=blossom] (24) at (-0.5, 1.5) {};
		\node [style=blossom] (25) at (-0.5, 1) {};
		\node [style=blossom] (26) at (-2, -0) {};
		\node [style=blossom] (27) at (-2.5, 0.5) {};
		\node [style=blossom] (28) at (-2.5, 1.5) {};
		\node [style=blossom] (29) at (-2.5, 2.5) {};
	\end{pgfonlayer}
	\begin{pgfonlayer}{edgelayer}
	\fill[gray!30] (1.center)--(4.center)--(5.center)--(6.center)--cycle;
	\fill[gray!30] (6.center)--(7.center)--(8.center)--cycle;
		\draw [branch] (1) to (4);
		\draw [branch] (4) to (5);
		\draw [branch] (5) to (6);
		\draw [branch] (6) to (7);
		\draw [branch] (7) to (8);
		\draw [branch] (7) to (9);
		\draw [branch] (10) to (8);
		\draw [branch] (4) to (11);
		\draw [thick, gem] (1) to (6);
		\draw [gem] (4) to (6);
		\draw [thick, gem] (8) to (6);
		\draw [gem] (15) to (10);
		\draw [gem] (10) to (16);
		\draw [gem] (17) to (8);
		\draw [gem] (18) to (8);
		\draw [gem] (19) to (9);
		\draw [gem] (20) to (9);
		\draw [gem] (22) to (7);
		\draw [gem] (21) to (7);
		\draw [gem] (6) to (23);
		\draw [gem] (24) to (5);
		\draw [gem] (28) to (4);
		\draw [gem] (29) to (1);
		\draw [gem] (27) to (11);
		\draw [gem] (26) to (11);
		\draw [gem] (5) to (25);
	\end{pgfonlayer}
\end{tikzpicture}	
\begin{tikzpicture}[vertex/.style={circle, inner sep=2pt, fill=black}, blossom/.style={diamond, inner sep=2pt, fill=green!70!black}, branch/.style={very thick, brown}, gem/.style={green!70!black}, scale=.8]
		\clip[use as bounding box] (-3.5,-0.5) rectangle (3,5.5);
		\begin{pgfonlayer}{nodelayer}
		\node [style=vertex] (1) at (-2, 2.5) {};
		\node [style=vertex] (2) at (-2, 1.5) {};
		\node [style=vertex] (3) at (-1, 1.5) {};
		\node [style=vertex] (4) at (0, 2.5) {};
		\node [style=vertex] (5) at (1, 2.5) {};
		\node [style=vertex] (6) at (2, 3.5) {};
		\node [style=vertex] (7) at (2, 2.5) {};
		\node [style=vertex] (8) at (2, 4.5) {};
		\node [style=vertex] (9) at (-2, 0.5) {};
		\node [style=blossom] (10) at (1.75, 5) {};
		\node [style=blossom] (11) at (1.75, 4.25) {};
		\node [style=blossom] (12) at (2.5, 2.75) {};
		\node [style=blossom] (13) at (2.5, 2.25) {};
		\node [style=blossom] (14) at (0.75, 2) {};
		\node [style=blossom] (15) at (1.25, 2) {};
		\node [style=blossom] (16) at (0, 2) {};
		\node [style=blossom] (17) at (-0.5, 1.5) {};
		\node [style=blossom] (18) at (-2, -0) {};
	\end{pgfonlayer}
	\begin{pgfonlayer}{edgelayer}
		\draw [branch] (1) to (2);
		\draw [branch] (2) to (3);
		\draw [branch] (3) to (4);
		\draw [branch] (4) to (5);
		\draw [branch] (5) to (6);
		\draw [branch] (5) to (7);
		\draw [branch] (8) to (6);
		\draw [branch] (2) to (9);
		\draw [gem] (10) to (8);
		\draw [gem] (11) to (6);
		\draw [gem] (12) to (7);
		\draw [gem] (13) to (7);
		\draw [gem] (15) to (5);
		\draw [gem] (14) to (5);
		\draw [gem] (4) to (16);
		\draw [gem] (17) to (3);
		\draw [gem] (2) to (4);
		\draw [gem] (1) to (4);
		\draw [gem] (4) to (6);
		\draw [style=gem, in=165, out=120, looseness=1.75] (1) to (6);
		\draw [style=gem, in=135, out=132, looseness=2.25] (2) to (6);
		\draw [style=gem, in=120, out=135, looseness=2.75] (9) to (6);
		\draw [gem] (18) to (9);
		\draw [style=gem] (3) to (9);
		\draw [style=gem, bend right, looseness=1.25] (6) to (7);
		\draw [style=gem, bend left=45, looseness=1.00] (8) to (7);
	\end{pgfonlayer}
	\end{tikzpicture}
	\caption{\label{fig:case a}Case a): the vertex $w$ is grafted onto the corner $c$ of $\tau$ to obtain $\tau'$. One can perform three closures of blossoms to obtain the three triangular faces of $S_1\cup S_2$ (shaded); when performed on $\tau$, such closures yield $C$ (below); when performed on $\tau'$ and followed by the closures of $\alpha$ an $\beta$, they yield $C'$ (above). Since $C$ differs from $C'$ by the collapse of the two faces adjacent to $e$ (and the identification of $w$ with $z$), subsequent closures can be performed in the same way above and below, yielding the result on the right, where the contour of the infinite face is exactly the same.}
\end{figure}
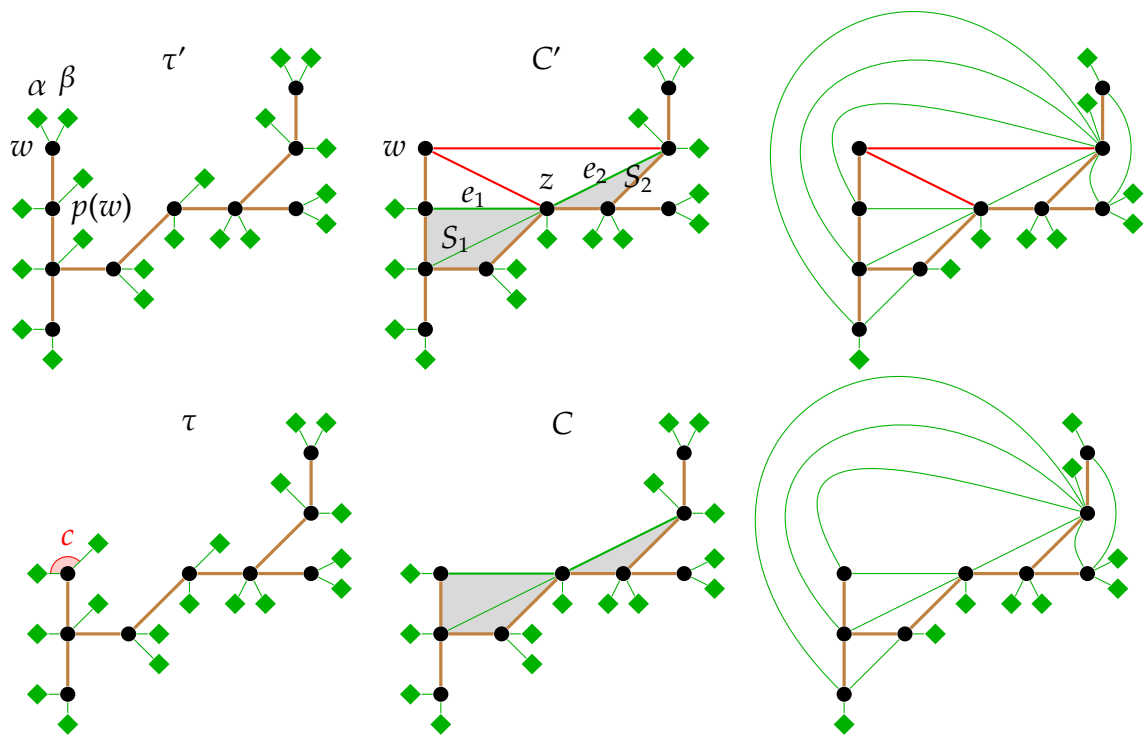

\begin{figure}\centering
\begin{tikzpicture}[vertex/.style={circle, inner sep=2pt, fill=black}, blossom/.style={diamond, inner sep=2pt, fill=green!70!black}, branch/.style={very thick, brown}, gem/.style={green!70!black}, scale=.6]
	\begin{pgfonlayer}{nodelayer}
	\node at (-2.5,4.5) {$\tau'$};
		\node [style=vertex, label=left:$w$] (0) at (-2, 3) {};
		\node [style=vertex] (1) at (-2, 2) {};
		\node [style=blossom] (2) at (-2.25, 3.5) {};
		\node [style=blossom] (3) at (-1.75, 3.5) {};
		\node [style=vertex] (4) at (-2.75, -0) {};
		\node [style=vertex] (5) at (-4, 1) {};
		\node [style=vertex] (6) at (-4.75, 2) {};
		\node [style=vertex] (7) at (0, 2) {};
		\node [style=blossom] (8) at (-3.75, 1.5) {};
		\node [style=blossom] (9) at (-4.25, 2) {};
		\node [style=blossom] (10) at (-4.25, 2.5) {};
		\node [style=vertex] (11) at (-4.75, 3) {};
		\node [style=blossom] (12) at (-4.75, 3.5) {};
		\node [style=vertex] (13) at (0, 3) {};
		\node [style=blossom] (14) at (-1.5, 2.5) {};
		\node [style=blossom] (15) at (-5.25, 3.5) {};
		\node [style=blossom] (16) at (-4.25, 0.5) {};
		\node [style=blossom] (17) at (-2.75, -0.5) {};
		\node [style=blossom] (18) at (-2.25, 0.25) {};
		\node [style=blossom] (19) at (-2, 1.5) {};
		\node [style=blossom] (20) at (0, 1.5) {};
		\node [style=blossom] (21) at (0.5, 1.75) {};
		\node [style=blossom] (22) at (0.25, 3.5) {};
		\node [style=blossom] (23) at (-0.25, 3.5) {};
	\end{pgfonlayer}
	\begin{pgfonlayer}{edgelayer}
		\draw [style=branch] (0) to (1);
		\draw [style=gem] (2) to (0);
		\draw [style=gem] (3) to (0);
		\draw [style=branch] (4) to (1);
		\draw [style=branch] (5) to (4);
		\draw [style=branch] (6) to (5);
		\draw [style=branch] (11) to (6);
		\draw [style=gem] (10) to (6);
		\draw [style=gem] (6) to (9);
		\draw [style=gem] (8) to (5);
		\draw [style=gem] (12) to (11);
		\draw [style=branch] (1) to (7);
		\draw [style=branch] (13) to (7);
		\draw [style=gem] (14) to (1);
		\draw [style=gem] (15) to (11);
		\draw [style=gem] (5) to (16);
		\draw [style=gem] (4) to (17);
		\draw [style=gem] (4) to (18);
		\draw [style=gem] (1) to (19);
		\draw [style=gem] (7) to (20);
		\draw [style=gem] (7) to (21);
		\draw [style=gem] (23) to (13);
		\draw [style=gem] (22) to (13);
	\end{pgfonlayer}
\end{tikzpicture}
\begin{tikzpicture}[vertex/.style={circle, inner sep=2pt, fill=black}, blossom/.style={diamond, inner sep=2pt, fill=green!70!black}, branch/.style={very thick, brown}, gem/.style={green!70!black}, scale=.6]
	\begin{pgfonlayer}{nodelayer}
	\node at (-2.5,4.5) {$C'$};
	\node at (-1.2,2.5) {$e_1$};
	\node at (-0.4,2.4) {$S_1$};
		\node [style=vertex,label=left:$w$] (0) at (-2, 3) {};
		\node [style=vertex] (1) at (-2, 2) {};
		\node [style=blossom] (2) at (-2.25, 3.5) {};
		\node [style=vertex] (3) at (-2.75, -0) {};
		\node [style=vertex] (4) at (-4, 1) {};
		\node [style=vertex] (5) at (-4.75, 2) {};
		\node [style=vertex] (6) at (0, 2) {};
		\node [style=blossom] (7) at (-3.75, 1.5) {};
		\node [style=blossom] (8) at (-4.25, 2) {};
		\node [style=blossom] (9) at (-4.25, 2.5) {};
		\node [style=vertex] (10) at (-4.75, 3) {};
		\node [style=blossom] (11) at (-4.75, 3.5) {};
		\node [style=vertex,label=right:$z$] (12) at (0, 3) {};
		\node [style=blossom] (13) at (-5.25, 3.5) {};
		\node [style=blossom] (14) at (-4.25, 0.5) {};
		\node [style=blossom] (15) at (-2.75, -0.5) {};
		\node [style=blossom] (16) at (-2.25, 0.25) {};
		\node [style=blossom] (17) at (-2, 1.5) {};
		\node [style=blossom] (18) at (0, 1.5) {};
		\node [style=blossom] (19) at (0.5, 1.75) {};
		\node [style=blossom] (20) at (0.25, 3.5) {};
		\node [style=blossom] (21) at (-0.25, 3.5) {};
	\end{pgfonlayer}
	\begin{pgfonlayer}{edgelayer}
	\fill[gray!20] (12.center)--(6.center)--(1.center)--cycle;
		\draw [style=branch] (0) to (1);
		\draw [style=gem] (2) to (0);
		\draw [style=branch] (3) to (1);
		\draw [style=branch] (4) to (3);
		\draw [style=branch] (5) to (4);
		\draw [style=branch] (10) to (5);
		\draw [style=gem] (9) to (5);
		\draw [style=gem] (5) to (8);
		\draw [style=gem] (7) to (4);
		\draw [style=gem] (11) to (10);
		\draw [style=branch] (1) to (6);
		\draw [style=branch] (12) to (6);
		\draw [style=gem] (13) to (10);
		\draw [style=gem] (4) to (14);
		\draw [style=gem] (3) to (15);
		\draw [style=gem] (3) to (16);
		\draw [style=gem] (1) to (17);
		\draw [style=gem] (6) to (18);
		\draw [style=gem] (6) to (19);
		\draw [style=gem] (21) to (12);
		\draw [style=gem] (20) to (12);
		\draw [style=gem] (1) to (12);
		\draw [red, thick] (0) to (12);
	\end{pgfonlayer}
\end{tikzpicture}
\begin{tikzpicture}[vertex/.style={circle, inner sep=2pt, fill=black}, blossom/.style={diamond, inner sep=2pt, fill=green!70!black}, branch/.style={very thick, brown}, gem/.style={green!70!black}, scale=.6]
	\begin{pgfonlayer}{nodelayer}
	\node at (-2.5,4.5) {$C'_1$};
	\node at (-3.5,3.2) {$\eta$};
	\node at (-3.5,2.2) {$S_3$};
	\node at (-1.3,2.7) {$f_1$};
		\node [style=vertex] (0) at (-2, 3) {};
		\node [style=vertex] (1) at (-2, 2) {};
		\node [style=blossom] (2) at (-2.25, 3.5) {};
		\node [style=vertex] (3) at (-2.75, -0) {};
		\node [style=vertex] (4) at (-4, 1) {};
		\node [style=vertex] (5) at (-4.75, 2) {};
		\node [style=vertex] (6) at (0, 2) {};
		\node [style=vertex] (7) at (-4.75, 3) {};
		\node [style=vertex] (8) at (0, 3) {};
		\node [style=blossom] (9) at (-5.25, 3.5) {};
		\node [style=blossom] (10) at (-4.25, 0.5) {};
		\node [style=blossom] (11) at (-2.75, -0.5) {};
		\node [style=blossom] (12) at (-2.25, 0.25) {};
		\node [style=blossom] (13) at (-2, 1.5) {};
		\node [style=blossom] (14) at (0, 1.5) {};
		\node [style=blossom] (15) at (0.5, 1.75) {};
		\node [style=blossom] (16) at (0.25, 3.5) {};
		\node [style=blossom] (17) at (-0.25, 3.5) {};
	\end{pgfonlayer}
	\begin{pgfonlayer}{edgelayer}
	\fill[gray!20] (0.center)--(1.center)--(3.center)--(4.center)--(5.center)--(7.center)--cycle;
		\draw [style=branch] (0) to (1);
		\draw [style=gem] (2) to (0);
		\draw [style=branch] (3) to (1);
		\draw [style=branch] (4) to (3);
		\draw [style=branch] (5) to (4);
		\draw [style=branch] (7) to (5);
		\draw [style=branch] (1) to (6);
		\draw [style=branch] (8) to (6);
		\draw [style=gem] (9) to (7);
		\draw [style=gem] (4) to (10);
		\draw [style=gem] (3) to (11);
		\draw [style=gem] (3) to (12);
		\draw [style=gem] (1) to (13);
		\draw [style=gem] (6) to (14);
		\draw [style=gem] (6) to (15);
		\draw [style=gem] (17) to (8);
		\draw [style=gem] (16) to (8);
		\draw [style=gem] (1) to (8);
		\draw [style=gem, red, thick] (0) to (8);
		\draw [style=gem] (4) to (1);
		\draw [style=gem] (5) to (1);
		\draw [style=gem, red] (5) to (0);
		\draw [style=gem, red] (7) to (0);
	\end{pgfonlayer}
\end{tikzpicture}
\begin{tikzpicture}[vertex/.style={circle, inner sep=2pt, fill=black}, blossom/.style={diamond, inner sep=2pt, fill=green!70!black}, branch/.style={very thick, brown}, gem/.style={green!70!black}, scale=.6]
	\begin{pgfonlayer}{nodelayer}
	\node at (-2.5,4.5) {$T$};
	\clip[use as bounding box] (-5.5,-0.2) rectangle (2.2,4.5);
		\node [style=vertex] (0) at (-2, 3) {};
		\node [style=vertex] (1) at (-2, 2) {};
		\node [style=vertex] (2) at (-2.75, -0) {};
		\node [style=vertex] (3) at (-4, 1) {};
		\node [style=vertex] (4) at (-4.75, 2) {};
		\node [style=vertex] (5) at (0, 2) {};
		\node [style=vertex] (6) at (-4.75, 3) {};
		\node [style=vertex] (7) at (0, 3) {};
		\node [style=vertex] (8) at (1, 2.5) {};
		\node [style=vertex] (9) at (2, 2.5) {};
	\end{pgfonlayer}
	\begin{pgfonlayer}{edgelayer}
		\draw [style=branch] (0) to (1);
		\draw [style=branch] (2) to (1);
		\draw [style=branch] (3) to (2);
		\draw [style=branch] (4) to (3);
		\draw [style=branch] (6) to (4);
		\draw [style=branch] (1) to (5);
		\draw [style=branch] (7) to (5);
		\draw [style=gem] (1) to (7);
		\draw [style=gem, red, thick] (0) to (7);
		\draw [style=gem] (3) to (1);
		\draw [style=gem] (4) to (1);
		\draw [style=gem, red] (4) to (0);
		\draw [style=gem, red] (6) to (0);
		\draw [style=gem, bend right=60, looseness=1.25] (6) to (3);
		\draw [style=gem, bend right=90, looseness=1.25] (6) to (2);
		\draw [style=gem] (7) to (8);
		\draw [style=gem] (8) to (5);
		\draw [style=gem] (7) to (9);
		\draw [style=gem] (5) to (9);
		\draw [style=gem, bend right] (1) to (9);
		\draw [style=gem, bend left, red] (0) to (9);
		\draw [style=gem, bend left=40] (6) to (9);
		\draw [style=gem, bend right] (2) to (9);
		\draw[->, very thick] (8) to (9);
	\end{pgfonlayer}
\end{tikzpicture}\\[20pt]
\begin{tikzpicture}[vertex/.style={circle, inner sep=2pt, fill=black}, blossom/.style={diamond, inner sep=2pt, fill=green!70!black}, branch/.style={very thick, brown}, gem/.style={green!70!black}, scale=.6]
	\begin{pgfonlayer}{nodelayer}
	\node at (-2.5,4.5) {$\tau$};
		\node [style=vertex] (1) at (-2, 2) {};
		\node[red] at (-2.2,2.7) {$c$};
		\node [style=vertex] (4) at (-2.75, -0) {};
		\node [style=vertex] (5) at (-4, 1) {};
		\node [style=vertex] (6) at (-4.75, 2) {};
		\node [style=vertex] (7) at (0, 2) {};
		\node [style=blossom] (8) at (-3.75, 1.5) {};
		\node [style=blossom] (9) at (-4.25, 2) {};
		\node [style=blossom] (10) at (-4.25, 2.5) {};
		\node [style=vertex] (11) at (-4.75, 3) {};
		\node [style=blossom] (12) at (-4.75, 3.5) {};
		\node [style=vertex] (13) at (0, 3) {};
		\node [style=blossom] (14) at (-1.5, 2.5) {};
		\node [style=blossom] (15) at (-5.25, 3.5) {};
		\node [style=blossom] (16) at (-4.25, 0.5) {};
		\node [style=blossom] (17) at (-2.75, -0.5) {};
		\node [style=blossom] (18) at (-2.25, 0.25) {};
		\node [style=blossom] (19) at (-2, 1.5) {};
		\node [style=blossom] (20) at (0, 1.5) {};
		\node [style=blossom] (21) at (0.5, 1.75) {};
		\node [style=blossom] (22) at (0.25, 3.5) {};
		\node [style=blossom] (23) at (-0.25, 3.5) {};
	\end{pgfonlayer}
	\begin{pgfonlayer}{edgelayer}
		\begin{scope}
	\clip (4.center) to (1.center) to (14.center) [bend right=90] to (4.center);
	\fill[red!20, draw=red] (-2,2) circle (11pt);
	\end{scope}
		\draw [style=branch] (4) to (1);
		\draw [style=branch] (5) to (4);
		\draw [style=branch] (6) to (5);
		\draw [style=branch] (11) to (6);
		\draw [style=gem] (10) to (6);
		\draw [style=gem] (6) to (9);
		\draw [style=gem] (8) to (5);
		\draw [style=gem] (12) to (11);
		\draw [style=branch] (1) to (7);
		\draw [style=branch] (13) to (7);
		\draw [style=gem] (14) to (1);
		\draw [style=gem] (15) to (11);
		\draw [style=gem] (5) to (16);
		\draw [style=gem] (4) to (17);
		\draw [style=gem] (4) to (18);
		\draw [style=gem] (1) to (19);
		\draw [style=gem] (7) to (20);
		\draw [style=gem] (7) to (21);
		\draw [style=gem] (23) to (13);
		\draw [style=gem] (22) to (13);
	\end{pgfonlayer}
\end{tikzpicture}
\begin{tikzpicture}[vertex/.style={circle, inner sep=2pt, fill=black}, blossom/.style={diamond, inner sep=2pt, fill=green!70!black}, branch/.style={very thick, brown}, gem/.style={green!70!black}, scale=.6]
	\begin{pgfonlayer}{nodelayer}
	\node at (-2.5,4.5) {$C$};
	\node at (-1.2,2.5) {$e_1$};
	\node at (-0.4,2.4) {$S_1$};
		\node [style=vertex] (1) at (-2, 2) {};
		\node [style=vertex] (3) at (-2.75, -0) {};
		\node [style=vertex] (4) at (-4, 1) {};
		\node [style=vertex] (5) at (-4.75, 2) {};
		\node [style=vertex] (6) at (0, 2) {};
		\node [style=blossom] (7) at (-3.75, 1.5) {};
		\node [style=blossom] (8) at (-4.25, 2) {};
		\node [style=blossom] (9) at (-4.25, 2.5) {};
		\node [style=vertex] (10) at (-4.75, 3) {};
		\node [style=blossom] (11) at (-4.75, 3.5) {};
		\node [style=vertex,label=right:$z$] (12) at (0, 3) {};
		\node [style=blossom] (13) at (-5.25, 3.5) {};
		\node [style=blossom] (14) at (-4.25, 0.5) {};
		\node [style=blossom] (15) at (-2.75, -0.5) {};
		\node [style=blossom] (16) at (-2.25, 0.25) {};
		\node [style=blossom] (17) at (-2, 1.5) {};
		\node [style=blossom] (18) at (0, 1.5) {};
		\node [style=blossom] (19) at (0.5, 1.75) {};
		\node [style=blossom] (20) at (0.25, 3.5) {};
		\node [style=blossom] (21) at (-0.25, 3.5) {};
	\end{pgfonlayer}
	\begin{pgfonlayer}{edgelayer}
		\fill[gray!20] (12.center)--(6.center)--(1.center)--cycle;
		\draw [style=branch] (3) to (1);
		\draw [style=branch] (4) to (3);
		\draw [style=branch] (5) to (4);
		\draw [style=branch] (10) to (5);
		\draw [style=gem] (9) to (5);
		\draw [style=gem] (5) to (8);
		\draw [style=gem] (7) to (4);
		\draw [style=gem] (11) to (10);
		\draw [style=branch] (1) to (6);
		\draw [style=branch] (12) to (6);
		\draw [style=gem] (13) to (10);
		\draw [style=gem] (4) to (14);
		\draw [style=gem] (3) to (15);
		\draw [style=gem] (3) to (16);
		\draw [style=gem] (1) to (17);
		\draw [style=gem] (6) to (18);
		\draw [style=gem] (6) to (19);
		\draw [style=gem] (21) to (12);
		\draw [style=gem] (20) to (12);
		\draw [style=gem] (1) to (12);
	\end{pgfonlayer}
\end{tikzpicture}
\begin{tikzpicture}[vertex/.style={circle, inner sep=2pt, fill=black}, blossom/.style={diamond, inner sep=2pt, fill=green!70!black}, branch/.style={very thick, brown}, gem/.style={green!70!black}, scale=.6]
	\begin{pgfonlayer}{nodelayer}
	\node at (-2.5,4.5) {$C_1$};
	\node at (-3.5,2.5) {$S_3$};
		\node [style=vertex] (1) at (-2, 2) {};
		\node [style=vertex] (3) at (-2.75, -0) {};
		\node [style=vertex] (4) at (-4, 1) {};
		\node [style=vertex] (5) at (-4.75, 2) {};
		\node [style=vertex] (6) at (0, 2) {};
		\node [style=vertex] (7) at (-4.75, 3) {};
		\node [style=vertex] (8) at (0, 3) {};
		\node [style=blossom] (9) at (-5.25, 3.5) {};
		\node [style=blossom] (10) at (-4.25, 0.5) {};
		\node [style=blossom] (11) at (-2.75, -0.5) {};
		\node [style=blossom] (12) at (-2.25, 0.25) {};
		\node [style=blossom] (13) at (-2, 1.5) {};
		\node [style=blossom] (14) at (0, 1.5) {};
		\node [style=blossom] (15) at (0.5, 1.75) {};
		\node [style=blossom] (16) at (0.25, 3.5) {};
		\node [style=blossom] (17) at (-0.25, 3.5) {};
	\end{pgfonlayer}
	\begin{pgfonlayer}{edgelayer}
	\fill[gray!20] (1.center)--(3.center)--(4.center)--(5.center)--(7.center)--(8.center)--cycle;
		\draw [style=branch] (3) to (1);
		\draw [style=branch] (4) to (3);
		\draw [style=branch] (5) to (4);
		\draw [style=branch] (7) to (5);
		\draw [style=branch] (1) to (6);
		\draw [style=branch] (8) to (6);
		\draw [style=gem] (9) to (7);
		\draw [style=gem] (4) to (10);
		\draw [style=gem] (3) to (11);
		\draw [style=gem] (3) to (12);
		\draw [style=gem] (1) to (13);
		\draw [style=gem] (6) to (14);
		\draw [style=gem] (6) to (15);
		\draw [style=gem] (17) to (8);
		\draw [style=gem] (16) to (8);
		\draw [style=gem] (1) to (8);
		\draw [style=gem] (4) to (1);
		\draw [style=gem] (5) to (1);
		\draw [style=gem, red] (5) to (8);
		\draw [style=gem, red] (7) to (8);
	\end{pgfonlayer}
\end{tikzpicture}
\begin{tikzpicture}[vertex/.style={circle, inner sep=2pt, fill=black}, blossom/.style={diamond, inner sep=2pt, fill=green!70!black}, branch/.style={very thick, brown}, gem/.style={green!70!black}, scale=.6]
	\begin{pgfonlayer}{nodelayer}
	\node at (-2.5,4.5) {$t$};
	\clip[use as bounding box] (-5.5,-0.2) rectangle (2.2,4.5);
		\node [style=vertex] (1) at (-2, 2) {};
		\node [style=vertex] (2) at (-2.75, -0) {};
		\node [style=vertex] (3) at (-4, 1) {};
		\node [style=vertex] (4) at (-4.75, 2) {};
		\node [style=vertex] (5) at (0, 2) {};
		\node [style=vertex] (6) at (-4.75, 3) {};
		\node [style=vertex] (7) at (0, 3) {};
		\node [style=vertex] (8) at (1, 2.5) {};
		\node [style=vertex] (9) at (2, 2.5) {};
	\end{pgfonlayer}
	\begin{pgfonlayer}{edgelayer}
		\draw [style=branch] (2) to (1);
		\draw [style=branch] (3) to (2);
		\draw [style=branch] (4) to (3);
		\draw [style=branch] (6) to (4);
		\draw [style=branch] (1) to (5);
		\draw [style=branch] (7) to (5);
		\draw [style=gem] (1) to (7);
		\draw [style=gem] (3) to (1);
		\draw [style=gem] (4) to (1);
		\draw [style=gem, red] (4) to (7);
		\draw [style=gem, red] (6) to (7);
		\draw [style=gem, bend right=60, looseness=1.25] (6) to (3);
		\draw [style=gem, bend right=90, looseness=1.25] (6) to (2);
		\draw [style=gem] (7) to (8);
		\draw [style=gem] (8) to (5);
		\draw [style=gem] (7) to (9);
		\draw [style=gem] (5) to (9);
		\draw [style=gem, bend right] (1) to (9);
		\draw [style=gem, bend left=40] (6) to (9);
		\draw [style=gem, bend right] (2) to (9);
		\draw[->, very thick] (8) to (9);
	\end{pgfonlayer}
\end{tikzpicture}
\caption{\label{fig:case b}Case b): the blossoming tree $\tau'$ is obtained by grafting $w$ onto the corner $c$, and the vertex $w$ is adjacent to a single endpoint of the root edge in $T=\Xi(\tau',\epsilon)$. The intermediate maps $C'$ and $C'_1$ are obtained by first closing the blossom that creates $e_1$, then the right blossom $\beta$ of $w$, and then four blossoms, two of which become edges with $w$ as an endpoint. The corresponding closure operations performed on $\tau$ (when possible) yield the intermediate maps $C$ and $C_1$. Finally, we obtain $T$ and $t=\coll(T,e)$.}
\end{figure}
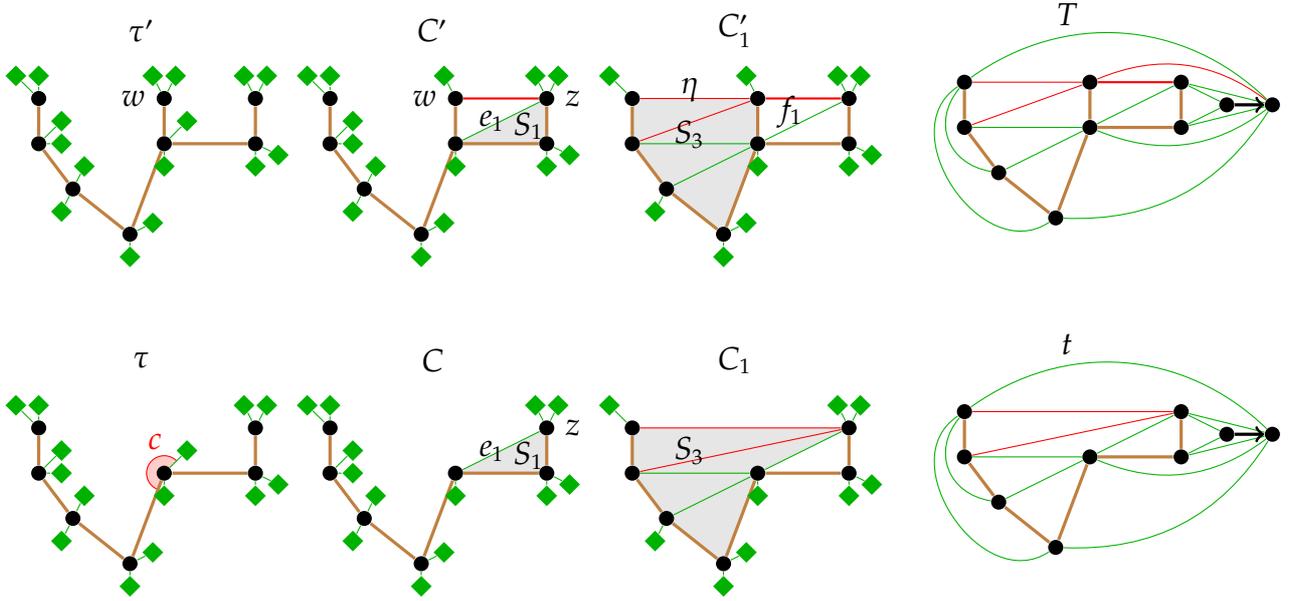

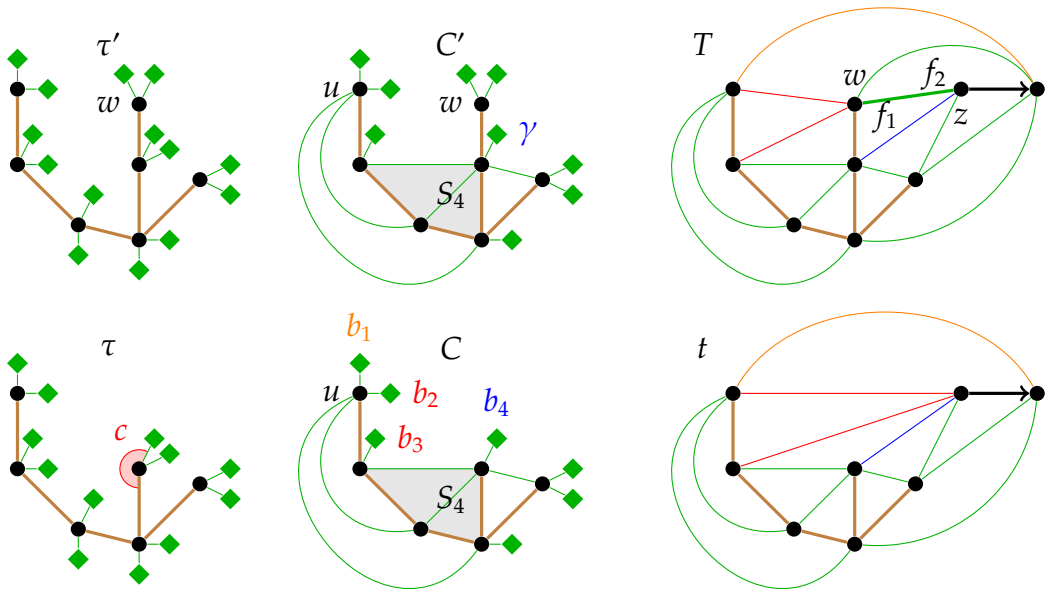
\begin{figure}\centering
\begin{tikzpicture}[vertex/.style={circle, inner sep=2pt, fill=black}, blossom/.style={diamond, inner sep=2pt, fill=green!70!black}, branch/.style={very thick, brown}, gem/.style={green!70!black}, scale=.8]
\clip[use as bounding box] (-2.5,0) rectangle (2,4);
	\begin{pgfonlayer}{nodelayer}
	\node at (-0.5,4) {$\tau'$};
		\node [style=vertex, label=left:$w$] (0) at (0, 3) {};
		\node [style=vertex] (1) at (0, 2) {};
		\node [style=vertex] (2) at (0, 0.75) {};
		\node [style=vertex] (3) at (-1, 1) {};
		\node [style=vertex] (4) at (-2, 2) {};
		\node [style=blossom] (5) at (-0.75, 1.5) {};
		\node [style=blossom] (6) at (-1.5, 2) {};
		\node [style=blossom] (7) at (-1.75, 2.5) {};
		\node [style=vertex] (8) at (-2, 3.25) {};
		\node [style=blossom] (9) at (-1.5, 3.25) {};
		\node [style=vertex] (10) at (1, 1.75) {};
		\node [style=blossom] (11) at (0.25, 2.5) {};
		\node [style=blossom] (12) at (0.5, 2.25) {};
		\node [style=blossom] (13) at (1.5, 2) {};
		\node [style=blossom] (14) at (1.5, 1.5) {};
		\node [style=blossom] (15) at (0, 0.25) {};
		\node [style=blossom] (16) at (0.5, 0.75) {};
		\node [style=blossom] (17) at (-1, 0.5) {};
		\node [style=blossom] (18) at (-2, 3.75) {};
		\node [style=blossom] (19) at (-0.25, 3.5) {};
		\node [style=blossom] (20) at (0.25, 3.5) {};
	\end{pgfonlayer}
	\begin{pgfonlayer}{edgelayer}
		\draw [style=branch] (8) to (4);
		\draw [style=branch] (4) to (3);
		\draw [style=branch] (3) to (2);
		\draw [style=branch] (2) to (1);
		\draw [style=branch] (1) to (0);
		\draw [style=branch] (2) to (10);
		\draw [style=gem] (18) to (8);
		\draw [style=gem] (8) to (9);
		\draw [style=gem] (7) to (4);
		\draw [style=gem] (4) to (6);
		\draw [style=gem] (5) to (3);
		\draw [style=gem] (3) to (17);
		\draw [style=gem] (2) to (15);
		\draw [style=gem] (2) to (16);
		\draw [style=gem] (10) to (14);
		\draw [style=gem] (10) to (13);
		\draw [style=gem] (12) to (1);
		\draw [style=gem] (11) to (1);
		\draw [style=gem] (19) to (0);
		\draw [style=gem] (20) to (0);
	\end{pgfonlayer}
\end{tikzpicture}
\begin{tikzpicture}[vertex/.style={circle, inner sep=2pt, fill=black}, blossom/.style={diamond, inner sep=2pt, fill=green!70!black}, branch/.style={very thick, brown}, gem/.style={green!70!black}, scale=.8]
	\clip[use as bounding box] (-3.5,0) rectangle (2,4);
	\begin{pgfonlayer}{nodelayer}
	\node at (-0.5,4) {$C'$};
	\node at (-0.5,1.5) {$S_4$};
		\node [style=vertex, label=left:$w$] (0) at (0, 3) {};
		\node [style=vertex] (1) at (0, 2) {};
		\node [style=vertex] (2) at (0, 0.75) {};
		\node [style=vertex] (3) at (-1, 1) {};
		\node [style=vertex] (4) at (-2, 2) {};
		\node [style=blossom] (5) at (-1.75, 2.5) {};
		\node [style=vertex, label=left:$u$] (6) at (-2, 3.25) {};
		\node [style=blossom] (7) at (-1.5, 3.25) {};
		\node [style=vertex] (8) at (1, 1.75) {};
		\node [style=blossom, label=right:\textcolor{blue}{$\gamma$}] (9) at (0.25, 2.5) {};
		\node [style=blossom] (10) at (1.5, 2) {};
		\node [style=blossom] (11) at (1.5, 1.5) {};
		\node [style=blossom] (12) at (0.5, 0.75) {};
		\node [style=blossom] (13) at (-2, 3.75) {};
		\node [style=blossom] (14) at (-0.25, 3.5) {};
		\node [style=blossom] (15) at (0.25, 3.5) {};
	\end{pgfonlayer}
	\begin{pgfonlayer}{edgelayer}
	\fill[gray!20] (4.center)--(3.center)--(2.center)--(1.center)--(4.center);
		\draw [style=branch] (6) to (4);
		\draw [style=branch] (4) to (3);
		\draw [style=branch] (3) to (2);
		\draw [style=branch] (2) to (1);
		\draw [style=branch] (1) to (0);
		\draw [style=branch] (2) to (8);
		\draw [style=gem] (13) to (6);
		\draw [style=gem] (6) to (7);
		\draw [style=gem] (5) to (4);
		\draw [style=gem] (2) to (12);
		\draw [style=gem] (8) to (11);
		\draw [style=gem] (8) to (10);
		\draw [style=gem] (9) to (1);
		\draw [style=gem] (14) to (0);
		\draw [style=gem] (15) to (0);
		\draw [style=gem] (3) to (1);
		\draw [style=gem] (4) to (1);
		\draw [style=gem] (1) to (8);
		\draw [style=gem, bend right=75, looseness=1.50] (6) to (3);
		\draw [style=gem, bend right=105, looseness=2.00] (6) to (2);
	\end{pgfonlayer}
\end{tikzpicture}
\begin{tikzpicture}[vertex/.style={circle, inner sep=2pt, fill=black}, blossom/.style={diamond, inner sep=2pt, fill=green!70!black}, branch/.style={very thick, brown}, gem/.style={green!70!black}, scale=.8]
\clip[use as bounding box] (-4,0) rectangle (2,4);
	\begin{pgfonlayer}{nodelayer}
	\node at (-2.5,4) {$T$};
	\node at (0.5,2.8) {$f_1$};
	\node at (1.3,3.5) {$f_2$};
		\node [style=vertex, label=above:$w$] (0) at (0, 3) {};
		\node [style=vertex] (1) at (0, 2) {};
		\node [style=vertex] (2) at (0, 0.75) {};
		\node [style=vertex] (3) at (-1, 1) {};
		\node [style=vertex] (4) at (-2, 2) {};
		\node [style=vertex] (5) at (-2, 3.25) {};
		\node [style=vertex] (6) at (1, 1.75) {};
		\node [style=vertex, label=below:$z$] (7) at (1.75, 3.25) {};
		\node [style=vertex] (8) at (3, 3.25) {};
	\end{pgfonlayer}
	\begin{pgfonlayer}{edgelayer}
		\draw [style=branch] (5) to (4);
		\draw [style=branch] (4) to (3);
		\draw [style=branch] (3) to (2);
		\draw [style=branch] (2) to (1);
		\draw [style=branch] (1) to (0);
		\draw [style=branch] (2) to (6);
		\draw [style=gem] (3) to (1);
		\draw [style=gem] (4) to (1);
		\draw [style=gem] (1) to (6);
		\draw [style=gem, bend right=75, looseness=1.50] (5) to (3);
		\draw [style=gem, bend right=105, looseness=2.00] (5) to (2);
		\draw [style=gem, red] (5) to (0);
		\draw [style=gem, red] (4) to (0);
		\draw [style=gem, very thick] (0) to (7);
		\draw [style=gem, blue] (1) to (7);
		\draw [style=gem] (6) to (7);
		\draw [style=gem, bend left=60, looseness=1.00] (0) to (8);
		\draw [style=gem] (6) to (8);
		\draw [style=gem, bend right=45, looseness=1.00] (2) to (8);
		\draw [style=gem, bend left=60, looseness=1.00, orange] (5) to (8);
		\draw[very thick, ->] (7) to (8);
	\end{pgfonlayer}
\end{tikzpicture}\\
\begin{tikzpicture}[vertex/.style={circle, inner sep=2pt, fill=black}, blossom/.style={diamond, inner sep=2pt, fill=green!70!black}, branch/.style={very thick, brown}, gem/.style={green!70!black}, scale=.8]
\clip[use as bounding box] (-2.5,0) rectangle (2,5);
	\begin{pgfonlayer}{nodelayer}
	\node at (-0.5,4) {$\tau$};
	\node[red] at (-0.3,2.6) {$c$};
		\node [style=vertex] (1) at (0, 2) {};
		\node [style=vertex] (2) at (0, 0.75) {};
		\node [style=vertex] (3) at (-1, 1) {};
		\node [style=vertex] (4) at (-2, 2) {};
		\node [style=blossom] (5) at (-0.75, 1.5) {};
		\node [style=blossom] (6) at (-1.5, 2) {};
		\node [style=blossom] (7) at (-1.75, 2.5) {};
		\node [style=vertex] (8) at (-2, 3.25) {};
		\node [style=blossom] (9) at (-1.5, 3.25) {};
		\node [style=vertex] (10) at (1, 1.75) {};
		\node [style=blossom] (11) at (0.25, 2.5) {};
		\node [style=blossom] (12) at (0.5, 2.25) {};
		\node [style=blossom] (13) at (1.5, 2) {};
		\node [style=blossom] (14) at (1.5, 1.5) {};
		\node [style=blossom] (15) at (0, 0.25) {};
		\node [style=blossom] (16) at (0.5, 0.75) {};
		\node [style=blossom] (17) at (-1, 0.5) {};
		\node [style=blossom] (18) at (-2, 3.75) {};
	\end{pgfonlayer}
	\begin{pgfonlayer}{edgelayer}
	\begin{scope}
	\clip (2.center) to (1.center) to (11.center) [bend right=90, looseness=1.1] to (2.center);
	\fill[red!20, draw=red] (0,2) circle (9pt);
	\end{scope}
		\draw [style=branch] (8) to (4);
		\draw [style=branch] (4) to (3);
		\draw [style=branch] (3) to (2);
		\draw [style=branch] (2) to (1);
		\draw [style=branch] (2) to (10);
		\draw [style=gem] (18) to (8);
		\draw [style=gem] (8) to (9);
		\draw [style=gem] (7) to (4);
		\draw [style=gem] (4) to (6);
		\draw [style=gem] (5) to (3);
		\draw [style=gem] (3) to (17);
		\draw [style=gem] (2) to (15);
		\draw [style=gem] (2) to (16);
		\draw [style=gem] (10) to (14);
		\draw [style=gem] (10) to (13);
		\draw [style=gem] (12) to (1);
		\draw [style=gem] (11) to (1);
	\end{pgfonlayer}
\end{tikzpicture}
\begin{tikzpicture}[vertex/.style={circle, inner sep=2pt, fill=black}, blossom/.style={diamond, inner sep=2pt, fill=green!70!black}, branch/.style={very thick, brown}, gem/.style={green!70!black}, scale=.8]
	\clip[use as bounding box] (-3.5,0) rectangle (2,5);
	\begin{pgfonlayer}{nodelayer}
	\node at (-0.5,4) {$C$};
	\node at (-0.5,1.5) {$S_4$};
		\node [style=vertex] (1) at (0, 2) {};
		\node [style=vertex] (2) at (0, 0.75) {};
		\node [style=vertex] (3) at (-1, 1) {};
		\node [style=vertex] (4) at (-2, 2) {};
		\node [style=blossom, label=right:$\textcolor{red}{b_3}$] (5) at (-1.75, 2.5) {};
		\node [style=vertex, label=left:$u$] (6) at (-2, 3.25) {};
		\node [style=blossom, label=right:$\textcolor{red}{b_2}$] (7) at (-1.5, 3.25) {};
		\node [style=vertex] (8) at (1, 1.75) {};
		\node [style=blossom, label=above:$\textcolor{blue}{b_4}$] (9) at (0.25, 2.5) {};
		\node [style=blossom] (10) at (1.5, 2) {};
		\node [style=blossom] (11) at (1.5, 1.5) {};
		\node [style=blossom] (12) at (0.5, 0.75) {};
		\node [style=blossom, label=above:$\textcolor{orange}{b_1}$] (13) at (-2, 3.75) {};
	\end{pgfonlayer}
	\begin{pgfonlayer}{edgelayer}
	\fill[gray!20] (4.center)--(3.center)--(2.center)--(1.center)--(4.center);
		\draw [style=branch] (6) to (4);
		\draw [style=branch] (4) to (3);
		\draw [style=branch] (3) to (2);
		\draw [style=branch] (2) to (1);
		\draw [style=branch] (2) to (8);
		\draw [style=gem] (13) to (6);
		\draw [style=gem] (6) to (7);
		\draw [style=gem] (5) to (4);
		\draw [style=gem] (2) to (12);
		\draw [style=gem] (8) to (11);
		\draw [style=gem] (8) to (10);
		\draw [style=gem] (9) to (1);
		\draw [style=gem] (3) to (1);
		\draw [style=gem] (4) to (1);
		\draw [style=gem] (1) to (8);
		\draw [style=gem, bend right=75, looseness=1.50] (6) to (3);
		\draw [style=gem, bend right=105, looseness=2.00] (6) to (2);
	\end{pgfonlayer}
\end{tikzpicture}
\begin{tikzpicture}[vertex/.style={circle, inner sep=2pt, fill=black}, blossom/.style={diamond, inner sep=2pt, fill=green!70!black}, branch/.style={very thick, brown}, gem/.style={green!70!black}, scale=.8]
\clip[use as bounding box] (-4,0) rectangle (2,5);
	\begin{pgfonlayer}{nodelayer}
	\node at (-2.5,4) {$t$};
		\node [style=vertex] (1) at (0, 2) {};
		\node [style=vertex] (2) at (0, 0.75) {};
		\node [style=vertex] (3) at (-1, 1) {};
		\node [style=vertex] (4) at (-2, 2) {};
		\node [style=vertex] (5) at (-2, 3.25) {};
		\node [style=vertex] (6) at (1, 1.75) {};
		\node [style=vertex] (7) at (1.75, 3.25) {};
		\node [style=vertex] (8) at (3, 3.25) {};
	\end{pgfonlayer}
	\begin{pgfonlayer}{edgelayer}
		\draw [style=branch] (5) to (4);
		\draw [style=branch] (4) to (3);
		\draw [style=branch] (3) to (2);
		\draw [style=branch] (2) to (1);
		\draw [style=branch] (2) to (6);
		\draw [style=gem] (3) to (1);
		\draw [style=gem] (4) to (1);
		\draw [style=gem] (1) to (6);
		\draw [style=gem, bend right=75, looseness=1.50] (5) to (3);
		\draw [style=gem, bend right=105, looseness=2.00] (5) to (2);
		\draw [style=gem, red] (5) to (7);
		\draw [style=gem, red] (4) to (7);
		\draw [style=gem, blue] (1) to (7);
		\draw [style=gem] (6) to (7);
		\draw [style=gem] (6) to (8);
		\draw [style=gem, bend right=45, looseness=1.00] (2) to (8);
		\draw [style=gem, bend left=60, looseness=1.00, orange] (5) to (8);
		\draw[->, very thick] (7) to (8);
	\end{pgfonlayer}
\end{tikzpicture}
	\caption{\label{fig:case c}Case c): $\tau'$ is obtained from $\tau$ by grafting $w$ and with its two blossom children onto the corner $c$; in the triangulation $T=\Xi(\tau',\epsilon)$, the vertex $w$ is a neighbour of both endpoints of the root edge. The intermediate map $C$ is built by first performing the two closures that create the faces of $S_4$ and then performing the other three possible closures. Above, the map $C'$ is the one obtained by closing the corresponding blossoms of $\tau'$ in the same order. Note how the blossoms $b_1,\ldots, b_4$ of $C$ end up joined to the endpoints of the root edge in $t$; above, the corresponding blossoms $b_2, b_3$ are joined to $w$, while $\gamma=b_4$ and $b_1$ are joined to the endpoints of the root edge. The blossoms $\alpha, \beta$ and the edge $(w,p(w))$ create the faces $f_1, f_2$ of $T$.}
\end{figure}

\end{proof}

What we now wish to do is obtain the existence of a growth scheme for triangulations from the existence of a growth scheme for 4-ary trees, which is proven in \cite{LW04}. We will do this in two steps.

\begin{lemma}\label{lemma:h function}
There exists a function $h:\mathbb{N}^2\to [0,1]$ such that for all $a,b\in\mathbb{N}^2$
\begin{itemize}
\item if $a+b\geq 1$, then $h(a,b)+h(b,a)=1$;	
\item $h(a+1,b)\frac{|\CT^4_{a+1}|}{|\CT^4_a|}+h(b+1,a)\frac{|\CT^4_{b+1}|}{|\CT^4_b|}=\frac{|\STr_{a+b+2}|}{|\STr_{a+b+1}|}\frac{a+b+2}{a+b+1}.$
\end{itemize}
	
\end{lemma}

\begin{proof}
In order to prove the statement, we rephrase it in terms of flow networks. For each $n\in\mathbb{N}$, consider the bipartite directed graph whose vertices are indexed by pairs $(a,b)\in\mathbb{N}^2$ such that $a+b=n$ or $a+b=n+1$, with an edge from each pair $(a,b)$ such that $a+b=n$ to the two pairs $(a+1,b)$ and $(a,b+1)$. Add a source $s$ connected to all vertices $(a,b)$ such that $a+b=n$ and a sink $t$ connected to all vertices $(a,b)$ such that $a+b=n+1$. Let $G_n$ be the flow network given by the directed graph just described (see Figure~\ref{fig:flow network}), with capacities assigned as follows:
\begin{itemize}
\item the edge from $s$ to a pair $(a,b)$ such that $a+b=n$ has capacity $C_{s\to(a,b)}=\frac{|\CT^4_a||\CT^4_b|}{(n+1)|\STr_{n+1}|}$;
\item the edge from a pair $(a,b)$ such that $a+b=n+1$ to $t$ has capacity $C_{(a,b)\to t}=\frac{|\CT^4_a||\CT^4_b|}{(n+2)|\STr_{n+2}|}$;
\item edges not involving the source or sink have infinite capacity.	
\end{itemize}

Note that we have $\sum_{a=0}^N|\CT^4_a||\CT^4_{N-a}|=(N+1)|\STr_{N+1}|$ by \eqref{formula triangulations-trees}, so the total capacity of edges issued from the source is one, and the same is true for the total capacity of edges involving the sink.

The existence of a function $h$ satisfying the requirements of the lemma is implied by the existence of a flow of value $1$ on each network $G_n$. Indeed, given a unit flow $(f_e)_{e\in E(G_n)}$ on the flow network $G_n$, we may assume without loss of generality that $f_{(a,b)\to(a+1,b)}=f_{(b,a)\to(b,a+1)}$, as replacing $f_{(a,b)\to(a+1,b)}$ and $f_{(b,a)\to(b,a+1)}$ by $\frac{f_{(a,b)\to(a+1,b)}+f_{(b,a)\to(b,a+1)}}{2}$ for all $a,b$ still yields a unit flow on $G_n$, thanks to the fact that $C_{s\to(a,b)}=C_{s\to(b,a)}$ and $C_{(a,b)\to t}=C_{(b,a)\to t}$.

We will now set $h(a+1,b)=\frac{f_{(a,b)\to (a+1,b)}}{C_{(a+1,b)\to t}}$ for all $(a,b)\in\mathbb{N}^2$ and $h(0,b)=0$ for all $b\in\mathbb{N}$. When $a,b\geq 1$, we have
$$h(a,b)C_{(a,b)\to t}=f_{(a-1,b)\to (a,b)}=f_{(b,a-1)\to (b,a)}=C_{(a,b)\to t}-f_{(b-1,a)\to (b,a)}=(1-h(a,b))C_{(a,b)\to t};$$
moreover, $h(0,b)=0$ and $h(a,0)=1$ for all $a,b\geq 1$.

We also have 
$$h(a+1,b)C_{(a+1,b)\to t}+h(b+1,a)C_{(b+1,a)\to t}=f_{(a,b)\to(a+1,b)}+f_{(b,a)\to(b+1,a)}=$$
$$=f_{(a,b)\to(a+1,b)}+f_{(a,b)\to(a,b+1)}=C_{s\to(a,b)},$$
which is exactly the second required expression for $h$.

We now claim the existence of a unit flow on $G_n$, and will proceed to show this claim by the max-flow-min-cut theorem via a series of quick lemmas, culminating in Corollary~\ref{cor:unit flow}, which completes the proof of the existence of $h$.
\end{proof}

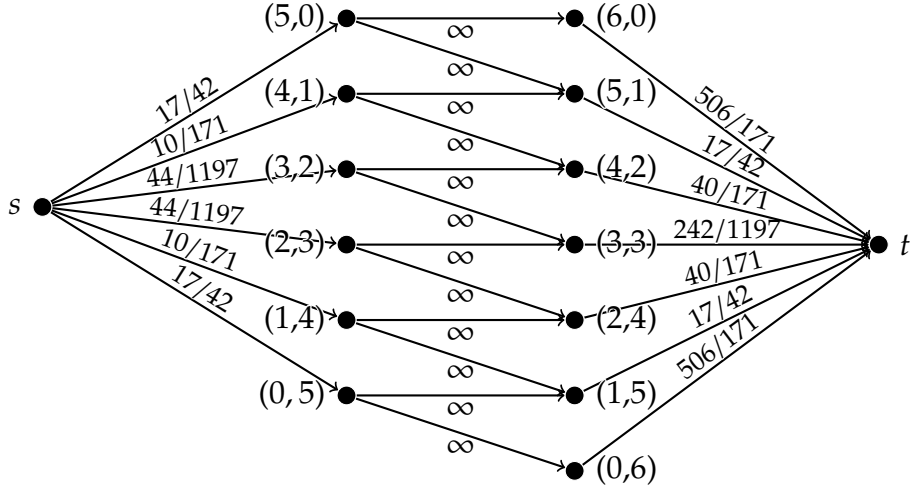
\begin{figure}\centering
\begin{tikzpicture}[vertex/.style={circle, fill=black, inner sep=2.5pt, draw=white}, simple/.style={->, thick}]
	\begin{pgfonlayer}{nodelayer}
		\node [style=vertex, label=left:${(0,5)}$] (0) at (-6, -0) {};
		\node [style=vertex, label=left:$\contour{white}{(1,4)}$] (1) at (-6, 1) {};
		\node [style=vertex, label=left:$\contour{white}{(2,3)}$] (2) at (-6, 2) {};
		\node [style=vertex, label=left:$\contour{white}{(3,2)}$] (3) at (-6, 3) {};
		\node [style=vertex, label=left:$\contour{white}{(4,1)}$] (4) at (-6, 4) {};
		\node [style=vertex, label=left:$\contour{white}{(5,0)}$] (5) at (-6, 5) {};
		\node [style=vertex, label=right:$\contour{white}{(5,1)}$] (6) at (-3, 4) {};
		\node [style=vertex, label=right:$\contour{white}{(4,2)}$] (7) at (-3, 3) {};
		\node [style=vertex, label=right:$\contour{white}{(3,3)}$] (8) at (-3, 2) {};
		\node [style=vertex, label=right:$\contour{white}{(2,4)}$] (9) at (-3, 1) {};
		\node [style=vertex, label=right:$\contour{white}{(1,5)}$] (10) at (-3, -0) {};
		\node [style=vertex, label=right:$\contour{white}{(0,6)}$] (11) at (-3, -1) {};
		\node [style=vertex, label=right:$\contour{white}{(6,0)}$] (12) at (-3, 5) {};
		\node [style=vertex, label=left:$s$] (13) at (-10, 2.5) {};
		\node [style=vertex, label=right:$t$] (14) at (1, 2) {};
		\foreach \i in {1,...,12} \node at (-4.5,0.5*\i-1.2) {\contour{white}{$\infty$}};
	\end{pgfonlayer}
	\begin{pgfonlayer}{edgelayer}
		\draw [style=simple] (5) to (6);
		\draw [style=simple] (4) to (6);
		\draw [style=simple] (4) to (7);
		\draw [style=simple] (3) to (7);
		\draw [style=simple] (3) to (8);
		\draw [style=simple] (2) to (8);
		\draw [style=simple] (2) to (9);
		\draw [style=simple] (1) to (9);
		\draw [style=simple] (1) to (10);
		\draw [style=simple] (0) to (10);
		\draw [style=simple] (0) to (11);
		\draw [style=simple] (5) to (12);
		\draw (13) edge [style=simple] node [midway, above=-4pt, sloped] {\footnotesize\contour{white}{$17/42$}} (5);
		\draw (13) edge [style=simple] node [midway, above=-4pt, sloped] {\footnotesize\contour{white}{$10/171$}} (4);
		\draw (13) edge [style=simple] node [midway, above=-4pt, sloped] {\footnotesize\contour{white}{$44/1197$}} (3);
		\draw (13) edge [style=simple] node [midway, above=-4pt, sloped] {\footnotesize\contour{white}{$44/1197$}} (2);
		\draw (13) edge [style=simple] node [midway, above=-4pt, sloped] {\footnotesize\contour{white}{$10/171$}} (1);
		\draw (13) edge [style=simple] node [midway, above=-4pt, sloped] {\footnotesize\contour{white}{$17/42$}} (0);
		\draw (11) edge [style=simple] node [midway, above=-4pt, sloped] {\footnotesize\contour{white}{$506/171$}} (14);
		\draw (10) edge [style=simple] node [midway, above=-4pt, sloped] {\footnotesize\contour{white}{$17/42$}} (14);
		\draw (9) edge [style=simple] node [midway, above=-4pt, sloped] {\footnotesize\contour{white}{$40/171$}} (14);
		\draw (8) edge [style=simple] node [midway, above=-4pt, sloped] {\footnotesize\contour{white}{$242/1197$}} (14);
		\draw (7) edge [style=simple] node [midway, above=-4pt, sloped] {\footnotesize\contour{white}{$40/171$}} (14);
		\draw (6) edge [style=simple] node [midway, above=-4pt, sloped] {\footnotesize\contour{white}{$17/42$}} (14);
		\draw (12) edge [style=simple] node [midway, above=-4pt, sloped] {\footnotesize\contour{white}{$506/171$}} (14);
	\end{pgfonlayer}
\end{tikzpicture}
\caption{\label{fig:flow network}The flow network $G_n$ for $n=5$. Edges are labelled with their capacities: for example, $\frac{|\CT^4_5||\CT^4_0|}{6|\STr_6|}=\frac{17}{42}$ and $\frac{|\CT^4_3||\CT^4_3|}{7|\STr_7|}=\frac{242}{1197}$.}
\end{figure}

Let $G_n$ be the flow network described within the proof of Lemma~\ref{lemma:h function} and depicted by Figure~\ref{fig:flow network}. An $s$-$t$-cut of $G_n$ is a partition of the vertices of $G_n$ as $S\cup T$, where $s\in S$ and $t\in T$, and its capacity is the sum of the capacities of edges of $G_n$ that join a vertex in $S$ to a vertex in $T$. For $n\geq 1$, let $S_n=\{(a,b)\in\mathbb{N}^2\mid a+b=n\}$, so that we can think of the vertices of $G_n$ as $S_n\cup S_{n+1}\cup\{s,t\}$. We can represent an $s$-$t$-cut of $G_n$ by the set $K=(S_n\cap T) \cup (S_{n+1}\cap S)$; this is a handy representation, as the capacity of the cut is infinite if there exists an edge $(v,w)$, where $v\in S_n$ and $w\in S_{n+1}$, such that $v,w\notin K$, and it is equal to $\sum_{v\in K}1_{v\in S_n}C_{s\to v}+1_{v\in S_{n+1}}C_{v\to t}$ otherwise.

In the following series of Lemmas, we refer to the flow network $G_n$ without redefining it and denote cuts by the corresponding set $K$ as constructed above; we write $|K|$ for the capacity of the cut.

\begin{lemma}
The cuts of $G_n$ having smallest capacity are of the form $K_i=\{(n-j,j)\mid 0\leq j\leq i\}\cup \{(n+1-j,j)\mid i+1\leq j\leq n+1\}$ for some $-1\leq i<n$, or $K=S_n$.
\end{lemma}

\begin{proof}
Note that all cuts listed in the statement have finite capacity.

	Consider a cut $K$ of $G_n$ that is not of the form in the statement and suppose it has finite capacity (otherwise it is not minimal). If $K$ strictly contains $S_n$ or some $K_i$, then it cannot be minimal, so we may assume this is not the case.
	
	 Under these assumptions, there is a vertex $(n-i,i)\in K$ such that $i>0$ and $(n-i+1,i-1)\notin K$. Indeed, if this were not the case then $K\cap S_n$ would be of the form $\{(n-j,j)\mid 0\leq j\leq i\}$ for some $i>0$, which, in order for $K$ to have finite capacity, would imply that it contains $K_i$ or (if $i=n$) $S_n$.
	 
	  Consider the cut $K'=K\cup \{(n-i,i+1)\}\setminus\{(n-i,i)\}$; we claim that its capacity is strictly smaller than that of $K$. Indeed, the cut $K'$ has finite capacity: if there is an edge of $G_n$ of infinite capacity that has no endpoint in $K'$, then it must involve the vertex $(n-i,i)$ and cannot be the one joining $(n-i,i)$ to $(n-i,i+1)$; it must therefore be the edge between $(n-i,i)$ and $(n-i+1,i)$. However, this cannot be the case, because $(n-i+1,i)\in K$ (otherwise $K$ would not contain either endpoint of the edge from $(n-i+1,i-1)$ to $(n-i+1,i)$ and therefore have infinite capacity) and so $(n-i+1,i)\in K'$.
	
	The difference in capacities is given by
	$$|K'|-|K|=C_{(n-i,i+1)\to t}-C_{s\to(n-i,i)}=\frac{|\CT^4_{i+1}||\CT^4_{n-i}|}{(n+2)|\STr_{n+2}|}-\frac{|\CT^4_{i}||\CT^4_{n-i}|}{(n+1)|\STr_{n+1}|};$$
	this quantity is strictly negative, as
\begin{equation}\label{eq:increasing ratio for 4-ary trees}\frac{\CT^4_{i+1}}{\CT^4_i}=\frac{4(4i+1)(4i+2)(4i+3)}{(3i+2)(3i+3)(3i+4)}=4\prod_{j=1}^3\frac{4i+j}{3i+j+1}\end{equation}
	is strictly increasing in $i$ (each term in the product is), which implies that 	
	$$\frac{(n+1)|\STr_{n+1}||\CT^4_{i+1}||\CT^4_{n-i}|}{(n+2)|\STr_{n+2}||\CT^4_{i}||\CT^4_{n-i}|}\leq\frac{(n+1)|\STr_{n+1}||\CT^4_{n+1}|}{(n+2)|\STr_{n+2}||\CT^4_{n}|}=\frac{4(4n+1)(4n+2)(4n+3)}{(3n+2)(3n+3)(3n+4)}\cdot \frac{(n+1)|\STr_{n+1}|}{(n+2)|\STr_{n+2}|}.$$
	Using the expression from \eqref{formula triangulations-trees} for the number of simple triangulations, we get that the ratio above equals
	\begin{equation}\label{eq:<1 ratio}\frac{4(4n+1)(4n+2)(4n+3)(n+1)(4n+1)!(n+2)!(3n+5)!}{(3n+2)(3n+3)(3n+4)(n+2)(4n+5)!(n+1)!(3n+2)!}=\frac{(4n+1)(3n+5)}{(3n+2)(4n+5)}<1.\end{equation}
\end{proof}

\begin{lemma}
Among the cuts of $G_n$ of the form $K_i=\{(n-j,j)\mid 0\leq j\leq i\}\cup \{(n+1-j,j)\mid i+1\leq j\leq n+1\}$, where $-1\leq i<n$, or $S_n$, the ones with smallest capacity are $K_{-1}$ (which is $S_{n+1}$) and $S_n$, whose capacity is 1.
\end{lemma}

\begin{proof}
We have the inequality $|K_{i+1}|>|K_i|$	 for $i=-1,\ldots,n-1$. Indeed, we have $K_{i+1}=K_i\cup\{(n-i-1,i+1)\}\setminus\{(n-i,i+1)\}$, so
$$|K_{i+1}|-|K_i|=C_{s\to(n-i-1,i+1)}-C_{(n-i,i+1)\to t}=\frac{|\CT^4_{n-i-1}||\CT^4_{i+1}|}{(n+1)|\STr_{n+1}|}-\frac{|\CT^4_{n-i}||\CT^4_{i+1}|}{(n+2)|\STr_{n+2}|},$$
and this quantity is strictly positive by the argument of the previous lemma. Indeed, since $\frac{\CT^4_{a+1}}{\CT^4_a}$ is increasing in $a$ by \eqref{eq:increasing ratio for 4-ary trees}, we have
$$\frac{|K_{i+1}|}{|K_i|}=\frac{(n+2)|\STr_{n+2}||\CT^4_{n-i-1}|}{(n+1)|\STr_{n+1}||\CT^4_{n-i}|}\geq\frac{(n+2)|\STr_{n+2}||\CT^4_{n}|}{(n+1)|\STr_{n+1}||\CT^4_{n+1}|}>	1,$$
as it is the inverse of \eqref{eq:<1 ratio}. It follows that $|K_{-1}|=1<|K_i|$ for $i\geq 0$; the capacity of the cut $S_n$ is also 1.
\end{proof}

Finally, by the max-flow-min-cut theorem, we have obtained the following, which concludes the proof of Lemma~\ref{lemma:h function}:
\begin{cor}\label{cor:unit flow}
For all $n\geq 1$, there exists a unit flow on the network $G_n$.	
\end{cor}

We are now ready to prove the main result, i.e. Theorem~\ref{triangulation growth}.

\begin{proof}[Proof of Theorem~\ref{triangulation growth}]
Let $(g_n)_{n\geq 0}$ be a growth scheme for complete 4-ary trees, whose existence is guaranteed by Theorem~\ref{d-ary trees growth}.
Define $f_n:\STr_{n+1}\times\STr_n\to [0,1]$ as follows. Given a triangulation $\Delta\in\STr_k$ and $\epsilon=\pm1$, let $S^\epsilon(\Delta)$ be the set of all pairs $(L,R)$ of 4-ary trees for which the corresponding blossoming tree $\tau$ is such that $\Xi(\tau,\epsilon)=\Delta$. Note that $|S^1(\Delta)\cup S^{-1}(\Delta)|=2k$, as the correspondence between pairs $(L,R)$ and simple triangulations with $2k$ faces is $2k$-to-$2$ by Construction~\ref{PS construction}.

For $T\in\STr_{n+1}$ and $t\in\STr_n$, set
$$f_n(T,t)=\frac{1}{2n+2}\sum_{\epsilon=\pm1}\sum_{(L,R)\in S^\epsilon(T)}\sum_{(l,r)\in S^\epsilon(t)}h(|L|,|R|)1_{R=r}g_{|l|}(L,l)+(1-h(|L|,|R|))1_{L=l}g_{|r|}(R,r).$$

We now check the three properties that $(f_n)_{n\geq 1}$ needs to satisfy in order to be a growth scheme for simple triangulations.

First of all, in order for $f_n(T,t)$ to be nonzero, we must have $g_{|r|}(R,r)>0$ or $g_{|l|}(L,l)>0$ for some $R,r$ or $L,l$ such that $(R,L)\in S^\epsilon(T)$ and $(r,l)\in S^\epsilon(t)$ for some $\epsilon=\pm1$. This implies that $R=\grow(r,v)$ or $L=\grow(l,v)$ for some leaf $v$ by definition of a growth scheme for 4-ary trees. Since we have $T=\Xi(\tau,\epsilon)$, where $\tau$ is the blossoming tree corresponding to $(L,R)$, and $t=\Xi(\tau',\epsilon)$, where $\tau'$ is the blossoming tree corresponding to $(l,r)$, Lemma~\ref{lemma:growing trees to growing triangulations} implies the existence of an edge $e$ of $T$ such that $t=\coll(T,e)$, as wanted.

Now consider $\sum_{t\in\STr_n}f_n(T,t)$ for any fixed $T\in\STr_{n+1}$, that is,
$$\frac{1}{2n+2}\sum_{t\in\STr_n}\sum_{\epsilon=\pm1}\sum_{(L,R)\in S^\epsilon(T)}\sum_{(l,r)\in S^\epsilon(t)}h(|L|,|R|)1_{R=r}g_{|l|}(L,l)+(1-h(|L|,|R|))1_{L=l}g_{|r|}(R,r).$$

For ease of notation, denote by $F(L,R,l,r)$ the expression appearing within the repeated sums (which does not depend on $\epsilon$). By changing the order of summation, we obtain
$$\frac{1}{2n+2}\sum_{\epsilon=\pm1}\sum_{(L,R)\in S^\epsilon(T)}\sum_{t\in\STr_n}\sum_{(l,r)\in S^\epsilon(t)}F(L,R,l,r),$$
where the two internal sums yield one term for each pair $(l,r)$ in the set $\cup_{k=0}^{n-1}\CT^4_k\times\CT^4_{n-1-k}$, of which the collection of sets $\{S^\epsilon(t)\}_{t\in\STr_n}$ (where $\epsilon$ is fixed) forms a partition. The above sum is therefore equal to
$$\frac{1}{2n+2}\sum_{\epsilon=\pm1}\sum_{(L,R)\in S^\epsilon(T)}\sum_{k=0}^{n-1}\sum_{l\in\CT^4_k,r\in\CT^4_{n-1-k}}F(L,R,l,r).$$
Using the fact that the expression for $F(L,R,l,r)$ is zero unless $R=r$ or $L=l$, one can further rewrite the above as
$$\frac{1}{2n+2}\sum_{\epsilon=\pm1}\sum_{(L,R)\in S^\epsilon(T)}\left(\sum_{l\in\CT^4_{|L|-1}}F(L,R,l,R)+\sum_{r\in\CT^4_{|R|-1}}F(L,R,L,r)\right)=$$
$$=\frac{1}{2n+2}\sum_{\epsilon=\pm1}\sum_{(L,R)\in S^\epsilon(T)}h(|L|,|R|)\sum_{l\in\CT^4_{|L|-1}}g_{|l|}(L,l)+(1-h(|L|,|R|))\sum_{r\in\CT^4_{|R|-1}}g_{|r|}(R,r).$$
Finally, using the fact that $(g_n)_{n\geq 0}$ is a growth scheme for 4-ary trees yields
$$\frac{1}{2n+2}\sum_{\epsilon=\pm1}\sum_{(L,R)\in S^\epsilon(T)}h(|L|,|R|)+(1-h(|L|,|R|))=\frac{1}{2n+2}\sum_{\epsilon=\pm1}\sum_{(L,R)\in S^\epsilon(T)}1=\frac{2n+2}{2n+2}=1,$$
where we have used the fact that $|S^1(T)\cup S^{-1}(T)|=2(n+1)$.

Note that, by combining it with the obvious fact that $f_n(T,t)>0$, the computation above also proves that the image of $f_n$ is contained in $[0,1]$.

Finally, consider the sum $\sum_{T\in\STr_{n+1}}f_n(T,t)$ for a fixed $t\in\STr_{n}$. This time, we are computing
$$\frac{1}{2n+2}\sum_{\epsilon=\pm1}\sum_{(l,r)\in S^\epsilon(t)}\sum_{T\in\STr_{n+1}}\sum_{(L,R)\in S^\epsilon(T)}F(L,R,l,r)=\frac{1}{2n+2}\sum_{\epsilon=\pm1}\sum_{(l,r)\in S^\epsilon(t)}\sum_{k=0}^{n}\sum_{L\in\CT^4_k,R\in\CT^4_{n-k}}F(L,R,l,r)=$$
$$=\frac{1}{2n+2}\sum_{\epsilon=\pm1}\sum_{(l,r)\in S^\epsilon(t)}\sum_{L\in\CT^4_{|l|+1}}h(|l|+1,|r|)g_{|l|}(L,l)+\sum_{R\in\CT^4_{|r|+1}}(1-h(|l|,|r|+1))g_{|r|}(R,r),$$
where we have performed the same manipulations as for the previous sum. We can now use the fact that $(g_n)_{n\geq 0}$ is a growth scheme for 4-ary trees to compute the value of $\sum_{L\in\CT^4_{|l|+1}}g_{|l|}(L,l)$ and $\sum_{R\in\CT^4_{|r|+1}}g_{|r|}(R,r)$, which gives 
$$\frac{1}{2n+2}\sum_{\epsilon=\pm1}\sum_{(l,r)\in S^\epsilon(t)}\left(h(|l|+1,|r|)\frac{|\CT^4_{|l|+1}|}{|\CT^4_{|l|}|}+(1-h(|l|,|r|+1))\frac{|\CT^4_{|r|+1}|}{|\CT^4_{|r|}|}\right)=$$
$$=\frac{1}{2n+2}\sum_{\epsilon=\pm1}\sum_{(l,r)\in S^\epsilon(t)}\frac{(|l|+|r|+2)|\STr_{|l|+|r|+2}|}{(|l|+|r|+1)|\STr_{|l|+|r|+1}|}=\frac{2n}{2n+2}\cdot\frac{(n+1)|\STr_{n+1}|}{n|\STr_n|}=\frac{|\STr_{n+1}|}{|\STr_n|},$$
as wanted, thanks to the properties of $h$ expressed in Lemma~\ref{lemma:h function}.
\end{proof}

\bibliographystyle{siam}

\end{document}